\documentclass{siamltex}
\usepackage{amsfonts,amsmath,mathrsfs}
\usepackage{amssymb}
\usepackage{hyperref}
\usepackage{graphicx}
\usepackage{wrapfig}
\usepackage{subfig}
\usepackage{epsfig}
\usepackage{float}
\usepackage{bbm}

\newcommand{\N}{\mathbb{N}}
\newcommand{\R}{\mathbb{R}}
\newcommand{\C}{\mathbb{C}}

\newtheorem{remark}[theorem]{Remark}
\newtheorem{algorithm}{Algorithm}

\newcommand{\T}{{\rm T}}

\title{Computational framework for applying electrical impedance tomography to head imaging}

\author{V. Candiani\footnotemark[2]
\and A. Hannukainen\footnotemark[2]
\and N. Hyv\"onen\footnotemark[2]
}

\begin{document}
\maketitle

\renewcommand{\thefootnote}{\fnsymbol{footnote}}

\footnotetext[2]{Aalto University, Department of Mathematics and Systems Analysis, P.O. Box 11100, FI-00076 Aalto, Finland (valentina.candiani@aalto.fi, antti.hannukainen@aalto.fi, nuutti.hyvonen@aalto.fi). This work was supported by the Academy of Finland (decisions 312124 and 312340).}

\begin{abstract}
This work introduces a computational framework for applying absolute electrical impedance tomography to head imaging without accurate information on the head shape or the electrode positions. A library of fifty heads is employed to build a principal component model for the typical variations in the shape of the human head, which leads to a relatively accurate parametrization for head shapes with only a few free parameters. The estimation of these shape parameters and the electrode positions is incorporated in a regularized Newton-type output least squares reconstruction algorithm. The presented numerical experiments demonstrate that strong enough variations in the internal conductivity of a human head can be detected by absolute electrical impedance tomography even if the geometric information on the measurement configuration is incomplete to an extent that is to be expected in practice. 
\end{abstract}

\renewcommand{\thefootnote}{\arabic{footnote}}



\pagestyle{myheadings}
\thispagestyle{plain}
\markboth{V. CANDIANI, A. HANNUKAINEN, AND N. HYV\"ONEN}{EIT FOR HEAD IMAGING}

\section{Introduction}
\label{sec:introduction}

{\em Electrical impedance tomography} (EIT) is a noninvasive imaging method that is based on current and voltage measurements on the boundary of the examined physical body. To be more precise, a set of contact electrodes is employed to drive current patterns into the object and the resulting electric potential is measured at (some of) the electrodes. Such measurements obviously depend on the interior conductivity of the object, and the aim of any algorithm designed for EIT is to (partially) invert this dependence and reconstruct (useful information about) the conductivity. The inverse problem of EIT is nonlinear and highly illposed; we refer to \cite{Borcea02,Cheney99,Uhlmann09} for general information about EIT and its applications. The aim of this work is to develop a computational framework for applying EIT to head imaging without accurate information on the electrode positions or the shape of the studied patient's head.

In medical applications of EIT, the exact shape of the imaged object is usually unknown. Even if one has some generic information about the average shape of the imaged part of the human body, the natural variations between different subjects are often significant: consider,~e.g.,~the head shapes and sizes over the whole human population. Unfortunately, EIT is known to be extremely sensitive to geometric mismodeling of the imaged object \cite{Barber88,Breckon88,Kolehmainen97}. The most straightforward remedy to inaccurate geometric modeling is difference imaging~\cite{Barber84}: if electrode measurements are performed at two separate times and the aim is to reconstruct the corresponding change in the conductivity, the modeling errors partially cancel out and it is possible to get reasonable reconstructions. However, waiting for a change in the conductivity distribution is not a plausible option in many applications such as detection of stroke. 

There exist a few previous algorithms for handling geometric uncertainties in absolute EIT imaging. The first one was introduced in \cite{Kolehmainen05,Kolehmainen07} for two-dimensional EIT and, intuitively speaking, it is based on compensating for the mismodeled geometry by reconstructing a mildly anisotropic conductivity. In \cite{Nissinen11,Nissinen11b}, the so-called approximation error approach~\cite{Kaipio05} was applied to EIT with geometric uncertainties. To put it short, the error caused by the mismodeling was included as an extra additive noise process in the measurement model, its statistics were estimated in advance based on heavy simulations and prior knowledge on all unknown parameters, and finally the actual inversion was performed within the Bayesian paradigm. Recently, \cite{Hyvonen17} built a polynomial surrogate for the dependence of the boundary measurements of EIT on all unknowns, including the parametrized measurement geometry, and employed this surrogate in straightforward Tikhonov regularization in two spatial dimensions. Of the aforementioned methods, the ones introduced in \cite{Kolehmainen05,Kolehmainen07} and \cite{Hyvonen17} have only been implemented in two spatial dimensions, and the approximation error approach employed in \cite{Nissinen11,Nissinen11b} requires a vast and expressive teaching sample describing the possible types of geometric mismodeling. 

In this work we extend the conceptually straightforward way of handling imprecisely known body shape and electrode positions introduced in \cite{Darde12,Darde13a,Darde13b}. The basic idea is to utilize the derivatives of the electrode potentials with respect to the exterior boundary shape and the electrode locations in a regularized Newton-type output least squares algorithm that simultaneously reconstructs all unknowns in the measurement setup. The main weakness of the algorithm in \cite{Darde13a,Darde13b} is arguably the need to employ relatively dense {\em finite element} (FE) meshes in order to overcome an instability in the computation of the shape derivatives. We facilitate this stability issue by adopting the so-called smoothened {\em complete electrode model} (CEM) from \cite{Hyvonen17b} as our forward model; consult \cite{Cheng89,Somersalo92} for information about the original CEM. Moreover, unlike in \cite{Darde13a,Darde13b} where the imaged three-dimensional objects were homogeneous in one spatial dimension, our parametrization of the imaged body is truly three-dimensional --- as is the human head. See also \cite{Jehl15} where the derivatives introduced in \cite{Darde12} were employed in simultaneous reconstruction of conductivity and electrode positions in a realistic but accurately modeled head geometry.

To be able to implement the algorithm of~\cite{Darde13a,Darde13b} for head imaging by EIT, one needs (i) an appropriate parametrization for the natural variations in head shapes and sizes over the human population and (ii) a robust and efficient way of meshing a given (parametrized) human head model to enable forward solution by a {\em finite element method} (FEM). To build the former, we employ a library of fifty human heads from \cite{Lee16} and form a {\em principal component} model for the associated geometric variations. It turns out that a few first principal components already give a reasonably accurate geometric representation. In particular, this means that we only need to estimate a few geometric parameters as a part of the reconstruction algorithm. On the other hand, the latter requirement is satisfied by introducing a customized workflow for mesh generation, compatible with attaching electrode patches anywhere on the surface of a parametrized head.

Our numerical experiments demonstrate that strong enough alterations in the internal conductivity of a human head can be reconstructed by EIT even if the electrodes are misplaced and the geometric model for the head is inaccurate to the extent that is to be expected taking into account the natural variations. In fact, it seems that the misplacement of electrodes is a worse hindrance for the application of absolute EIT to head imaging than a moderately mismodeled head shape (cf.~\cite{Hyvonen17c}): in many cases, altogether ignoring the inaccuracies in the head model and only including the estimation of the electrode locations and contact conductances in the algorithm already leads to a tolerable reconstruction of the conductivity. 

This text is organized as follows. Section~\ref{sec:CEM} reviews the smoothened CEM. The parametrized head model and the associated algorithm for forming head meshes are introduced in Section~\ref{sec:head}. Section~\ref{sec:deriv} considers the computation of the (shape) derivatives needed in our Newton-type reconstruction algorithm. The numerical experiments are documented in Section~\ref{sec:numer} and the conclusions drawn in Section~\ref{sec:conclusion}.

\section{Forward model}
\label{sec:CEM}
We employ as our forward model the smoothened CEM~\cite{Hyvonen17b} that has favorable regularity properties compared to the traditional CEM~\cite{Somersalo92}. Preliminary tests with water tank data have also indicated that the smoothened version models real-world measurements approximately as accurately as the traditional CEM~\cite{Hyvonen17b}

Consider a bounded Lipschitz domain $\Omega \subset \R^3$ and assume there are  $M \in \N \setminus \{ 1 \}$ electrodes $E_1, \dots, E_M$ attached to its boundary. The electrode patches are identified with the nonempty, connected and open surface patches that they cover. We assume the electrodes are well separated, i.e.,~$\overline{E}_m\cap \overline{E}_l = \emptyset$ if $m\not= l$, and denote $E = \cup E_m$.  The admittivity $\sigma: \Omega \to \C$ describing the electric properties of $\Omega$ is assumed to be bounded and strictly positive in the real part, that is,
\begin{equation}
\label{eq:sigma}
\sigma \in L_+^\infty(\Omega) := \big\{ \varsigma \in L^\infty(\Omega) \ | \ {\rm Re} (\varsigma ) > c \ \,  {\rm a.e.} \ {\rm in} \ \Omega \ {\rm for} \ {\rm some} \ c>0 \big\}.
\end{equation}
In other words, it is assumed that there are no ideally resistive or conductive regions of nonzero measure inside $\Omega$. As a specialty of the smoothened CEM, the {\em contact admittance} between the electrodes and the imaged object $\Omega$ is interpreted as a single function $\zeta \in L^\infty(\partial \Omega)$ such that
\begin{equation}
\label{eq:zeta}
{\rm Re} (\zeta) \geq 0,
\qquad \zeta_{\partial \Omega \setminus{\overline{E}}} \equiv 0, \qquad
{\rm Re}\big(\zeta|_{E_m}\big) \not\equiv 0 
\end{equation}
for all $m=1, \dots, M$ in the topology of $L^\infty(\partial \Omega)$. These conditions can be interpreted as follows: the contact conductance cannot be negative, the contact admittance vanishes away from the electrodes, and on each electrode there must be a region of nonzero measure where the contact conductance is positive,~i.e.,~there must be a possibility for feeding current through any of the electrodes. 

When a single measurement by an EIT device is performed, net currents $I_m\in\C $, $m=1, \dots, M$, are driven through the corresponding electrodes and (noisy versions of) the resulting constant electrode potentials $U_m \in \C$, $m=1, \dots, M$, are measured.
As it is natural to assume there are no sinks or sources inside $\Omega$, the current pattern $I = [I_1,\dots,I_M]^{\rm T}$ belongs to the mean-free subspace
\[ 
\C^M_\diamond \, := \, \Big\{J \in\C^M\,\Big|\, \sum_{m=1}^M J_m = 0\Big\}
\]
due to conservation of electric charge. The vector containing the electrode potentials, $U = [U_1,\dots,U_M]^{\rm T} \in \C^M$, is isomorphically identified with 
\begin{equation}
\label{eq:piecewise}
U \, = \, \sum_{m=1}^M U_m \chi_m \in L^2(\partial \Omega),
\end{equation}
where $\chi_m$ is the characteristic function of the respective electrode $E_m$ on $\partial \Omega$. In the following, an electrode potential pattern $U$ may be treated either as a piecewise constant function supported on $\overline{E}$ or as an element of $\C^M$; the correct interpretation should be clear from the context.

The forward problem corresponding to the smoothened CEM is as follows \cite{Hyvonen17b}: the electromagnetic potential $u$ inside $\Omega$ and the piecewise constant electrode potential $U$ weakly satisfy
\begin{equation}
\label{eq:cemeqs}
\begin{array}{ll}
\displaystyle{\nabla \cdot(\sigma\nabla u) = 0 \qquad}  &{\rm in}\;\; \Omega, \\[6pt] 
{\displaystyle {\nu\cdot\sigma\nabla u} = \zeta (U - u) } \qquad &{\rm on}\;\; \partial \Omega, \\[2pt] 
{\displaystyle \int_{E_m}\nu\cdot\sigma\nabla u\,{\rm d}S} = I_m, \qquad & m=1,\ldots,M, \\[4pt]
\end{array}
\end{equation}
where $\nu \in L^\infty(\partial \Omega, \R^3)$ denotes the exterior unit normal of $\partial \Omega$. If one chooses $\zeta$ to be constant on each electrode, defines $z_m := (1/\zeta)|_{E_m}$, $m=1, \dots, M$, and explicitly accounts for the second condition of \eqref{eq:zeta}, the forward problem \eqref{eq:cemeqs} reduces to that associated to the standard CEM (cf.,~e.g.,~\cite{Somersalo92}). Hence, the CEM can actually be considered a special case of its smoothened counterpart.

We look for the solution $(u,U)$ of \eqref{eq:cemeqs} in $\mathcal{H}^1$ defined via
\begin{equation}
  \label{eq:dirsum}
\mathcal{H}^s := H^s(\Omega) \oplus \C_{\diamond}^M, \qquad s \in \R.
\end{equation}
Note that the use of $\C_{\diamond}^M$ in \eqref{eq:dirsum} corresponds to systematically choosing the ground level of potential so that the mean of the electrode potentials is zero.
The following proposition summarizes the unique solvability of \eqref{eq:cemeqs} and the regularity properties of the associate forward solution.

\begin{proposition}
\label{prop:ex_smooth}
Under the assumptions \eqref{eq:sigma} and \eqref{eq:zeta}, the problem \eqref{eq:cemeqs} has a unique solution $(u,U) \in \mathcal{H}^1$ that depends boundedly on $I \in \C_\diamond^M$, that is,
$$
\| (u,U) \|_{\mathcal{H}^1} \leq C \| I \|_2,
$$
where $C= C(\Omega,E,\sigma,\zeta)>0$ is independent of $I$. If in addition $\sigma \in \mathcal{C}^\infty(\overline{\Omega}, \C^{n \times n})$, $\zeta \in \mathcal{C}^\infty(\partial \Omega)$ and $\partial \Omega\in \mathcal{C}^\infty$, then the solution to \eqref{eq:cemeqs} satisfies
\begin{equation}
\label{scont}
\| (u,U) \|_{\mathcal{H}^{s}} \leq C  \| I \|_2 \qquad \text{for} \ \text{any} \ s \in \R,
\end{equation}
where $C = C(\Omega, E, \sigma, \zeta, s) > 0$ is still independent of $I \in \C_\diamond^M$.
\end{proposition}

\begin{proof}
The assertions follow by slightly modifying the proofs of \cite[Theorem~2.2 \& Theorem~2.5]{Hyvonen17b} that are based on the use of quotient spaces instead of forcing the electrode potential patterns to be mean-free.
\end{proof}

The latter part of Proposition~\ref{prop:ex_smooth} could actually be made more explicit: the solution to \eqref{eq:cemeqs} depends boundedly on the current pattern in the topology of $\mathcal{H}^s$ for a given $s\geq1$ assuming that $\sigma$, $\zeta$ and $\partial \Omega$ are in certain regularity classes dictated by $s$; see~\cite[Theorem~2.5]{Hyvonen17b} for more information. This is the main reason why we resort to the smoothened CEM. By modeling the contact conductance as a smooth function, one achieves higher Sobolev regularity for the forward solution and thereby faster convergence of the numerical shape derivatives needed in our Newton type reconstruction algorithm; see~\cite{Hyvonen17b} for more information. For comparison, the highest attainable regularity for the forward solution of the traditional CEM is $(u,U) \in \mathcal{H}^{2-\epsilon}$, $\epsilon > 0$, even if $\partial \Omega$ and $\sigma$ are infinitely smooth.

\section{Parametrized head model}
\label{sec:head}
To build a realistic model for the variations in the shape and size of the human head, we exploit the library of $n=50$ heads from \cite{Lee16}; see the top row of Figure~\ref{eq:model_heads} for two examples. Each of these heads allows a star-shaped parametrization if the origin is chosen to be (approximately) the center of mass of its bottom face. To be more precise, we represent the crown of the $j$th head as the graph of a function
\begin{equation}
  \label{eq:jth_head}
S_j:
\left\{
\begin{array}{l}
\mathbb{S}_+ \to \R^3, \\[1mm]
\hat{x} \mapsto r_j(\hat{x}) \, \hat{x},
\end{array}
\right.
\end{equation}
where $\mathbb{S}_+$ is the upper unit hemisphere, i.e.,
$$
\mathbb{S}_+ = \big\{ x \in \R^3 \; | \; \| x \|_2 = 1 \ {\rm and} \ x_3 > 0 \big\},
$$
and $r_j:  \mathbb{S}_+ \to \R_+$ gives the distance from the origin to the surface of the $j$th head as a function of the direction $\hat{x} \in \mathbb{S}_+$. Because $r_j$ completely determines $S_j$, we somewhat misleadingly call $r_1, \dots, r_n$ the functional representations for our library of heads. The bottom row of Figure~\ref{eq:model_heads} shows such representations for the two example heads on the top row as functions of the polar $\theta \in (0, \pi/2)$ and azimuthal $\phi \in [0, 2\pi)$ angles.

  \begin{figure}{t}
    \center{
      {\includegraphics[width=6cm]{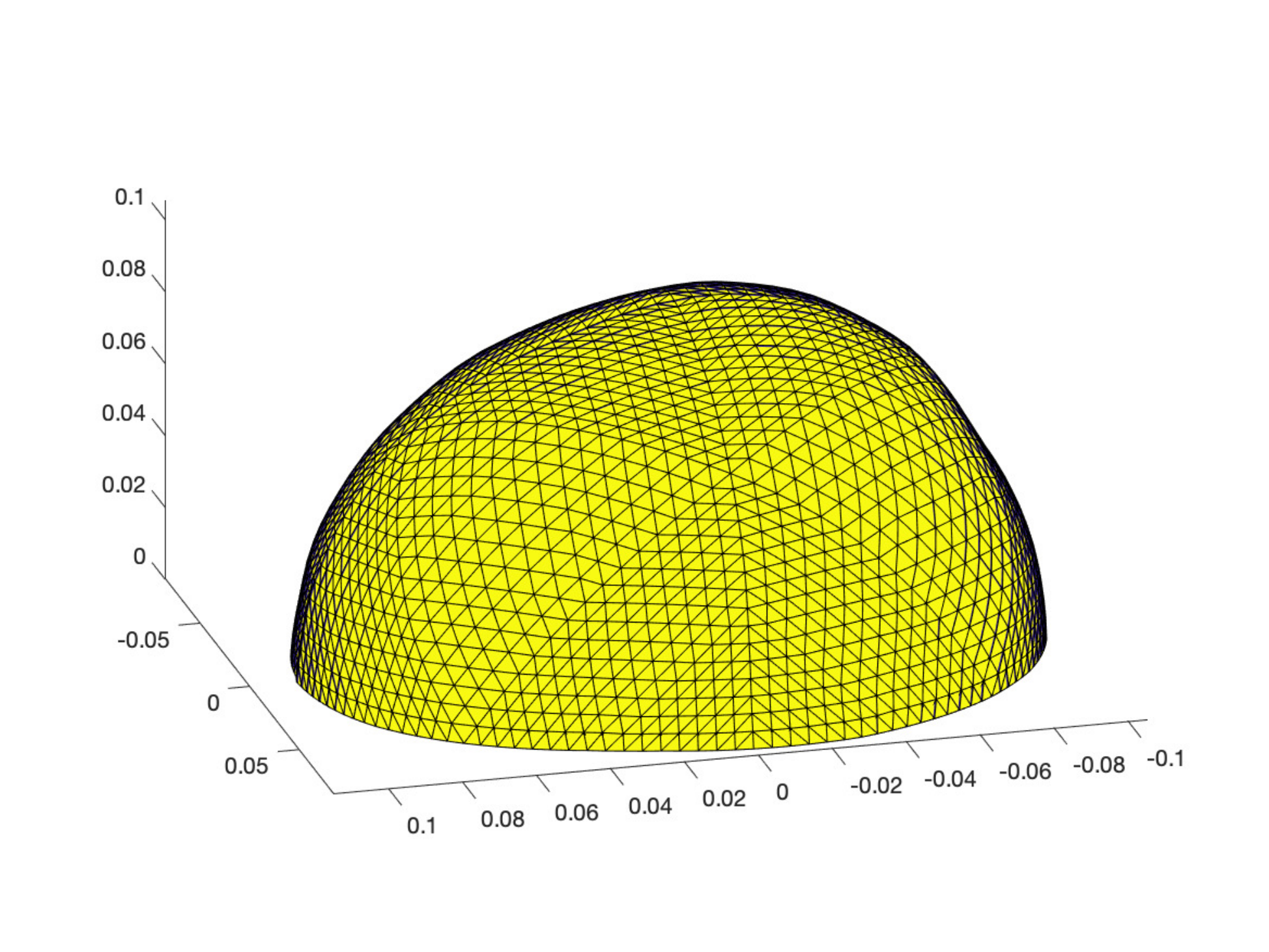}}
{\includegraphics[width=6cm]{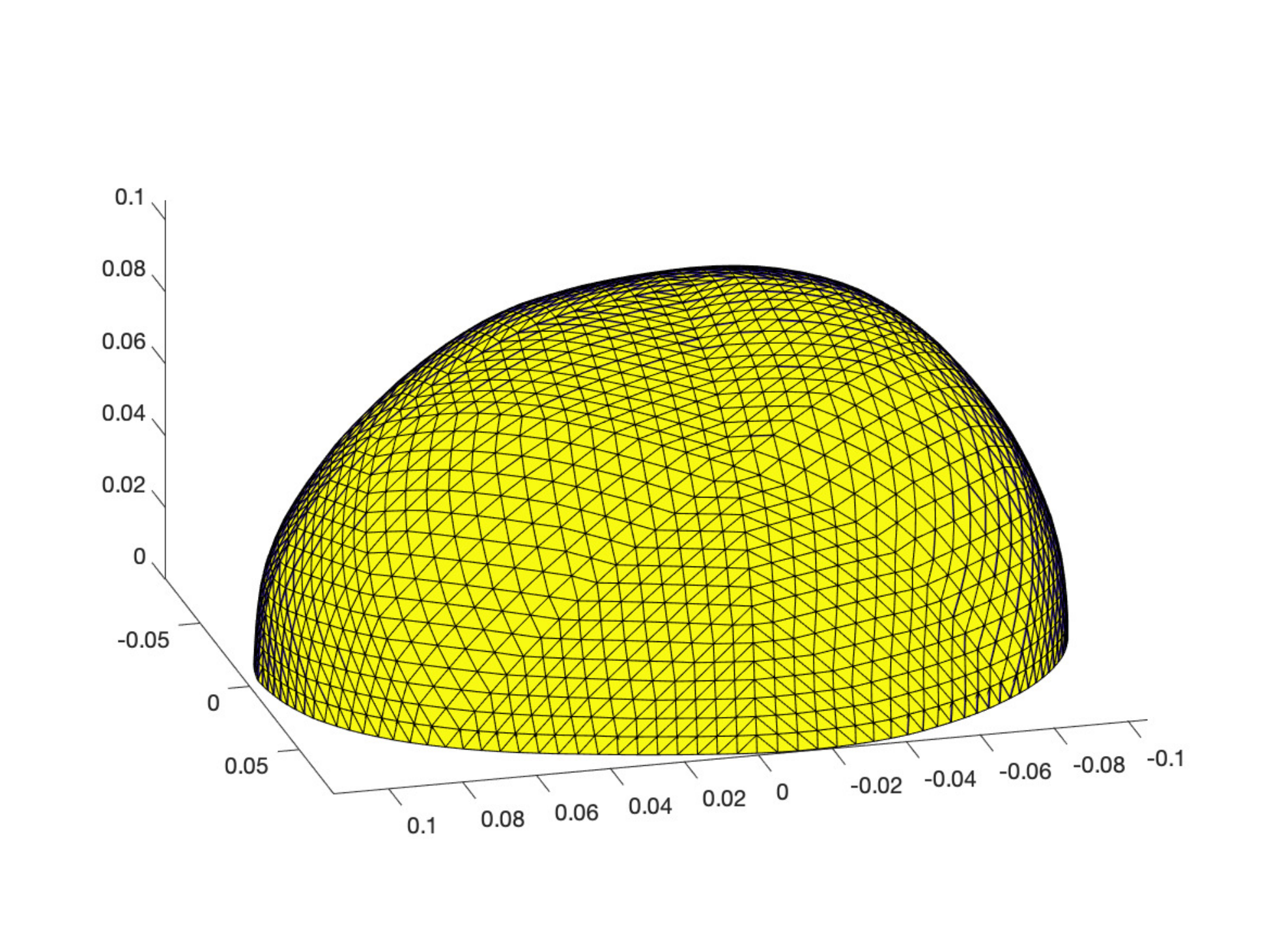}}
    }

\center{
  {\includegraphics[width=5.5cm]{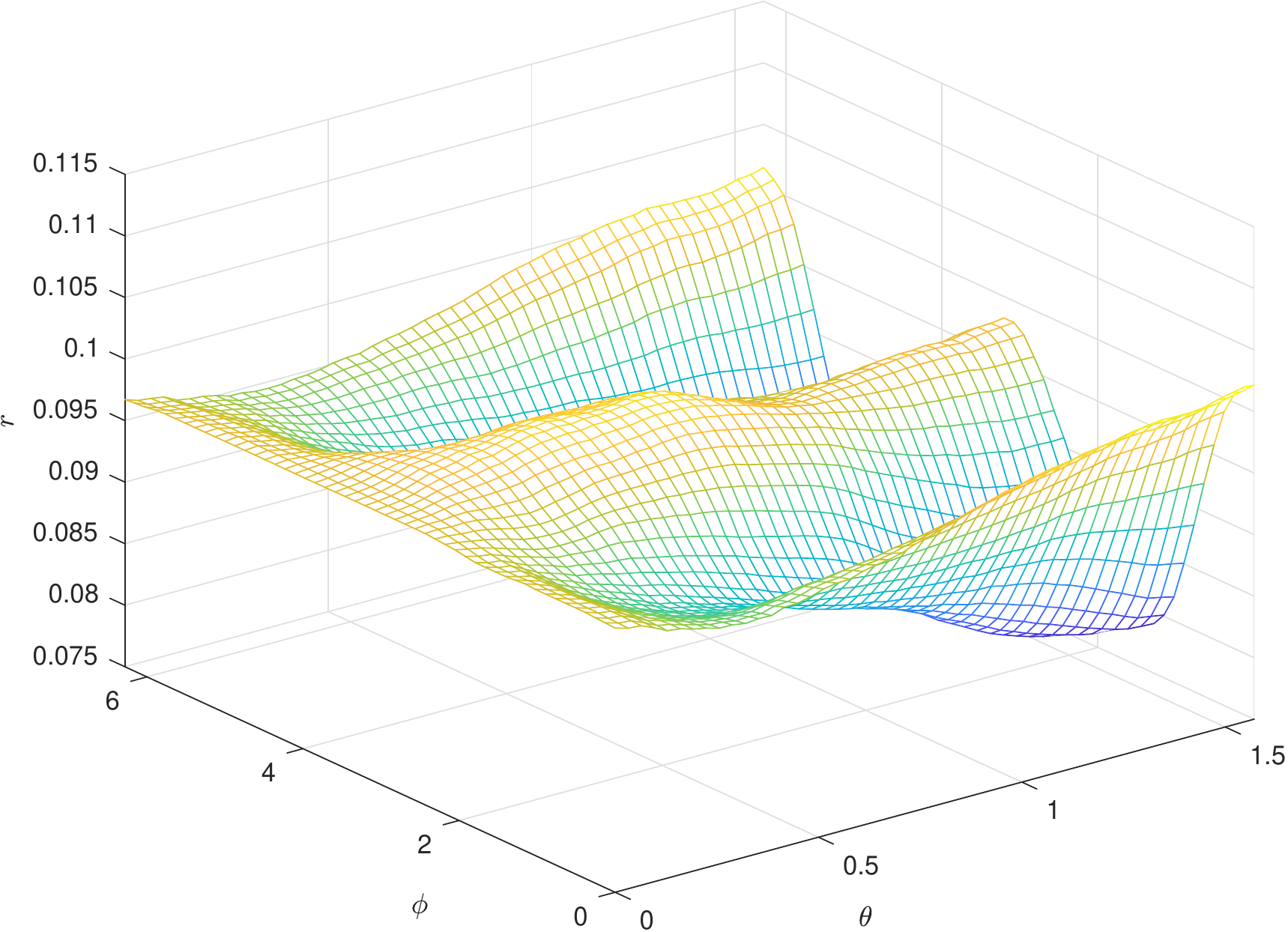}}
  \quad \quad
      {\includegraphics[width=5.5cm]{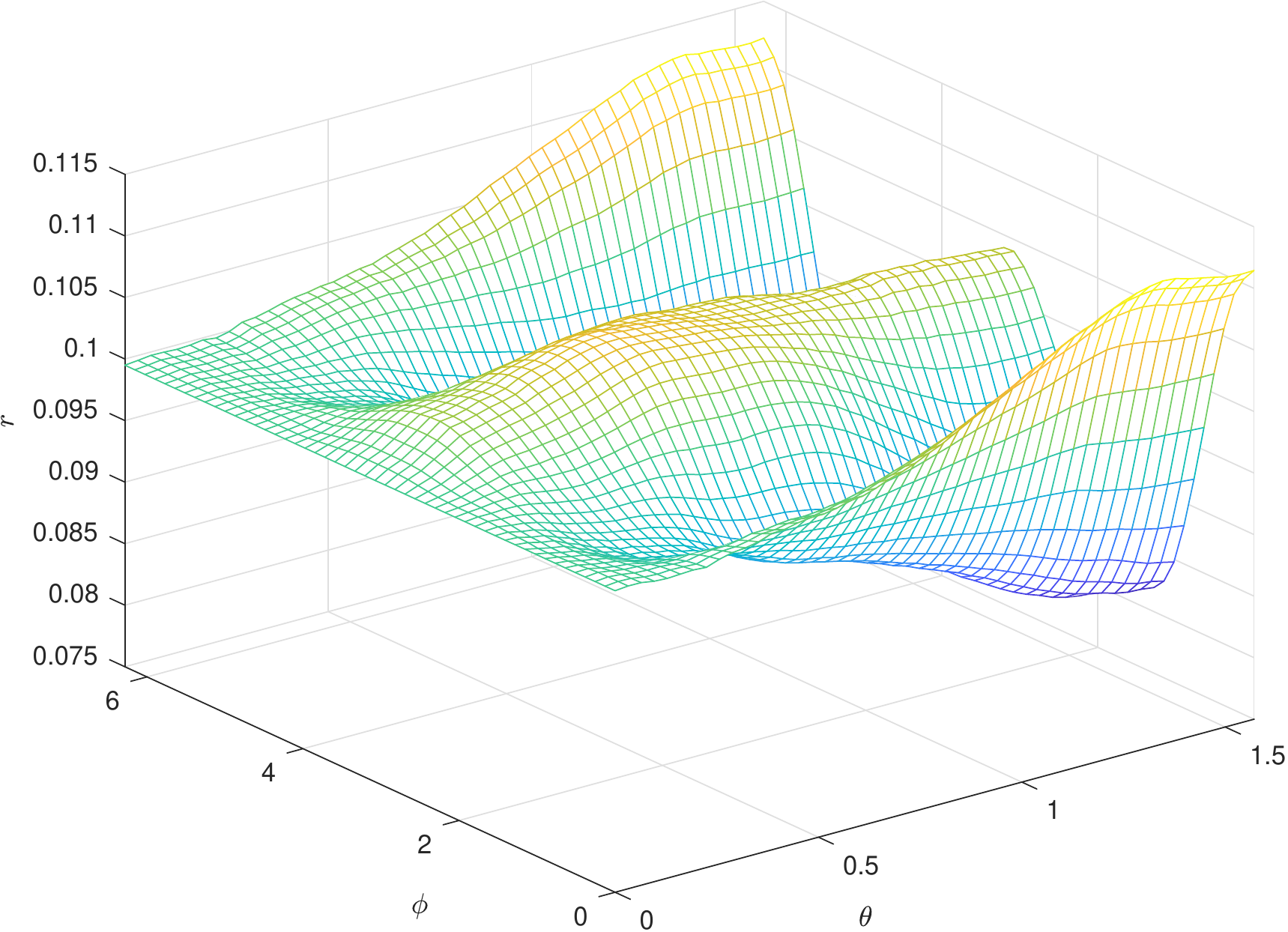}}
    }    
  \caption{Top row: two heads from the library of \cite{Lee16}. Bottom row: corresponding representations $r_j: \mathbb{S}_+ \to \R_+$ parametrized as functions of the polar and azimuthal angles. The unit of length is meter.}
\label{eq:model_heads}
\end{figure}

\subsection{Optimal parametrization}
\label{sec:optimpara}
Our aim is to form an accurate parametrization for the variations in the family of head representations $r_1, \dots,  r_n$ using as few degrees of freedom as possible. To this end, we set
\begin{equation}
  \label{eq:mean_etc}
\bar{r} = \frac{1}{n} \sum_{j=1}^n r_j \qquad {\rm and} \qquad \rho_j = r_j - \bar{r}, \quad j=1, \dots, n,
\end{equation}
i.e.,~$\bar{r}$ describes the average head and $\rho_j, \dots, \rho_n$ are the corresponding perturbations that define the employed library of heads. We then look for an $\tilde{n}$-dimensional subspace $V_{\tilde{n}} \subset  H^1(\mathbb{S}_+)$, $1 \leq \tilde{n} \leq n$,
that satisfies
\begin{equation}
  \label{eq:minim}
\sum_{j=1}^{n} \min_{\eta \in V_{\tilde{n}}} \| \rho_j - \eta \|_{H^1(\mathbb{S}_+)}^2 \leq \sum_{j=1}^{n} \min_{\eta \in W} \| \rho_j - \eta \|_{H^1(\mathbb{S}_+)}^2,
\end{equation}
for all $\tilde{n}$-dimensional subspaces $W$ of the Sobolev space $H^1(\mathbb{S}_+)$. In other words, the aim is to find an $\tilde{n}$-dimensional subspace that on average contains the best approximations for the perturbations $\rho_j, \dots, \rho_n$ if the quality of fit is measured by the squared norm of $H^1(\mathbb{S}_+)$. Here, it is implicitly assumed that $\rho_j, \dots, \rho_n \in H^1(\mathbb{S}_+)$.

For simplicity, suppose the functions $\rho_1, \dots, \rho_n$ are linearly independent. In the following lemma, we utilize the matrix $R = \R^{n \times n}$ defined componentwise as
$$
R_{ij} = ( \rho_i, \rho_j)_{H^1(\mathbb{S}_+)}, \qquad i,j = 1, \dots, n,
$$
where $( \cdot, \cdot)_{H^1(\mathbb{S}_+)}: H^1(\mathbb{S}_+) \times H^1(\mathbb{S}_+) \to \R$ is the real $H^1$ inner product on $\mathbb{S}_+$. It is straightforward to check that $R$ is symmetric and positive definite. The eigenvalues of $R$ are repeated according to their multiplicity and denoted by $\lambda_1 \geq \lambda_2 \geq \cdots \geq \lambda_n > 0$. A corresponding set of orthonormal eigenvectors is $v_1, \dots, v_n \in \R^n$. We also adopt the notation $\rho = [\rho_1, \dots, \rho_n]^{\rm T}: \mathbb{S}_+ \to \R^n$.

\begin{lemma}
  \label{lemma:prcomp}
  Assume that $\rho_1, \dots, \rho_n \in H^1(\mathbb{S}_+)$ are linearly independent and $1 \leq \tilde{n} \leq n$. An $\tilde{n}$-dimensional subspace $V_{\tilde{n}} \subset L^2(\mathbb{S}_+)$  that satisfies \eqref{eq:minim} is given by
$$
V_{\tilde{n}} = {\rm span} \{\hat{\rho}_1, \dots, \hat{\rho}_{\tilde{n}} \},
$$
where $\hat{\rho}_k := w_k^{\rm T}\rho$, $k=1, \dots, \tilde{n}$, are mutually orthonormal functions in $H^1(\mathbb{S}_+)$ with 
\begin{equation}
\label{eq:sceig}
w_k = \frac{1}{\sqrt{\lambda_k}} \, v_k, \qquad k=1, \dots, \tilde{n},
\end{equation}
being suitably scaled versions of the first $\tilde{n}$ eigenvectors for $R \in \R^{n \times n}$.
\end{lemma}

\begin{proof}
  First of all, any $\tilde{n}$-dimensional subspace $V_{\tilde{n}} \subset H^1(\mathbb{S}_+)$ can of course be given as a linear span of $\tilde{n}$ functions, say, $\hat{\rho}_1, \dots, \hat{\rho}_{\tilde{n}} \in H^1(\mathbb{S}_+)$ that are orthonormal in the inner product of $H^1(\mathbb{S}_+)$. Hence, by the Hilbert projection theorem,
  \begin{align}
    \label{eq:sqrsum}
  \sum_{j=1}^{n} \min_{\eta \in V_{\tilde{n}}} \| \rho_j - \eta \|_{H^1(\mathbb{S}_+)}^2 &=
  \sum_{j=1}^{n} \Big\| \rho_j - \sum_{k=1}^{\tilde{n}} (\rho_j, \hat{\rho}_k)_{H^1(\mathbb{S}_+)}  \hat{\rho}_k \Big\|_{H^1(\mathbb{S}_+)}^2 \nonumber \\
  &= \sum_{j=1}^n \| \rho_j \|_{H^1(\mathbb{S}_+)}^2 -  \sum_{j=1}^n \sum_{k=1}^{\tilde{n}} (\rho_j, \hat{\rho}_k)_{H^1(\mathbb{S}_+)}^2.
  \end{align}
Finding $V_{\tilde{n}}$ that satisfies $\eqref{eq:minim}$ is thus equivalent to introducing orthonormal functions $\hat{\rho}_1, \dots, \hat{\rho}_{\tilde{n}} \in H^1(\mathbb{S}_+)$ that maximize the sum of squared projections
  \begin{equation}
    \label{eq:sqproj}
  \sum_{j=1}^n \sum_{k=1}^{\tilde{n}} (\rho_j, \hat{\rho}_k)_{H^1(\mathbb{S}_+)}^2
  \end{equation}
appearing as the second term on the right-hand side of \eqref{eq:sqrsum}.
  
It follows easily from \eqref{eq:minim} that an optimal $V_{\tilde{n}}$ must, in fact, be  a subspace of ${\rm span}\{\rho_1, \dots, \rho_n\}$: otherwise one could reduce the dimension of $V_{\tilde{n}}$ by cutting away its component orthogonal to ${\rm span}\{\rho_1, \dots, \rho_n\}$ without affecting the value on the left-hand side of \eqref{eq:minim}. Hence, the maximizing orthonormal functions for \eqref{eq:sqproj}, if they exist, can be written as
$$
\hat{\rho}_k =  w_k^{\rm T} \rho, \qquad k=1, \dots, \tilde{n},
$$
for some $w_k \in \R^n$, $k=1, \dots, \tilde{n}$. The previously imposed condition that $\hat{\rho}_1, \dots, \hat{\rho}_{\tilde{n}}$ are orthonormal is equivalent to asking that
\begin{equation}
\label{eq:constraint}
(\hat{\rho}_l, \hat{\rho}_k)_{H^1(\mathbb{S}_+)} =  w_l^{\rm T} \! R  w_k  = \delta_{lk}, \qquad l,k=1, \dots, \tilde{n},
\end{equation}
i.e., the coefficient vectors $w_1, \dots, w_{\tilde{n}} \in \R^n$ are required to be orthonormal in the inner product induced by the symmetric and positive definite matrix $R \in \R^{n \times n}$.

To finalize the proof, we algebraically manipulate the expression \eqref{eq:sqproj} into a matrix form:
\begin{align}
  \label{eq:rayleigh}
\sum_{j=1}^n \sum_{k=1}^{\tilde{n}} (\rho_j, \hat{\rho}_k)_{H^1(\mathbb{S}_+)}^2
& = \sum_{k=1}^{\tilde{n}} \sum_{j=1}^n \Big( \sum_{i=1}^n (\rho_j, \rho_i)_{H^1(\mathbb{S}_+)} (w_k)_i\Big)^2 \nonumber \\
& = \sum_{k=1}^{\tilde{n}} \| R w_k \|_2^2 = \sum_{k=1}^{\tilde{n}} w_k^{\rm T} R^2 w_k
\end{align}
as $R$ is symmetric. It is straightforward to show that this expression is maximized under the constraint \eqref{eq:constraint} if the vectors $w_1, \dots, w_{\tilde{n}} \in \R^n$ are chosen according to \eqref{eq:sceig}: one can, e.g., set $z_k = R^{1/2} w_k$ in order to rewrite the maximization problem in a more standard form and subsequently deduce that $z_k = v_k$, $k=1,\dots, \tilde{n}$. This completes the proof.
\end{proof}

\begin{figure}
  \center{
    \includegraphics[width=7cm]{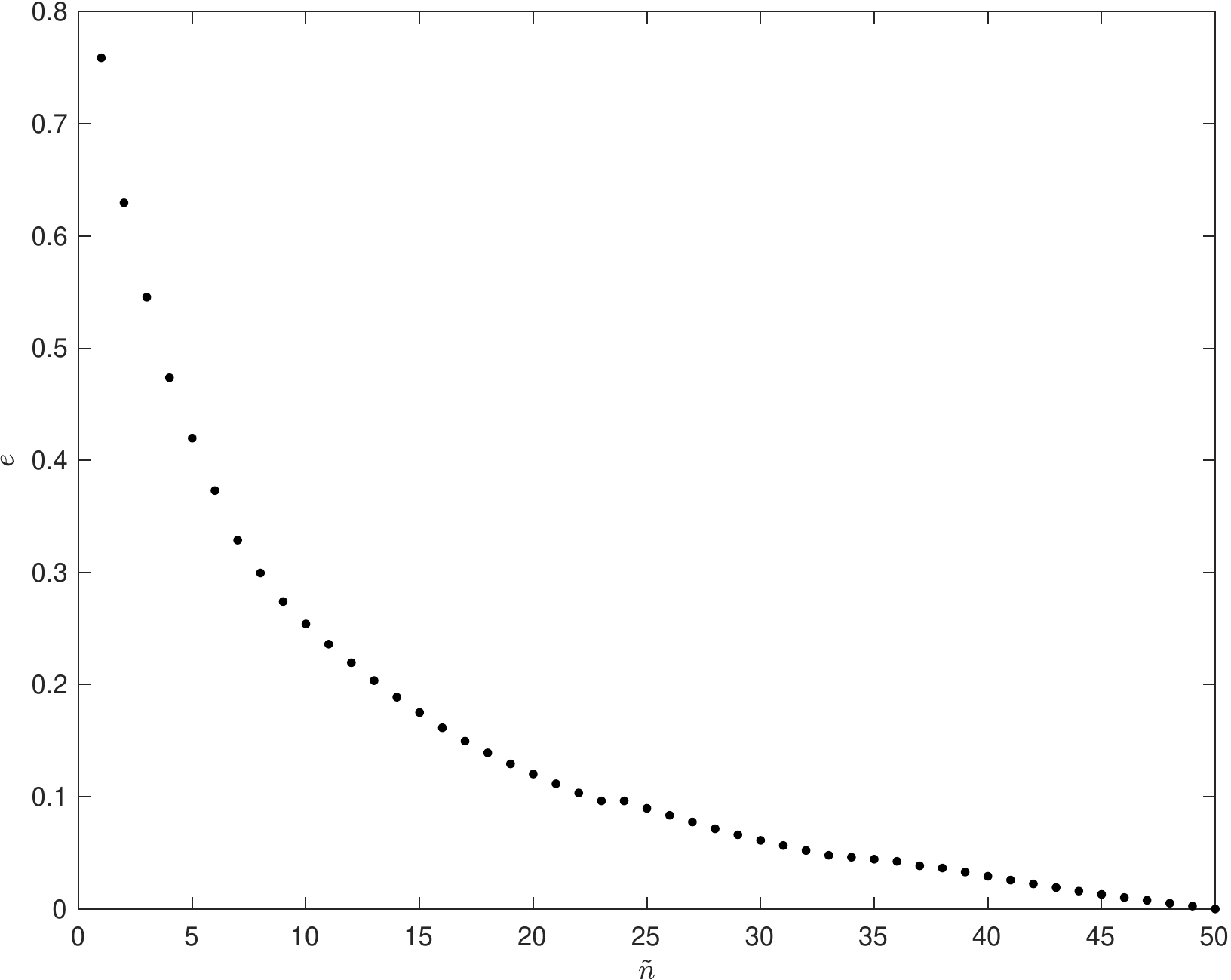}
    } 
  \caption{The relative representation error $e(\tilde{n})$ defined by \eqref{eq:repr_error} as a function of the dimension of the optimal subspace $V_{\tilde{n}}$.}
\label{fig:error}
\end{figure}

Notice that
$$
\sum_{j=1}^n \| \rho_j \|_{H^1(\mathbb{S}_+)}^2 = \, {\rm tr}(R) \, = \, \sum_{j=1}^n \lambda_j.
$$
In consequence, it follows from \eqref{eq:sqrsum} and \eqref{eq:rayleigh} that the left-hand side of \eqref{eq:minim} takes the value
$$
\sum_{j=1}^n \| \rho_j \|_{H^1(\mathbb{S}_+)}^2 - \sum_{j=1}^n \sum_{k=1}^{\tilde{n}} (\rho_j, \hat{\rho}_k)_{H^1(\mathbb{S}_+)}^2 = \sum_{j=1}^n \lambda_j - \sum_{k=1}^{\tilde{n}} w_k^{\rm T} R^2 w_k = \sum_{j=\tilde{n}+1}^n \lambda_j
$$
for an optimal $\tilde{n}$-dimensional subspace $V_{\tilde{n}}$. In particular, one has a simple formula for the relative squared error in the optimal representations for the considered library of heads:
\begin{equation}
\label{eq:repr_error}
e(\tilde{n}) :=
\frac{\sum_{j=1}^{n} \min_{\rho \in V_{\tilde{n}}} \| \rho_j - \rho \|_{H^1(\mathbb{S}_+)}^2}
     {\sum_{j=1}^{n} \| \rho_j \|_{H^1(\mathbb{S}_+)}^2}
      = \frac{\sum_{j=\tilde{n}+1}^n \lambda_j}{\sum_{j=1}^n \lambda_j}.
\end{equation}
Figure~\ref{fig:error} shows $e(\tilde{n})$ as a function of $\tilde{n}$. In particular, the optimal subspace $V_{\tilde{n}}$ is able to explain $58$\% of the (squared) geometric variations in the head library already for $\tilde{n}=5$.

Encouraged by Lemma~\ref{lemma:prcomp} and Figure~\ref{fig:error}, we introduce a parametrization for the crown of the human head as
\begin{equation}
\label{eq:theparam}
S(\hat{x}; \alpha) = \Big( \bar{r}(\hat{x}) +
\sum_{k=1}^{\tilde{n}} \alpha_k \hat{\rho}_k(\hat{x}) \Big) \hat{x}, \qquad \hat{x} \in \mathbb{S}_+,
\end{equation}
where $\bar{r}$ is the mean representation from \eqref{eq:mean_etc}, $\hat{\rho}_1, \dots, \hat{\rho}_{\tilde{n}}$ are the optimal orthonormal basis functions introduced in Lemma~\ref{lemma:prcomp}, and $\alpha \in \R^{\tilde{n}}$ contains the free shape parameters. In the numerical experiments of Section~\ref{sec:numer}, the number of free parameters appearing in \eqref{eq:theparam} is chosen to be $\tilde{n}=5$; the corresponding basis functions $\hat{\rho}_k$, $k=1, \dots, 5$, are visualized together with the average head in Figure~\ref{fig:pricomp}.

\begin{figure}
  \center{
  {\includegraphics[width=4cm]{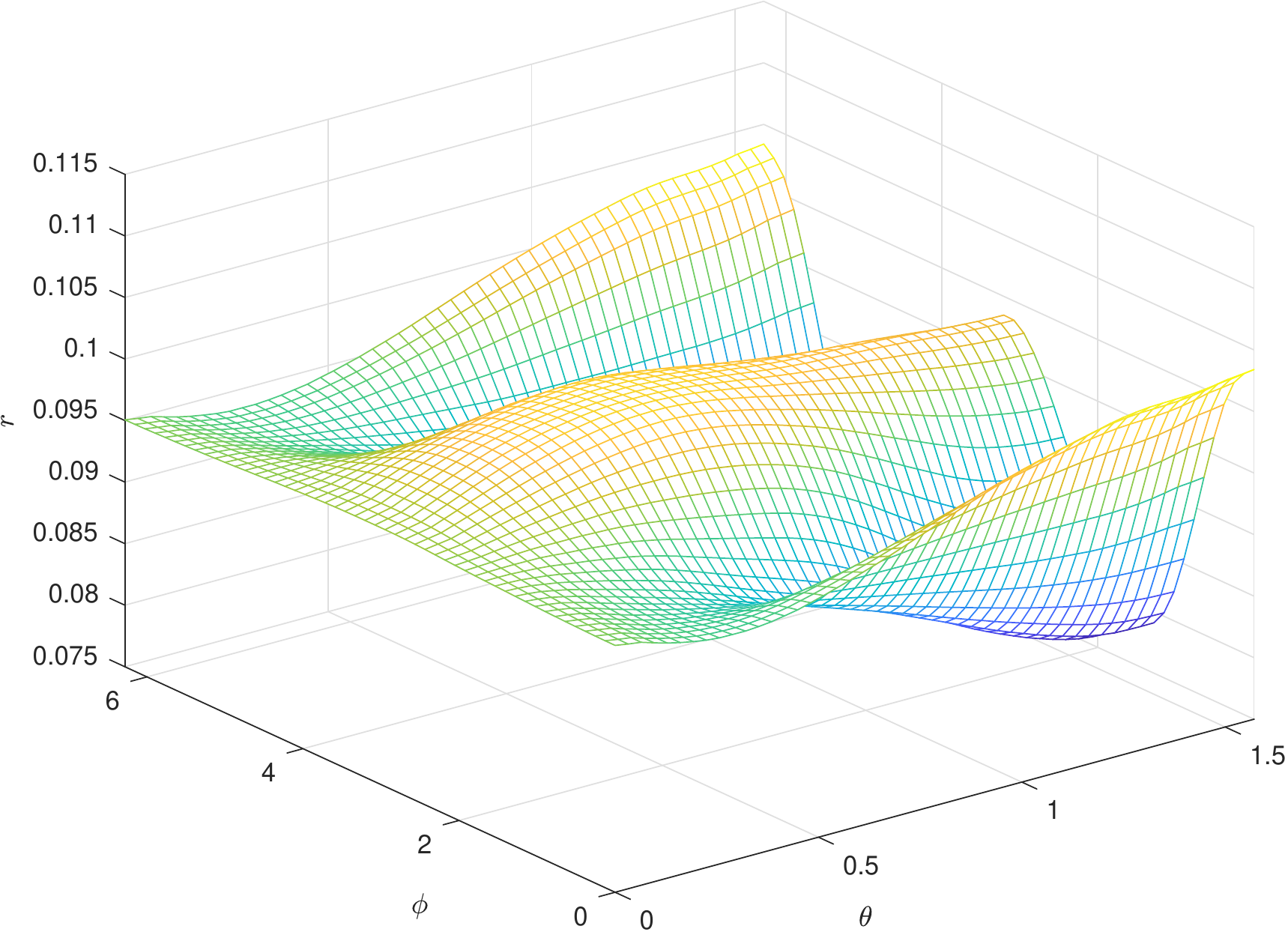}}
  \quad
      {\includegraphics[width=4cm]{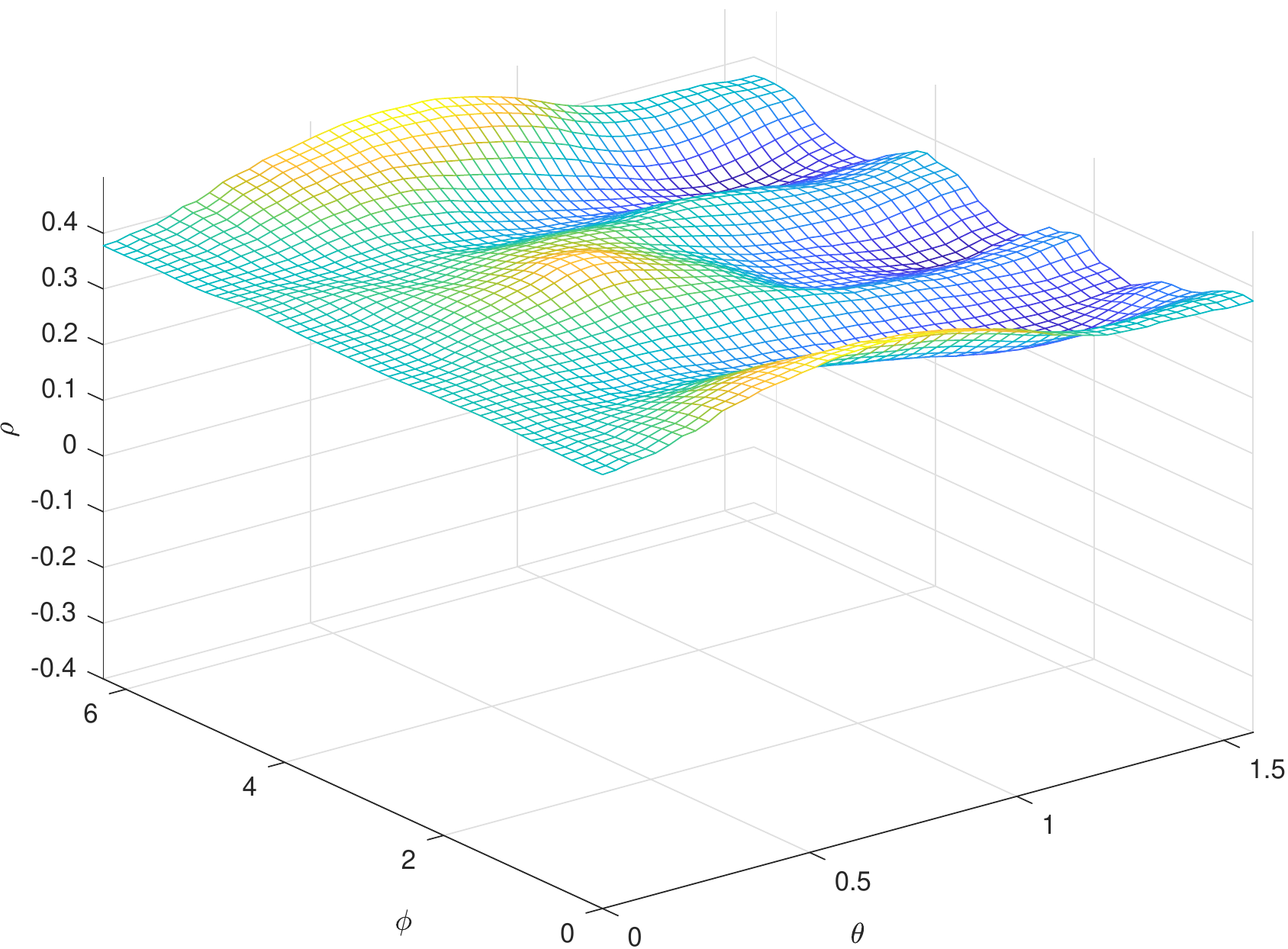}}
      \quad
      \includegraphics[width=4cm]{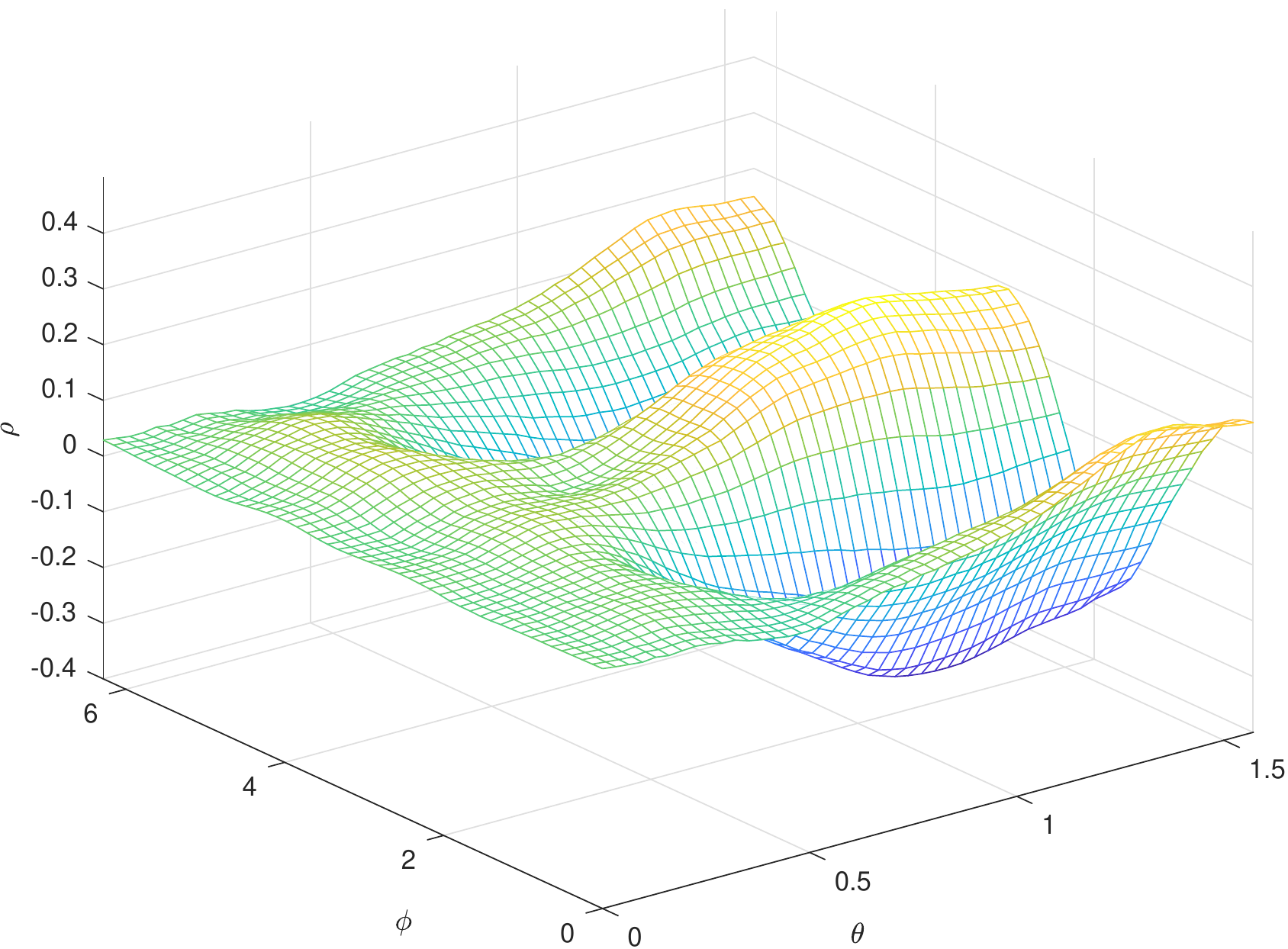}
  }
  \center{
{\includegraphics[width=4cm]{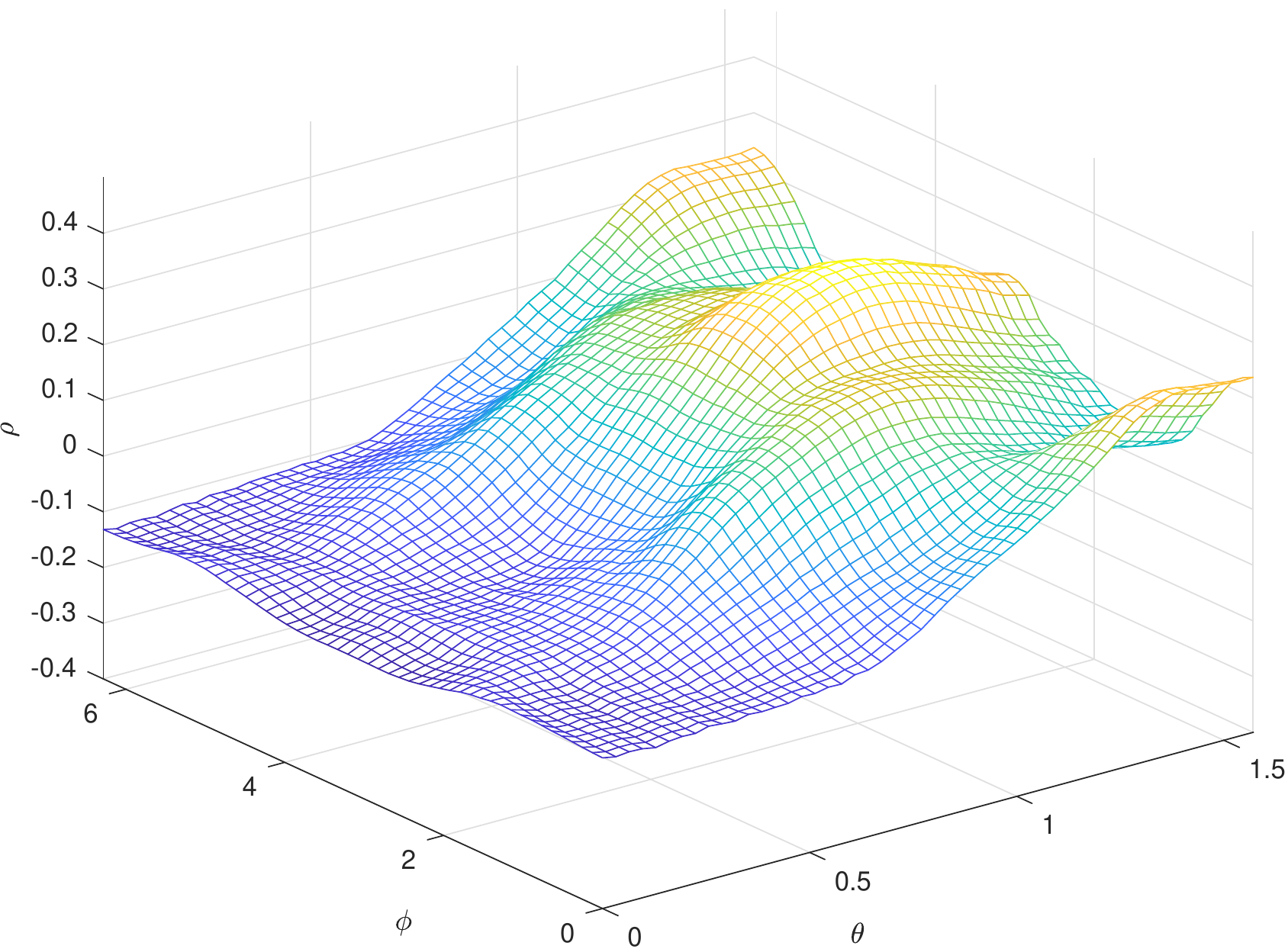}}
  \quad
      {\includegraphics[width=4cm]{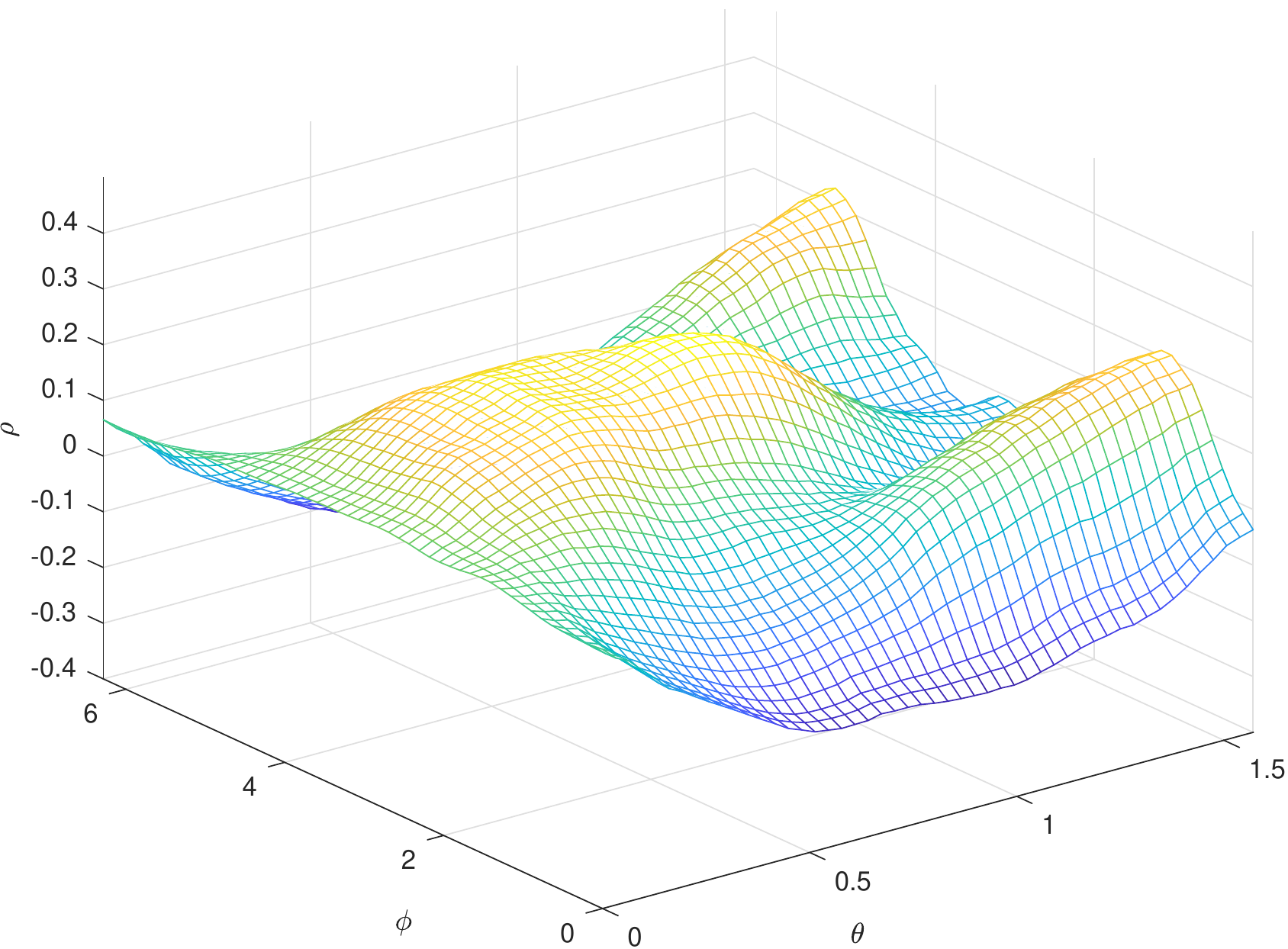}}
      \quad
      \includegraphics[width=4cm]{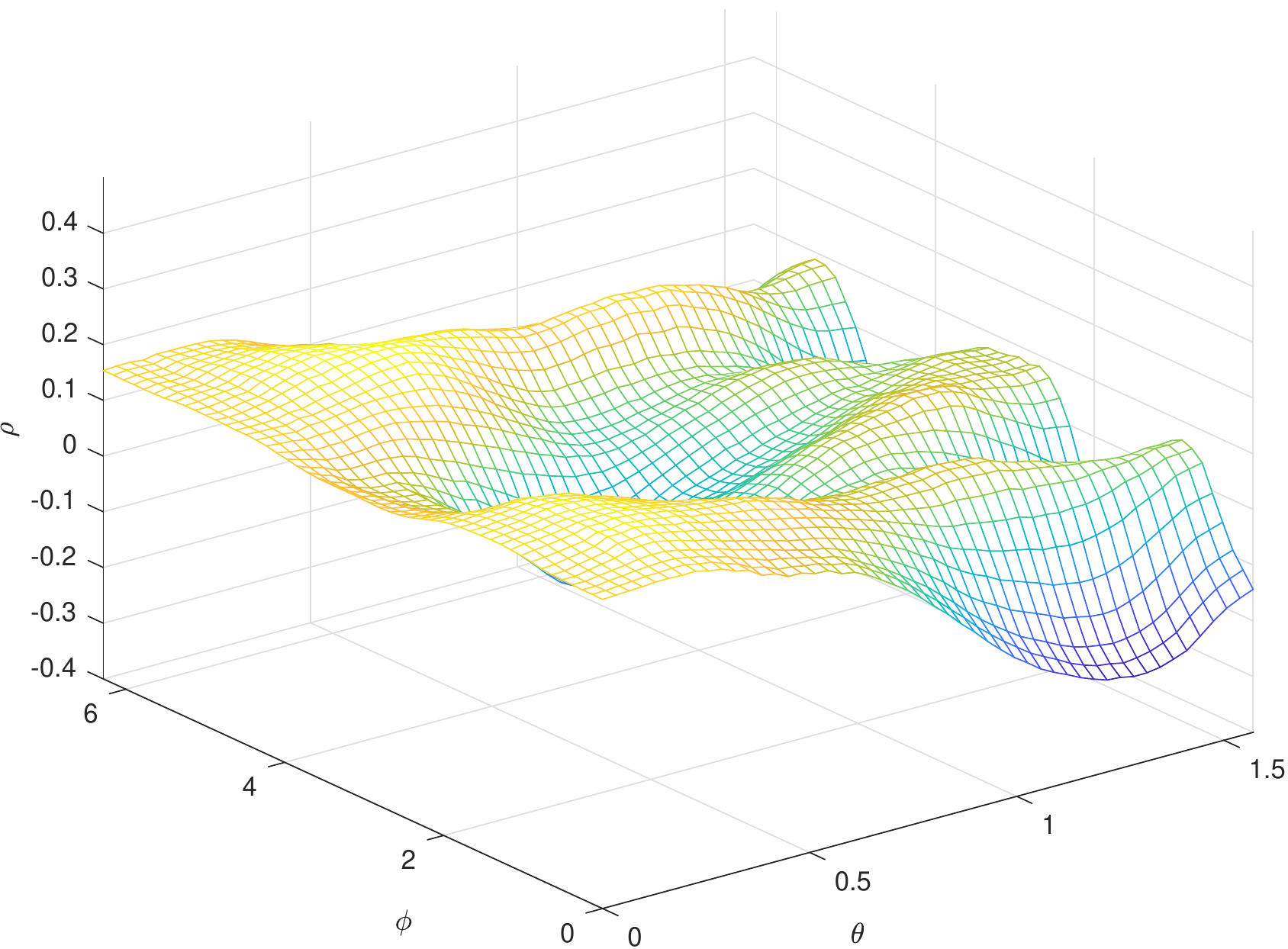}
  }
  \caption{The representation $\bar{r}: \mathbb{S}_+ \to \R_+$ for the mean head together with the first five basis functions for the optimal subspace $V_{\tilde{n}}$ as functions of the polar and azimuthal angles ($\tilde{n}\geq5$). Top left: the mean head. Top center: $\hat{\rho}_1:  \mathbb{S}_+ \to \R$. Top right:  $\hat{\rho}_2: \mathbb{S}_+ \to \R$. Bottom left:  $\hat{\rho}_4: \mathbb{S}_+ \to \R$. Bottom center:  $\hat{\rho}_5: \mathbb{S}_+ \to \R$. Bottom right: $\hat{\rho}_6: \mathbb{S}_+ \to \R$.}
\label{fig:pricomp}
\end{figure}

\begin{remark}
\label{remark:remark}
It is easy to see that the minimizing $\tilde{n}$-dimensional subspace $V_{\tilde{n}}$ for \eqref{eq:minim} is unique if $\lambda_{\tilde{n}} > \lambda_{\tilde{n}+1}$. Moreover, Lemma~\ref{lemma:prcomp} holds in exactly the same form even if $\rho_1, \dots, \rho_n$ are linearly dependent as long as there are at least $\tilde{n}$ linearly independent functions among them. On the other hand, if $\tilde{n}$ is larger or equal to the number of linearly independent functions among $\rho_1, \dots, \rho_n$, then the left-hand side of \eqref{eq:minim} obviously vanishes for any optimal $\tilde{n}$-dimensional subspace $V_{\tilde{n}}$. Finally, observe that one could as well have used any other Sobolev space $H^s(\mathbb{S}_+)$, $s\in \R$, in the above analysis, which would have led to $s$-dependent optimal representations.
\end{remark}
  
\subsection{Interpretation as principal components}
\label{sec:PC}
In this section, we continue to assume that $\rho_1, \dots, \rho_n \in H^1(\mathbb{S}_+)$ are linearly independent and consider the limiting case $\tilde{n}=n$.
We interpret $\alpha \in \R^n$ appearing in the parametrization \eqref{eq:theparam}, with $\tilde{n}=n$, as a random variable and the projection vectors
\begin{equation}
  \label{eq:alpha_sample}
\alpha^{(j)} := \big[(\rho_j, \hat{\rho}_1)_{H^1(\mathbb{S}_+)}, \dots, (\rho_j, \hat{\rho}_n)_{H^1(\mathbb{S}_+)}\big]^{\rm T}, \qquad j=1, \dots, n,
\end{equation}
onto the optimal basis $\hat{\rho}_1, \dots, \hat{\rho}_n$ as its realizations over our library of heads. This is a natural interpretation because substituting $\alpha = \alpha^{(j)}$ in \eqref{eq:theparam} leads to an exact representation for the $j$th head shape, i.e., for the $j$th realization of the human head in our sample (cf.~\eqref{eq:jth_head} and Remark~\ref{remark:remark}).

It follows trivially from \eqref{eq:mean_etc} that the sample mean of $\{ \alpha^{(1)}, \dots, \alpha^{(1)} \} \subset \R^n$ is zero. Moreover, the associated sample covariance matrix $\Gamma_\alpha\in \R^{n \times n}$ can be given componentwise as
\begin{equation}
  \label{eq:Gamma_alpha}
(\Gamma_\alpha)_{kl} = \frac{1}{n-1} \sum_{j=1}^n(\rho_j, \hat{\rho}_k)_{H^1(\mathbb{S}_+)} (\rho_j, \hat{\rho}_l)_{H^1(\mathbb{S}_+)} =
\frac{1}{n-1} \, w_k^{\rm T} R^2 w_l = \frac{\lambda_k}{n-1} \, \delta_{kl},
\end{equation}
where we used the same logic as in \eqref{eq:rayleigh} together with the orthogonality of the scaled eigenvectors of $R$ defined by \eqref{eq:sceig}. In light of the available sample, the random variables $\alpha_1, \dots, \alpha_n$ are thus uncorrelated and the corresponding variances form a non-increasing sequence as do the eigenvalues of $R$. Since $\alpha_1, \dots, \alpha_n$ appear as the multipliers of the basis functions $\hat{\rho}_1, \dots, \hat{\rho}_n$ in the parametrization \eqref{eq:theparam} that is exact for the $j$th head if $\alpha$ takes its $j$th realization, the functions $\hat{\rho}_1, \dots, \hat{\rho}_n$ can be interpreted as the {\em principal components} of the sample $\{ \rho_1, \dots, \rho_n \}$; see,~e.g.,~\cite{Jolliffe02} for more information on principal component analysis.

In the numerical experiments of Section~\ref{sec:numer}, the shape coefficient vector $\alpha \in \R^n$ is assumed to be {\em a priori} distributed as $\alpha \sim \mathcal{N}(0, \Gamma_\alpha)$ with the diagonal covariance matrix $\Gamma_\alpha \in \R^{n \times n}$ defined by \eqref{eq:Gamma_alpha}. To visually examine the validity of such a hypothesis, Figure~\ref{fig:histo} depicts the histograms of the first six components of $\alpha \in \R^n$ over the sample \eqref{eq:alpha_sample}. Due to the relatively small sample size, Figure~\ref{fig:histo} provides no conclusive evidence for the normality assumption, but it cannot be excluded either.

\begin{figure}
  \center{
    \includegraphics[width=7cm]{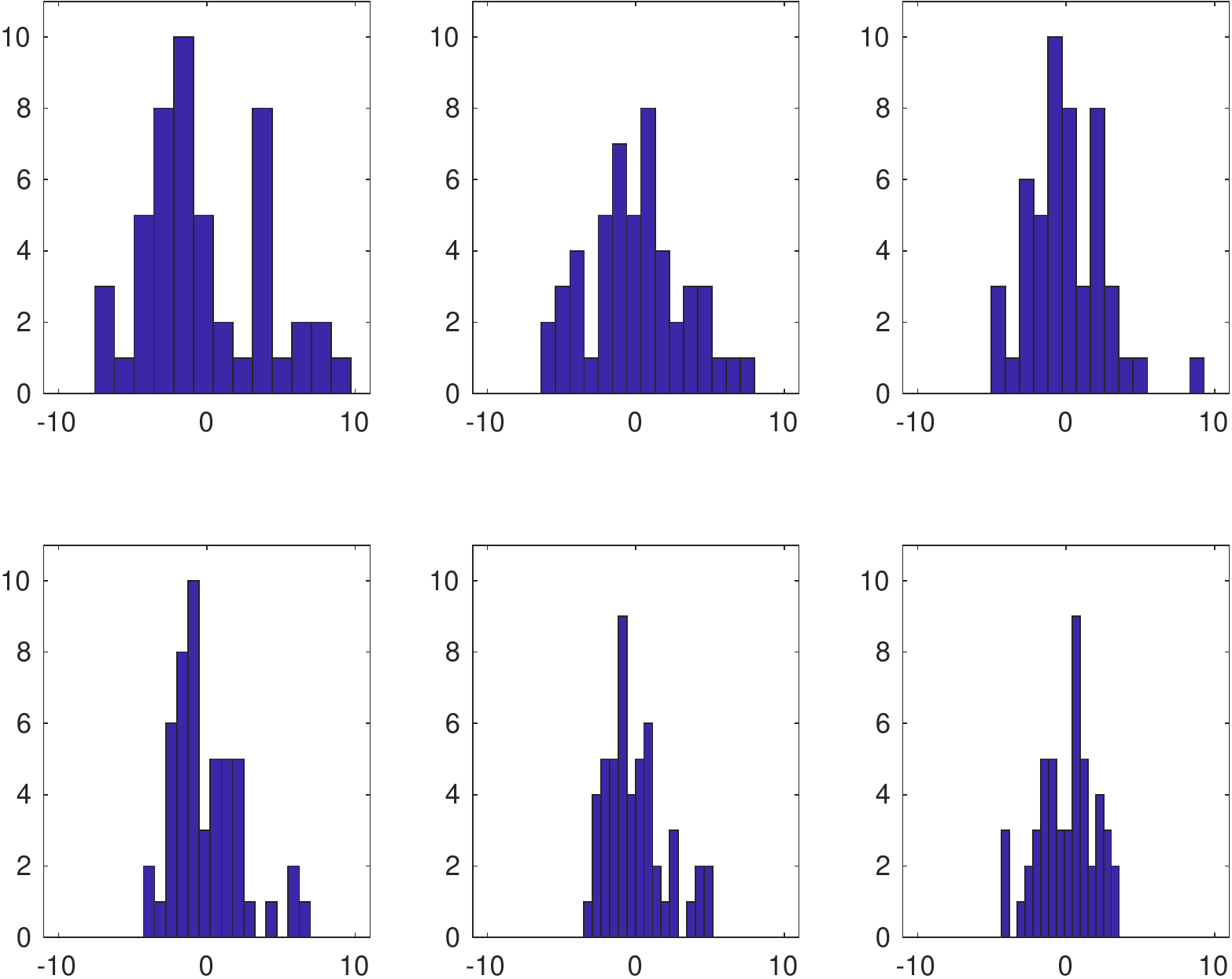}
    }
  \caption{Histograms of the samples $\{ \alpha_k^{(1)}, \dots, \alpha_k^{(n)} \}$ for $k=1, \dots, 6$. Top row: from left to right $k=1,2,3$. Bottom row: from left to right $k=4,5,6$.}
\label{fig:histo}
\end{figure}

\subsection{Building a FEM mesh with electrodes}
\label{sec:FEMmesh}
We start by introducing an exact definition for a set of circular electrodes $E_1, \dots, E_M \subset \partial \Omega$ with preassigned center points $x_1, \dots, x_M \in \partial \Omega$ and radii $R_1, \dots, R_M$. To this end, assume that the portion of $\partial \Omega$ to which the electrodes are attached is of class $\mathcal{C}^1$.
Let $Q_x: \R^3 \to \R^2$ be an orthogonal projection
onto the tangent plane of $\partial \Omega$ at $x \in \partial \Omega$, with the tangent plane identified with $\R^2$ via choosing on it a local orthonormal coordinate system for which $x$ acts as the origin.\footnote{It is easy to check that the nonuniqueness in the choice of such a coordinate system does not affect the definition of a set of circular electrodes by \eqref{eq:elec}.}  Let $B_R(x) \subset \R^3$ denote an open ball of radius $R>0$ around $x$. Using a compactness argument, it can be shown that the restriction $Q_x: B_R(x)\cap \partial \Omega \to \R^{2}$ is injective for all $x \in \partial \Omega$ if $R \leq R_0$ for some $R_0 = R_0(\Omega) > 0$; see~\cite{Darde12} for closely related results.
We then define
\begin{equation}
  \label{eq:elec}
E_m := \big\{ x \in \partial \Omega \cap  B_{R_0}(x_m) \ | \ \| Q_{x_m} x \|_2 < R_m \big \}, \qquad m=1,\dots, M,
\end{equation}
which results in electrodes that have circular projections onto the tangent planes at their respective centers if $R_1, \dots, R_M$ are small enough compared to $R_0$. In other words, the electrodes are modeled as elastic: although their projections are disks, the electrodes themselves elastically take the shape of the boundary they are attached to. This can be considered a reasonable approximation as long as the electrodes are small compared to the characteristic variations in the shape $\partial \Omega$.

Our workflow for generating a tetrahedral mesh for the head model consists of three steps: generation of an initial surface mesh, insertion of electrodes, and tetrahedral mesh generation. The initial surface mesh is constructed by subdividing $k$ times a coarse surface partition consisting  of four triangles, where $k \in \N$ can be chosen by the operator of the algorithm. An electrode $E_m$, $m=1, \dots, M$, is inserted onto the surface mesh by finding the triangles that intersect  the corresponding `extended electrode' 
$$
\hat{E}_m := \big \{ x \in \partial \Omega \ | \ \| Q_{x_m} x \|_2 < \mu R_m  \big \}
$$ 
with a suitably chosen $\mu > 1$. The union of these (closed) triangles is denoted by $S_m$. The nodes of the surface mesh on the boundary $\partial S_m$ are then projected onto the tangent plane at $x_m$ by $Q_{x_m}$. A (dense) mesh $\mathcal{T}_m$ for the resulting polygonal domain is generated using the Triangle software \cite{Shewchuk96}. To avoid hanging nodes, no new vertices are introduced on the boundary of the polygon at this stage. Subsequently, the surface mesh is modified by replacing the original triangles on $S_m$ with the (denser) triangulation $Q_{x_m}^{-1} \mathcal{T}_m$. After inserting all $M$ electrodes, the process is completed by generating a tetrahedral partition for the whole volume by TetGen \cite{Hang15} starting from the formed surface mesh.

\section{Computation of derivatives}
\label{sec:deriv}

The electrode potentials $U \in \R_\diamond^M$ corresponding to a given electrode current pattern $I \in \R_\diamond^M$ can be interpreted as a function of five variables:
\begin{equation}
  \label{eq:Uparam}
U: \, (\sigma, \zeta, \alpha, \beta; I) \mapsto U(\sigma, \zeta, \alpha, \beta; I),
\end{equation}
where $\sigma \in L^\infty_+(\Omega)$ is the conductivity, $\zeta: \partial \Omega \to \C$ is the contact admittance satisfying \eqref{eq:zeta}, $\alpha \in \R^{\tilde{n}}$ is the parameter vector determining the shape of the computational head model in \eqref{eq:theparam}, and $\beta = (\beta_1, \dots, \beta_M) \in [\mathbb{S}_+]^M$ define the directions of the electrode centers from the origin of the parametrization \eqref{eq:theparam}. The aim of this section is to summarize how the partial derivatives of $U$ can be efficiently computed. To maintain readability, we usually suppress the dependence of $U$ on the variables that are not in the focus of our attention at a particular point in the text.

\subsection{Derivative with respect to $\sigma$}

It is well known that the electrode potential pattern $U$ is Fr\'echet differentiable with respect to $\sigma$; see,~e.g.,~\cite{Garde17,Lechleiter06}. To be more precise, the Fr\'echet derivative of $U$ at $\sigma \in L^\infty_+(\Omega)$ in the direction $\eta \in L^\infty(\Omega)$, i.e.~$D_\sigma U(\sigma; \eta) \in \C_\diamond^M$, can be assembled using the formula (see, e.g., \cite{Hyvonen18})
\begin{equation}
\label{eq:sderiv}
D_\sigma U(\sigma; \eta) \cdot \tilde{I} = - \int_{\Omega} \eta \nabla u \cdot \nabla \tilde{u} \, {\rm d} x,
\end{equation}
where $(u, U) \in \mathcal{H}^1$ is the solution to \eqref{eq:cemeqs} for the applied current pattern $I \in \C_\diamond^M$ and $(\tilde{u}, \tilde{U}) \in \mathcal{H}^1$ is the forward solution for an auxiliary current pattern $\tilde{I}$. We frequently employ similar `tilde notation' in what follows.

Observe that knowing the projection $D_\sigma U(\sigma; \eta) \cdot \tilde{I}$ for all $\tilde{I}$ in a basis for $\C_\diamond^M$ uniquely determines $D_\sigma U(\sigma; \eta) \in \C_\diamond^M$, and since $M$ is typically well below hundred, this poses no computational challenges. On the other hand, if $M-1$ linearly independent current patterns are injected through the electrodes when taking the EIT measurements (as is the typical procedure), one must in any case solve the forward problem \eqref{eq:cemeqs} with these same current patterns at each iteration of a Newton-type reconstruction algorithm. In consequence, one need not solve any extra forward problems to be able to perform the necessary evaluations on the right-hand side of \eqref{eq:sderiv}. This means that the extra cost for computing the needed derivatives with respect to the conductivity corresponds to evaluating the integral on the right-hand side of \eqref{eq:sderiv} when $\eta$ runs through the degrees of freedom in the conductivity parametrization and $(u,U)$, $(\tilde{u}, \tilde{U})$ over all pairs of the already computed forward solutions. The same conclusion also applies to the formulas \eqref{eq:zderiv} and \eqref{eq:bderiv} below.

\subsection{Derivative with respect to $\zeta$}

In order to be able to discuss differentiation with respect to $\zeta$, it is convenient to first describe a suitable parametrization for the considered contact admittances. To this end, we introduce a contact admittance (half) shape
$$
\hat{\zeta}: [0,1) \to \R_+.
$$
We then parametrize $\zeta: \partial \Omega \to \C$ as
\begin{equation}
\label{eq:zeta_param}
\zeta(x) =
\left\{
\begin{array}{ll}
  \zeta_m  \, \hat{\zeta} \big(\|Q_{x_m} x \|_2 /R_m \big) \quad & {\rm if} \ x \in E_m, \ \ m=1, \dots, M, \\[1mm]
  0 \quad & {\rm otherwise},
\end{array}
\right.
\end{equation}
where the orthogonal projection $Q_{x_m}: \R^3 \to \R^2$ is as defined in Section~\ref{sec:FEMmesh} and $\zeta_m \in \C$, with ${\rm Re}(\zeta_m) > 0$, defines the relative admittance on the $m$th electrode. Based on the material in \cite{Vilhunen02}, it is straightforward to show that the derivative of $U$ with respect to $\zeta_m$, $m=1, \dots, M$, can be assembled via (cf.~\cite[(26)]{Hyvonen18})
\begin{equation}
  \label{eq:zderiv}
\frac{\partial  U(\zeta)}{\partial \zeta_m} \cdot \tilde{I} = - \int_{E_m} \hat{\zeta} \big(\|Q_{x_m} x \|_2/R_m \big) (U_m - u) (\tilde{U}_m - \tilde{u}) \, {\rm d} S_x,
\end{equation}
where we have used the same notation as in \eqref{eq:sderiv}.

\subsection{Derivative with respect to $\beta$}
To begin with, let us consider a general way of perturbing the shapes and sizes of the electrodes; see \cite{Darde12,Hyvonen17} for more information. Assume the portion of $\partial \Omega$ to which the electrodes are attached is of class $\mathcal{C}^2$. Let $a \in \mathcal{C}^1(E, \R^3)$ and define a perturbed set of electrodes via
\begin{equation}
\label{eq:perturbed}
E_m^a = \big\{ P_x \big(x +a(x) \big) \, \big| \, x \in E_m \big\} \subset \partial \Omega, \qquad m=1,\dots, M,
\end{equation}
where $P_x: \R^3 \supset B_R(x) \to \partial \Omega$, for some $R>0$, is the projection in the direction of $\nu(x)$ onto $\partial \Omega$. The contact admittance $\zeta: \partial \Omega \to \C$ is assumed to stretch accordingly, that is, the admittance on the perturbed electrodes is defined via
\begin{equation}
\label{eq:perturbedz}
\zeta^a\big(P_x (x + a(x) ) \big) = \zeta(x), \qquad x \in E,
\end{equation}
and it vanishes on $\partial \Omega \setminus \overline{E}^a$. It can be shown that there exists a constant $c = c(E,\Omega) >0$ such that the definitions \eqref{eq:perturbed} and \eqref{eq:perturbedz} are unambiguous if $\|a \|_{\mathcal{C}^1(E, \R^3)} < c$; see~\cite{Darde12} for the proof of a closely related result.

The electrode potential pattern $U = U(a) \in \C_\diamond^M$ can obviously be interpreted as a function of (a small enough) $a \in \mathcal{C}^1(E, \R^3)$. Indeed, for a given $a$ one simply considers the perturbed electrodes $E_1^a, \dots, E_M^a$ and contact admittance $\zeta^a$ in \eqref{eq:cemeqs} in place of the original unperturbed ones, which naturally leads to an $a$-dependent forward solution $(u(a), U(a)) \in \mathcal{H}^1$. 

The mapping $\mathcal{C}^1(E, \R^m) \ni a \mapsto U(a) \in \R_\diamond^M$ is Fr\'echet differentiable at the origin~\cite{Darde12,Hyvonen17b}. The associated Fr\'echet derivative in the direction $a \in \mathcal{C}^1(E, \R^m)$ can be assembled using
\begin{equation}
\label{eq:bderiv}
D_a U(0; a) \cdot \tilde{I} = \int_{E} a_\tau \cdot {\rm Grad}(\zeta) \, (U_m - u) (\tilde{U}_m - \tilde{u}) \, {\rm d} S,
\end{equation}
where $a_\tau: E \to \R^3$ is the tangential component of $a \in \mathcal{C}^1(E, \R^3)$ and ${\rm Grad}$ denotes the surface gradient~\cite{Colton98}. For a piecewise constant $\zeta$, as in the traditional CEM, the right-hand side of \eqref{eq:bderiv} reduces to an integral over $\partial E$~\cite{Darde12}.

In order to utilize \eqref{eq:bderiv} in a reconstruction algorithm, we still need to define vector fields that (approximately) correspond to changing the direction of, say, the $m$th electrode, i.e., the parameter $\beta_m \in \mathbb{S}_+$ in \eqref{eq:Uparam}. To this end, we represent $\beta_m = \beta_m(\theta_m, \phi_m)$ as a function of its polar $\theta_m \in (0, \pi/2)$ and azimuthal $\phi_m \in [0, 2\pi)$ angles over the upper hemisphere $\mathbb{S}_+$. We denote by $\hat{\theta}_m \in \R^3$ and $\hat{\phi}_m \in \R^3$ the tangent vectors for $\partial \Omega$ at the center point $x_m$ of $E_m$
obtained by (numerically) differentiating the parametrization \eqref{eq:theparam} with respect to $\theta$ and~$\phi$, respectively.

We extend $\hat{\theta}_m$ as a constant vector field over the whole electrode $E_m$ and define
$$
\tilde{a}^\theta_m(x) := \big( \hat{\theta}_m(x) \big)_\tau, \qquad x \in E_m,
$$
to be its tangential component. (Note that $\tilde{a}^\theta_m$ depends on the location on the electrode as does the tangent plane.)  To ensure that all points on $E_m$ move as much in the tangential direction,~i.e.,~that the size of the electrode is approximately maintained, we finally normalize
$$
a^\theta_m(x) := \frac{\| \tilde{a}^\theta_m(x_m)\|_2}{\| \tilde{a}^\theta_m(x) \|_2} \, \tilde{a}^\theta_m(x), \qquad x \in E_m,
$$
and extend $a^\theta_m$ as zero to the rest of $E$. Although the perturbation of $E_m$ by such a vector field $a^\theta_m$ in the spirit of \eqref{eq:perturbed} does not exactly maintain the circular shape of the elastic electrode $E_m$, it seems feasible that the variations in its shape and size are typically of higher order compared to the moved distance --- especially if the electrodes are small compared to the characteristic variations in the shape of $\partial \Omega$. A perturbation field for $E_m$ in the direction of $\hat{\phi}_m$ is introduced in the same manner.

To summarize, we slightly abuse the notation by identifying $U(\beta) = U(\theta, \phi)$, where $\theta = [\theta_1, \dots, \theta_M]^{\rm T}$ and  $\phi = [\phi_1, \dots, \phi_M]^{\rm T}$, and approximate the corresponding derivatives of $U$ based on the formulas
\begin{equation}
\label{eq:phideriv}
\frac{\partial  U(\theta, \phi)}{\partial \theta_m} \cdot \tilde{I} \approx \int_{E_m} a_m^\theta \cdot {\rm Grad} (\zeta) \, (U_m - u) (\tilde{U}_m - \tilde{u}) \, {\rm d} S
\end{equation}
and
\begin{equation}
\label{eq:thederiv}
\frac{\partial  U(\theta, \phi)}{\partial \phi_m} \cdot \tilde{I} \approx \int_{E_m} a_m^\phi \cdot {\rm Grad}(\zeta) \, (U_m - u) (\tilde{U}_m - \tilde{u}) \, {\rm d} S.
\end{equation}
for $m=1, \dots, M$.

\begin{remark}
\label{remark:theta_phi}
It would definitely be worthwhile to analyze how the vector fields $a_m^\theta: E \to \R^3$ and $a_m^\phi: E \to \R^3$ should be defined so that they would exactly correspond to differentiation with respect to the polar and azimuthal angles for the center point of the $m$th electrode. However, as the above introduced crude approximations seem to function satisfactorily in the reconstruction algorithm introduced in Section~\ref{sec:numer}, we postpone such analysis to future studies. In any case, it should be emphasized that utilizing the formulas \eqref{eq:phideriv} and \eqref{eq:thederiv} is computationally far more attractive than resorting to difference approximations: Using, e.g., central differences to approximate the derivatives on the left-hand sides of \eqref{eq:phideriv} and \eqref{eq:thederiv} would require remeshing and computing forward solutions of \eqref{eq:cemeqs} for $4 M$ extra measurement geometries. In our numerical studies with $M=32$ electrodes, this would correspond to forming $128$ extra FEM meshes and solving the associated forward problems with $M-1 = 31$ different current patterns at each step of the reconstruction algorithm.
\end{remark}
  
\subsection{Derivative with respect to $\alpha$}
In principle, one could try to employ the sampling formulas introduced in \cite{Darde13a,Hyvonen17b} to compute the Fr\'echet derivative of $U = U(\alpha)$ with respect to the shape parameter vector $\alpha \in \R^{\tilde{n}}$ in \eqref{eq:theparam}. However, the analysis in \cite{Darde13a,Hyvonen17b} is based on the assumption that the electrodes stretch accordingly when the shape of $\Omega$ changes. For cylindrical domains, as are the ones considered in \cite{Darde13a,Darde13b}, compensating for such stretching is straightforward because an electrode is essentially determined by its two end points; see~\cite{Darde13a}. However, in our inherently three-dimensional setting, analyzing how the changes in the sizes and shapes of the electrodes affect the shape derivatives becomes more complicated, and we choose to leave such consideration for future studies.

As discussed in Section~\ref{sec:optimpara}, we use $\tilde{n}=5$ in our numerical experiments, and thus approximating the derivatives with respect to such a low number of shape parameters can be done by resorting to,~e.g.,~central differences without increasing the computational load out of proportion. That is, we numerically approximate
\begin{equation}
  \label{eq:adervi}
\frac{\partial  U(\alpha)}{\partial \alpha_j} \approx \frac{U(\alpha+\epsilon \, {\rm e}_j) - U(\alpha-\epsilon\, {\rm e}_j)}{2 \epsilon}, \qquad j=1, \dots, \tilde{n},
\end{equation}
where ${\rm e}_j$ is the $j$th Cartesian basis vector and $\epsilon > 0$ is chosen appropriately.

Observe that \eqref{eq:adervi} is in any case computationally significantly less attractive than \eqref{eq:sderiv}, \eqref{eq:zderiv} and \eqref{eq:bderiv} since it requires extra solutions of the forward problem \eqref{eq:cemeqs} as well as remeshing of the domain. To be more precise, the use of \eqref{eq:adervi} demands solving the forward problem with all employed current patterns for $2\tilde{n}$ extra domain shapes. Hence, it is imperative to keep $\tilde{n}$ as low as possible in the reconstruction algorithm (cf.~Remark~\ref{remark:theta_phi}).

\section{Numerical experiments}
\label{sec:numer}

This section presents our numerical examples. We start by briefly reviewing the reconstruction algorithm that is based on a simple Gauss--Newton iteration; for more information consult,~e.g.,~\cite{Darde13b,Hyvonen17c}. Next, we explain how realistic measurement data are simulated. Finally, some reconstructions are presented with emphasis on analyzing how much ignoring different kinds of geometric inaccuracies in the forward model deteriorate their quality.

Although all theory presented above applies to complex-valued admittivities and contact admittances as well as to the resulting complex-valued potential measurements, we assume from this point on that the frequency of the alternating current fed through the electrodes is so low that the phase information of the measurements can be ignored and the amplitudes of the net currents and electrode potentials can be treated as if they resulted from measurements with direct current; see \cite{Vauhkonen97} for detailed analysis on the validity of such an approximation. In particular, we assume that $\sigma$ and $\zeta$ are real-valued and dub them the {\em conductivity} and the {\em contact conductance}, respectively, in what follows.

\subsection{The algorithm}
We assume to be able to drive $M-1$ linearly independent current patterns $I^{(1)}, \dots, I^{(M-1)} \in \R_\diamond^{M}$ through the electrodes and measure the corresponding noisy electrode potentials denoted by $V^{(1)}, \dots, V^{(M-1)} \in \R^{M}$. The potential measurements are packed into a single vector
\begin{equation}
\label{data}
\mathcal{V} := \big[(V^{(1)})^{\rm T},\ldots, (V^{(M-1)})^{\rm T}\big]^\T\in \R^{M(M-1)},
\end{equation}
and we  analogously  introduce the stacked forward map
$$
\mathcal{U}: \R_+^N \times \R_+^M \times \R^{\tilde{n}} \times (0,\pi/2)^M \times [0,2\pi)^M \to \R^{M(M-1)}
$$
via
$$
\mathcal{U}(\sigma, \zeta, \alpha, \theta, \phi)
= \big[U(\sigma, \zeta, \alpha, \theta, \phi; I^{(1)})^{\rm T},\ldots, U(\sigma, \zeta, \alpha, \theta, \phi; I^{(M-1)})^{\rm T}\big]^\T.
$$
Here, the conductivity $\sigma \in \R_+^N$ is identified with the $N \in \N$ degrees of freedom used to parametrize it, the contact conductance $\zeta \in \R^M$ is identified with the coefficient vector in \eqref{eq:zeta_param}, $\alpha \in \R^{\tilde{n}}$ is the parameter vector in \eqref{eq:theparam} determining the shape of the computational head model, and $\theta \in (0,\pi/2)^M$ and  $\phi \in [0,2\pi)^M$ define the polar and azimuthal  angles of the electrode center points, respectively. As the shape of the computational domain changes during the iteration, the degrees of freedom $\sigma \in \R_+^N$ are chosen to be the coefficients of a piecewise linear parametrization for the conductivity with respect to a FE mesh in a large {\em storage head} that encloses all incarnations of the computational head model encountered by our iterative reconstruction algorithm. The conductivity parametrization is transferred between the current computational head and the storage head via linear projections between the corresponding FE meshes.
 
The reconstruction algorithm aims at computing a minimizer for the Tikhonov functional
\begin{align}
  \label{eq:tikhfun}
F(\sigma, \zeta, \alpha, \theta, \phi)  & = \big\|\mathcal{U}(\sigma, \zeta, \alpha, \theta, \phi) - \mathcal{V}\big\|_{\Gamma^{-1}_{\eta}}^2    + \big\|\sigma - \bar{\sigma} \big \|_{\Gamma^{-1}_\sigma}^2 + \big\|\zeta -\bar{\zeta} \big \|_{\Gamma^{-1}_\zeta}^2 \nonumber \\ & \quad + \big\|\alpha - \bar{\alpha} \big\|_{\Gamma^{-1}_\alpha}^2 + \big\|\theta -\bar{\theta} \big\|_{\Gamma^{-1}_\theta}^2  + \big\|\phi -\bar{\phi} \big\|_{\Gamma^{-1}_\phi}^2,
\end{align}
where the weighted norms are defined by $\|x\|_A^2 := x^\T \! A x$. A minimizer of \eqref{eq:tikhfun} can be interpreted as a {\em maximum a posteriori} (MAP) estimate for the to-be-reconstructed parameters, assuming the measurements are corrupted by additive Gaussian noise with zero mean and covariance $\Gamma_\eta \in \R^{M(M-1) \times M(M-1)}$ and the prior distributions for  $\sigma, \zeta, \alpha, \theta$ and $\phi$ are also Gaussian with the mean-covariance pairs $(\bar{\sigma}, \Gamma_\sigma)$, $(\bar{\zeta}, \Gamma_\zeta)$, $(\bar{\alpha}, \Gamma_\alpha)$, $(\bar{\theta}, \Gamma_\theta)$ and $(\bar{\phi}, \Gamma_\phi)$, respectively. The Cholesky factors of the above introduced positive definite covariance matrices are denoted by $L_*$, where $*$ stands for the variable in question. Moreover,
$$
J_{\mathcal{U}}: \R_+^N \times \R_+^M \times \R^{\tilde{n}} \times (0,\pi/2)^M \times [0,2\pi)^M \to \R^{M(M-1) \times (N+3M+\tilde{n})} 
$$
denotes the Jacobian matrix of $\mathcal{U}(\sigma, \zeta, \alpha, \theta, \phi)$ that can be numerically approximated by following the guidelines in Section~\ref{sec:deriv}.


\begin{algorithm} \label{alg:1}
Assume the covariance matrices $\Gamma_{\eta}, \Gamma_\sigma, \Gamma_\zeta, \Gamma_\alpha$, $\Gamma_\theta$, $\Gamma_\phi$ and the expected values $\bar{\alpha}$, $\bar{\theta}$, $\bar{\phi}$  are given. Compute Cholesky factorizations for the positive definite covariance matrices and form a block-diagonal matrix
$$
L = \diag (L_\sigma, L_\zeta, L_\alpha, L_\theta, L_\phi).
$$
Choose the prior means for the conductivity and contact conductances to be the homogeneous estimates $\bar{\sigma} = \tau_{\sigma}\mathbf{1}$ and $\bar{\zeta} = \tau_{\zeta}\mathbf{1}$ where
$$
(\tau_{\sigma}, \tau_{\zeta}) = \underset{(\tau_\sigma, \tau_\zeta) \in \R_+^2}{\arg \min} \,\big\| L_{\eta} \big(\mathcal{U} (\tau_\sigma \mathbf{1}, \tau_\zeta \mathbf{1}, \bar{\alpha}, \bar{\theta}, \bar{\phi})- \mathcal{V} \big) \big\|_2^2,
$$
and ${\bf 1} = [1, \ldots, 1]^\T$ is at each occurrence a constant vector of the appropriate length.

Initialize
$$
b^{(0)} = \big[\bar{\sigma}^\T, \bar{\zeta}^\T, \bar{\alpha}^\T, \bar{\theta}^\T, \bar{\phi}^\T \big]^\T
$$
and iterate as follows starting from $j=0$:
\begin{enumerate}
\item Form
\begin{align*}
A = \begin{bmatrix}
L_{\eta} J_{\mathcal{U}}(b^{(j)}) \\ L
\end{bmatrix}
\qquad \text{and} \qquad
y = \begin{bmatrix}
L_{\eta} \big(\mathcal{U}(b^{(j)})-\mathcal{V}\big) \\ L (b^{(j)} - b^{(0)})
\end{bmatrix}.
\end{align*}
\item Solve the direction $\Delta b$ from 
$$
\Delta b = \underset{z}{\arg\min} \, \|A z - y\|_2^2.
$$
\item Set $\tilde{b}^{(j+1)} = b^{(j)} - q \Delta b$ where $0 < q \leq 1$ is fixed by the operator of the algorithm.
\item
\begin{enumerate}
  \item  If $F(\tilde{b}^{(j+1)}) < F(b^{(j)})$, set $b^{(j+1)} = \tilde{b}^{(j+1)}$, update $j=j+1$ and return to step 1.
  \item If $F(\tilde{b}^{(j+1)}) \geq F(b^{(j)})$, perform a brief line search on $[b^{(j)}, \tilde{b}^{(j+1)}]$. If decrease in the value of $F$ is observed, dub the corresponding minimizer $b^{(j+1)}$ and return to step 1; otherwise, terminate the iteration.
    \end{enumerate}
\end{enumerate}
Define $[\sigma_*^\T, \zeta_*^\T, \alpha_*^\T, \theta_*^\T, \phi_*^\T]^\T := b^{(j)}$ to be the reconstruction.
\end{algorithm}

We choose $q=0.5$ in all numerical examples documented in Section~\ref{sec:reco}, which seems to provide a suitable trade-off between robustness and speed of convergence for the considered setups.

\subsection{Simulation of measurement data}
\label{sec:data}

We consider a couple of different target heads. In each case, the parameters $\alpha_{\rm trgt} \in \R^{n_0}$ defining the shape of the target are drawn from (a variant of) the distribution $\mathcal{N}(0, \Gamma_\alpha)$, where, allowing a slight abuse of notation, $\Gamma_\alpha \in \R^{n_0 \times n_0}$ is the diagonal covariance matrix defined componentwise by \eqref{eq:Gamma_alpha}; for a motivation of this choice, see the principal component construction in Section~\ref{sec:PC}. The number of components in $\alpha_{\rm trgt}$ is chosen considerably higher than the number of shape parameters reconstructed by Algorithm~\ref{alg:1},~i.e.,~$n_0=10$ vs.~$\tilde{n}=5$. In other words, the heads used for simulating data are more detailed than the reconstruction algorithm is capable to accurately represent.

The expected values for the polar and azimuthal angles of the electrode centers, i.e.~$\bar{\theta}$ and $\bar{\phi}$, are chosen to correspond to the angular positions to which one originally aimed to attach the electrodes. The actual central angles of the target electrodes, i.e.~$\theta_{\rm trgt}$ and $\phi_{\rm trgt}$, are then drawn from the distributions $\mathcal{N}(\bar{\theta}, \Gamma_\theta)$ and $\mathcal{N}(\bar{\phi}, \Gamma_\phi)$, where
\begin{equation}
\label{eq:Gamma_thph}
\Gamma_\theta = \varsigma_\theta^2 \mathbb{I} \qquad {\rm and} \qquad
\Gamma_\phi = \varsigma_\phi^2 \mathbb{I}. 
\end{equation}
Here, $\mathbb{I} \in \R^{M \times M}$ is the identity matrix and $\varsigma_\theta,\varsigma_\phi>0$ determine the standard deviations in the two angular directions. Notice that $\varsigma_\theta$ and $\varsigma_\phi$ must be chosen so small that the electrodes are not at a risk to overlap or move outside the crown of the computational head. The radii of all electrodes are chosen to be the same, say, $R>0$, which is the most common setup encountered in practice.

The relative contact conductances $(\zeta_{\rm trgt})_m \in \R_+$, $m=1, \dots, M$,  appearing in \eqref{eq:zeta_param} are independently drawn from $\mathcal{N}(\tilde{\tau}_\zeta, \tilde{\varsigma}_\zeta^2)$, where $\tilde{\tau}_\zeta > 0$ is chosen so much larger than $\tilde{\varsigma}_\zeta>0$ that negative contact conductances are practically impossible. The corresponding conductance shape is chosen to be $\hat{\zeta}_{\rm trgt} \equiv 1$ in \eqref{eq:zeta_param}, that is, we employ the traditional CEM with constant contact conductances when simulating data. Finally, $\sigma_{\rm trgt} \in \R_+^{N_0}$, $N_0 \in \N$, corresponds to a piecewise linear parametrization for the studied target conductivity on a dense FE mesh associated to the target head and the electrode positions defined by $\alpha_{\rm trgt}$, $\theta_{\rm trgt}$ and $\phi_{\rm trgt}$.


We approximate the ideal data $\mathcal{U}(\sigma_{\rm trgt}, \zeta_{\rm trgt}, \alpha_{\rm trgt}, \theta_{\rm trgt}, \phi_{\rm trgt})$ by FEM with piecewise linear basis functions and denote the resulting (almost) noiseless data by $\mathcal{U}_{\rm trgt} \in \R^{M(M-1)}$. The actual noisy data is then formed as
$$
\mathcal{V} = \mathcal{U}_{\rm trgt} + \eta,
$$
where $\eta \in \R^{M(M-1)}$ is a realization of a zero-mean Gaussian with the diagonal covariance matrix
\begin{equation}
\label{eq:noisecov} 
\Gamma_\eta = \varsigma_\eta^2 \big( \max_{j} \mathcal (\mathcal{U}_{\rm trgt})_j   - \min_{j}  (\mathcal{U}_{\rm trgt})_j \big)^2 \, \mathbb{I},
\end{equation}
with $\mathbb{I} \in \R^{M(M-1) \times M(M-1)}$. The free parameter $\varsigma_\eta >0$ can be tuned to set the relative noise level. Such a noise model has been used with real-world data,~e.g.,~in~\cite{Darde13b}.

\subsection{Reconstructions}
\label{sec:reco}

Let us then focus on the parameter choices for Algorithm~\ref{alg:1}. In \eqref{eq:tikhfun} we employ $\Gamma_\eta \in \R^{M(M-1) \times M(M-1)}$, $\Gamma_\theta \in \R^{M \times M}$ and $\Gamma_\phi \in \R^{M \times M}$ as defined in \eqref{eq:noisecov} and \eqref{eq:Gamma_thph}, respectively. In other words, we assume to know an accurate model for the measurement noise as well as to have a good understanding of the ways the electrodes are typically misplaced. Moreover, it is not far-fetched to assume the intended angular positions for the electrodes, i.e.~$\bar{\theta}$ and $\bar{\phi}$, are also known. The matrix $\Gamma_\alpha \in \R^{\tilde{n} \times \tilde{n}}$, with $\tilde{n} = 5$,  in \eqref{eq:tikhfun} is defined by \eqref{eq:Gamma_alpha} and the corresponding expected value $\bar{\alpha}$ is set to zero; see Section~\ref{sec:PC}.

In the reconstruction algorithm, we choose the form
\begin{equation}
\label{eq:smooth_zeta}   
\hat{\zeta}(t) = \exp\left( 2 - \frac{2}{1 - t^2} \right)
\end{equation}
for the contact shape $\hat{\zeta}: [0,1) \to \R_+$ appearing in \eqref{eq:zeta_param}. This leads to a smooth contact conductance $\zeta: \partial \Omega \to \R_+ \cup \{0\}$ with $\zeta(x_m) = 1$ for each electrode center $x_m$, $m=1, \dots, M$.\footnote{Replacing $2$ by any other positive real number in \eqref{eq:smooth_zeta} would also lead to a smooth contact conductance, but this particular choice seems to function well in our numerical examples.} The `expected value' $\bar{\zeta}$ in \eqref{eq:tikhfun} is selected automatically by Algorithm~\ref{alg:1}. We choose  the corresponding prior covariance to be $\Gamma_\zeta = (2 \tilde{\varsigma}_\zeta \bar{\zeta}/\tilde{\tau}_\zeta )^2 \mathbb{I}$, where $\tilde{\tau}_\zeta$ and  $\tilde{\varsigma}$ are the expected value and standard deviation employed when drawing the target contact conductance values. The reason for this ad hoc choice is two-fold: First of all, as the contact conductance models for the data simulation and the reconstruction process are different, there is no reason to expect that the distributions of the associated parameters are the same; for further analysis, see~\cite{Hyvonen17b}, where it is noted that the smoothened model often leads to significantly higher peak values for the contact conductance function. Secondly, the estimation of the contact conductances is usually a rather stable process which allows one to employ only mild regularization,~i.e.,~twice as large ratio between the prior standard deviation and expected value than used in the data simulation step.

The prior covariance matrix for the conductivity is defined componentwise as 
\begin{align}
  \label{eq:prior}
(\Gamma_\sigma)_{ij} = \varsigma_\sigma^2 \exp \left(-\frac{\|x_i-x_j\|_2^2}{2 l^2} \right), \qquad i,j = 1, \dots, N, 
\end{align}
where $\varsigma_\sigma>0$ is the pointwise standard deviation, $l > 0$ is the correlation length, and $x_i, x_j \subset \R^2$ are coordinates of FE nodes in the storage head. Notice that there is no need to choose the `expected value' $\bar{\sigma}$ in \eqref{eq:tikhfun} as it is automatically selected by Algorithm~\ref{alg:1}. 

The electrode potentials and associated Jacobian matrices needed in Algorithm~\ref{alg:1} were approximated by FEM with piecewise linear basis functions on meshes with about 20\,000 nodes and 85\,000 tetrahedra; see Section~\ref{sec:FEMmesh} and,~e.g.,~the right-hand image of Figure~\ref{fig:target_head1}. The number of degrees of freedom for the conductivity parametrization on the storage head was about $N = 9000$. All computations were performed with a nonoptimized MATLAB implementation on a laptop with 16\,GB RAM and an Intel CPU having clock speed 3.5\,GHz.

\subsubsection{Test 1: one insulating and one conductive inclusion}
The left-hand image of Figure~\ref{fig:target_head1} illustrates the first target head with the misplaced electrodes on its boundary. The azimuthal and polar angles of the electrodes were drawn as explained in Section \ref{sec:data} with $\varsigma_\theta = \varsigma_\phi = 0.03$ radians in \eqref{eq:Gamma_thph}. The right-hand image of Figure~\ref{fig:target_head1} shows the geometric initial guess for Algorithm~\ref{alg:1},~i.e.,~the average head shape together with electrodes at their intended positions organized into three belts with constant polar angles. Counting upwards from the bottom belt, there are altogether $M = 16+10+6 = 32$ electrodes and their common radius is $R=0.75$\,cm. The mean and standard deviation used for drawing the relative contact conductances $\zeta_{\rm trgt}\in \R^M$ were $\tilde{\tau}_\zeta=100\,{\rm S}/{\rm m}^2$ and $\tilde{\varsigma}_\zeta = 10\,{\rm S}/{\rm m}^2$, respectively. Observe that such contacts are comparable to those documented in \cite{Cheng89,Darde13b,Heikkinen02,Hyvonen17}, if the quality of contact is measured by the ratio between the contact conductance and the conductivity level of the imaged body. The relative noise level in \eqref{eq:noisecov} was set to $\varsigma_\eta = 10^{-3}$ (cf.~\cite{Darde13b}), and the pointwise standard deviation and correlation length for the prior in \eqref{eq:prior} were chosen to be $\varsigma_\sigma = 0.1 \, {\rm S}/{\rm m}$ and $l = 3.3 \, {\rm cm}$, respectively. 

\begin{figure}[t]
\center{
  {\includegraphics[width=6cm]{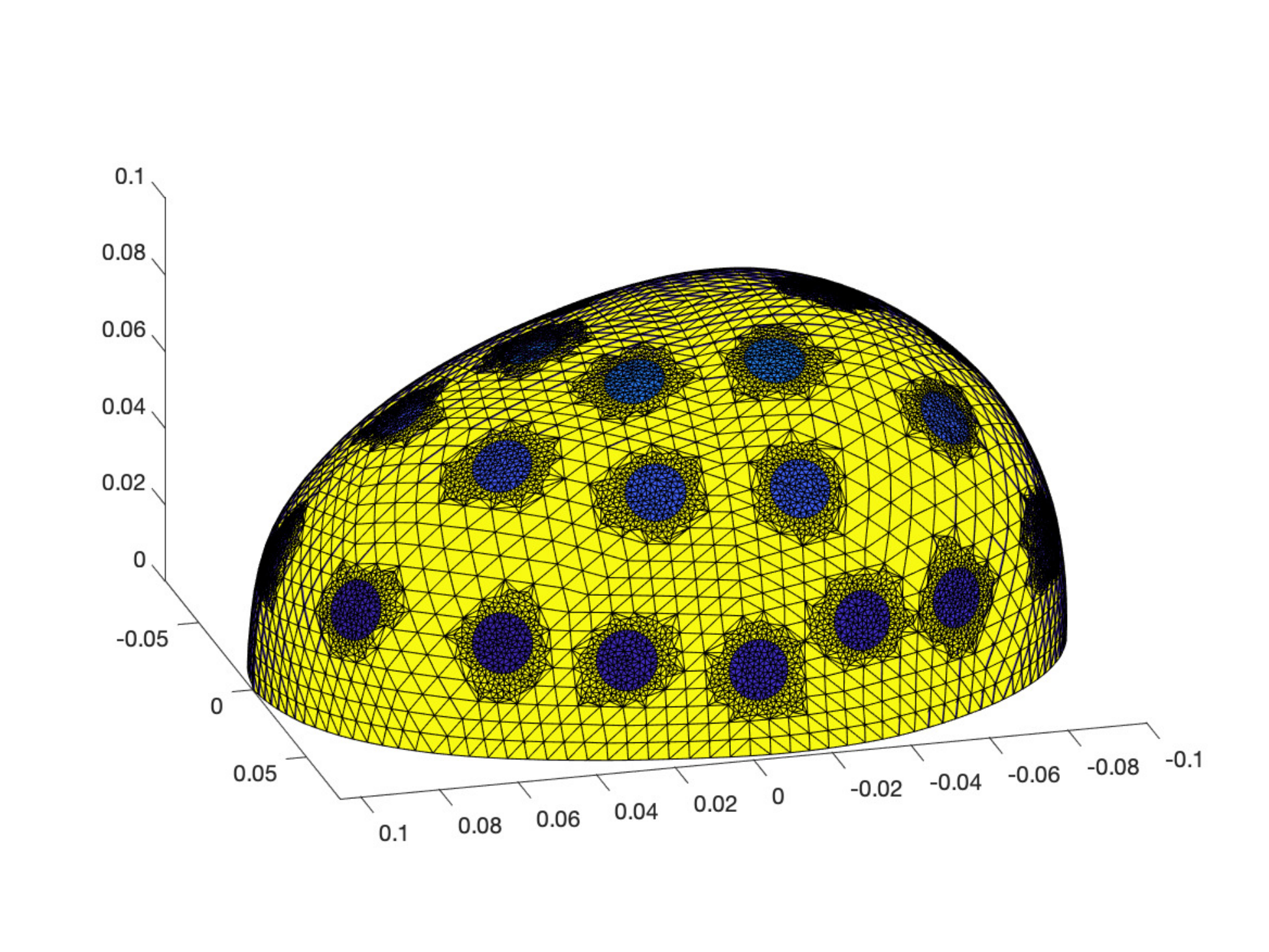}}
{\includegraphics[width=6cm]{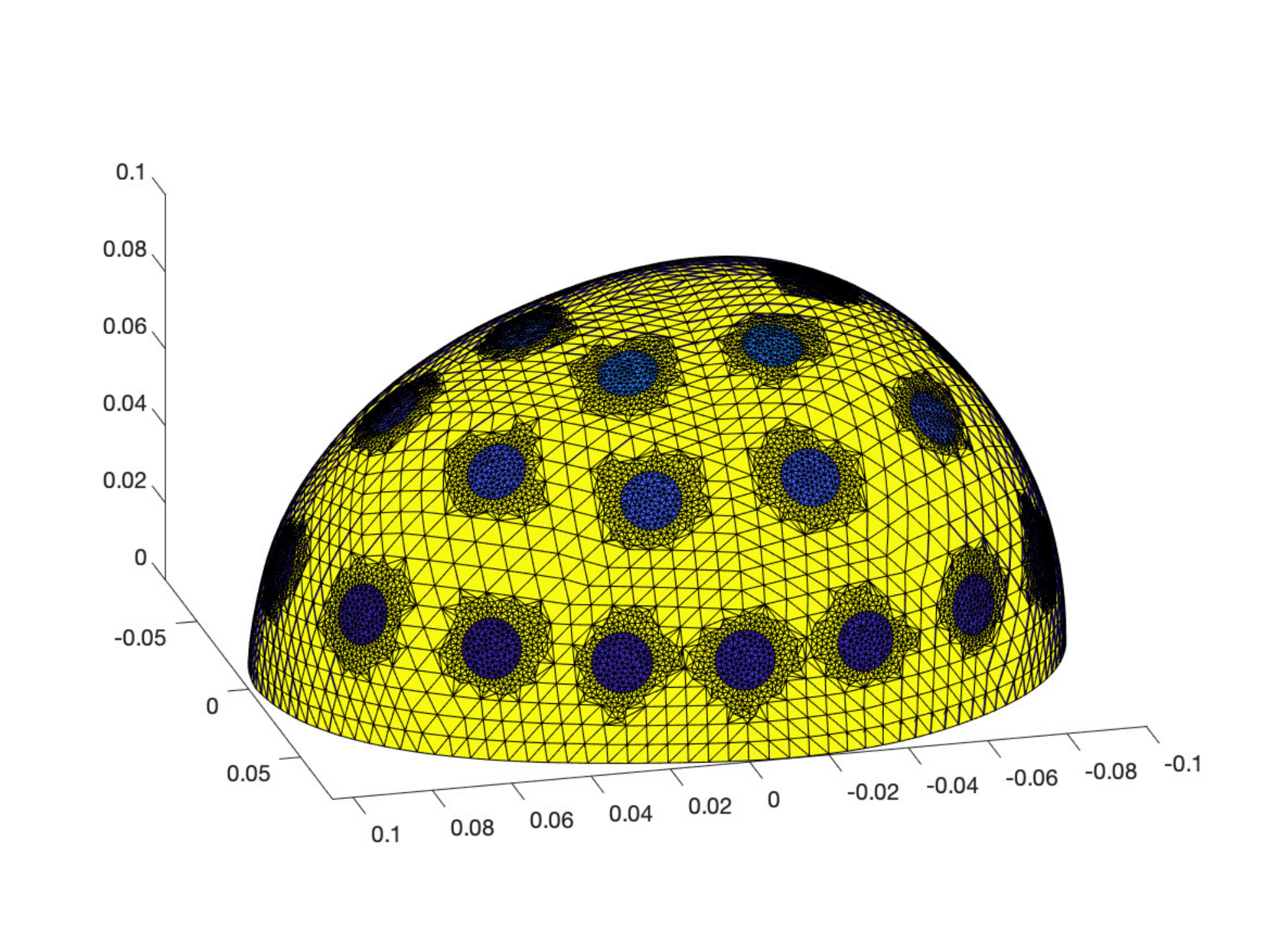}}
}
  \caption{Left: the first target head with the misplaced electrodes. Right: the corresponding initial guess for Algorithm~1 with electrodes at the intended positions. The unit of length is meter.}
  \label{fig:target_head1}
\end{figure}

The target conductivity consists of two balls with constant conductivities $\sigma_1 = 2\,{\rm S}/{\rm m}$ and $\sigma_2 = 0.02\,{\rm S}/{\rm m}$ embedded in a homogeneous background characterized by $\sigma_0 = 0.2 \, {\rm S}/{\rm m}$. Take note that such a background level is in line with conductivities reported for white and gray matter in the human brain \cite{Gabriel96,Latikka01}.
The center and radius of the conductive inclusion are $y_1 = (-2,-2,6)$\,cm and $r_1= 1.5$\,cm, respectively, whereas the insulating inclusion has radius $r_2 = 2$\,cm and it is centered around $y_2 = (3,3,3)$\,cm.

\begin{figure}[t]
  \center{
    {\includegraphics[width=6cm]{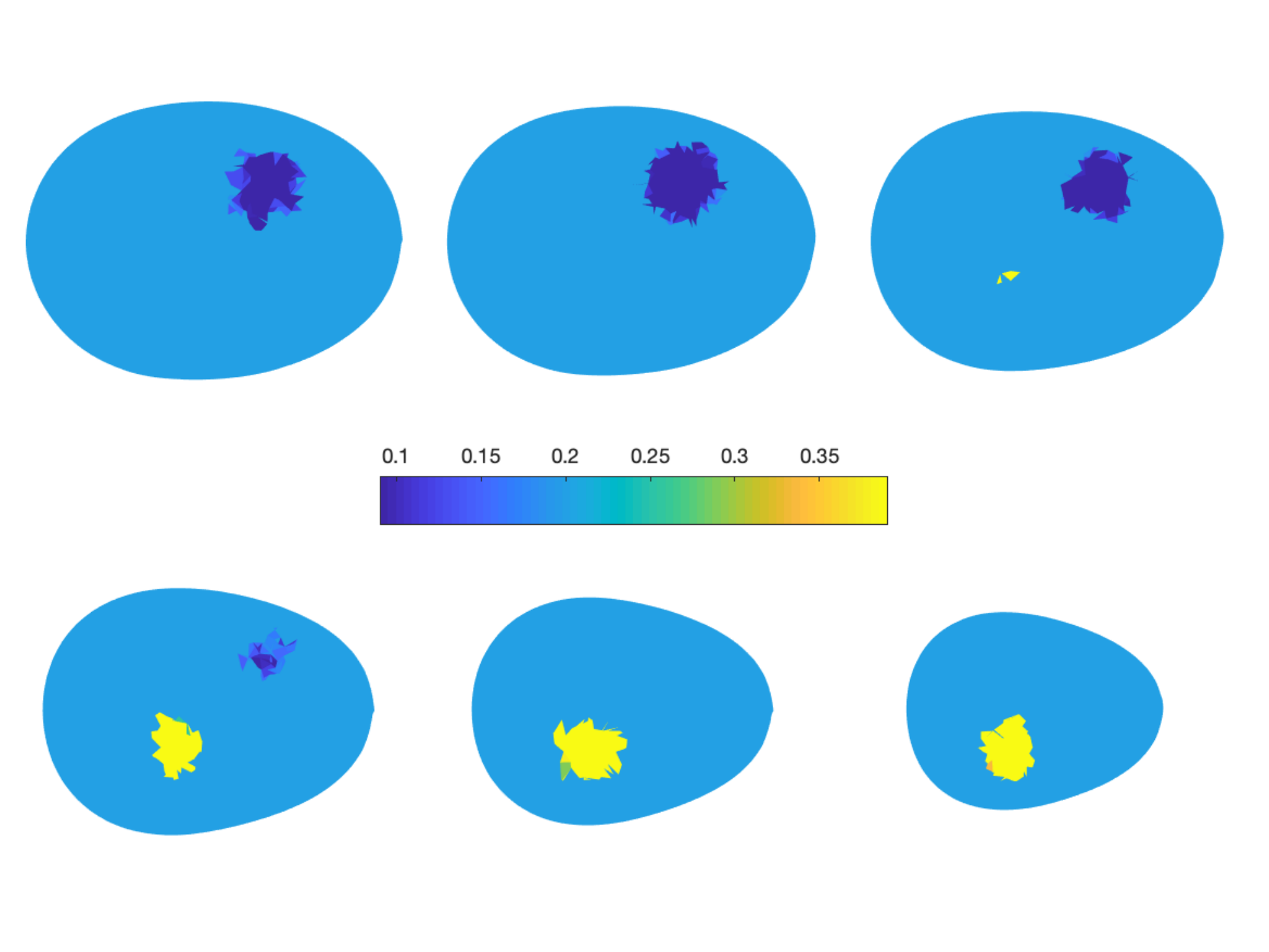}}
    \quad
{\includegraphics[width=6cm]{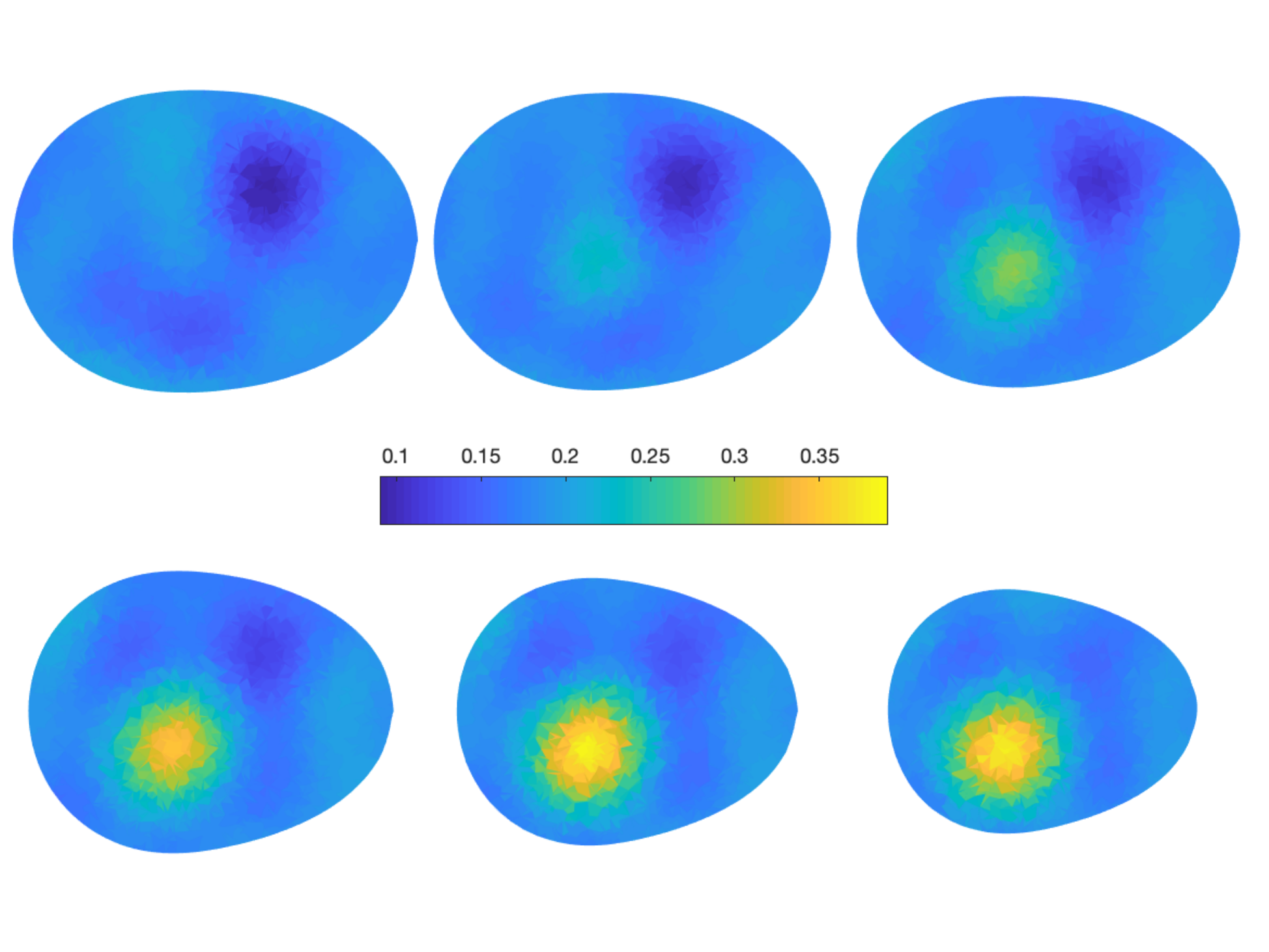}}
  }

\smallskip
  
  \center{
    {\includegraphics[width=6cm]{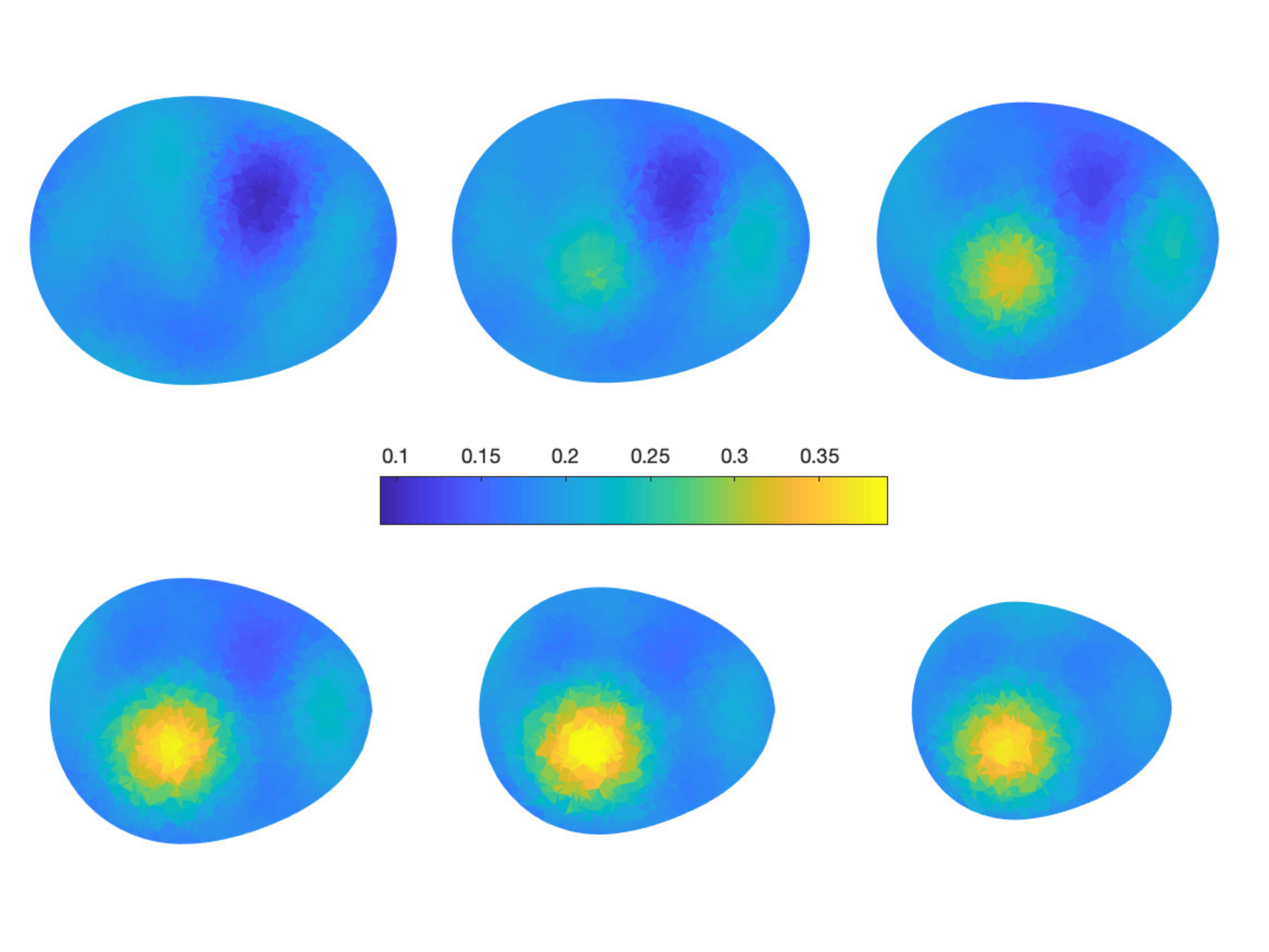}}
    \quad
{\includegraphics[width=6cm]{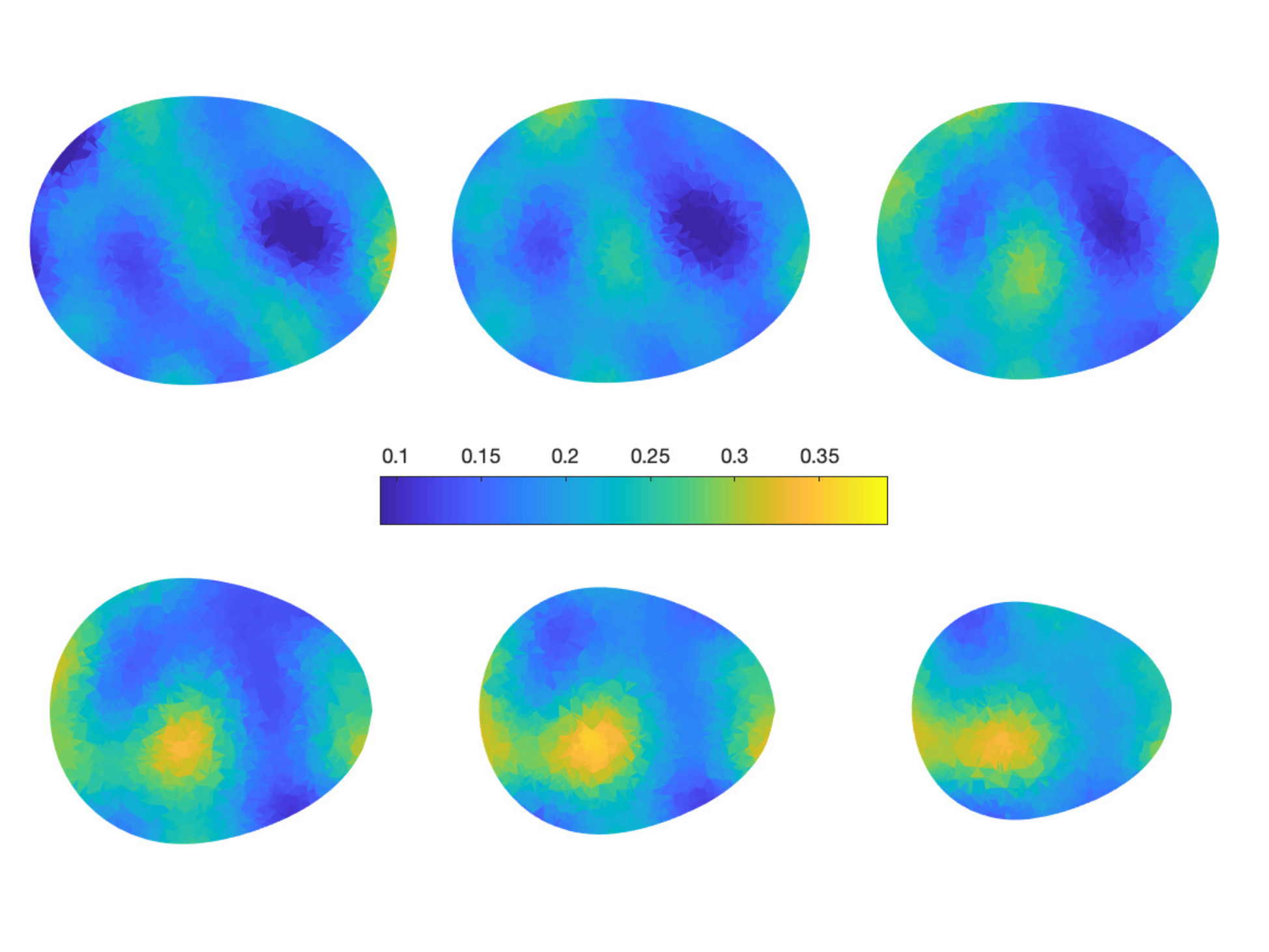}}
  }
  \caption{Six horizontal cross-sections of the target conductivity and of three reconstructions for test~1; the slices are at levels $2$, $3$, $4$, $5$, $6$ and $7$\,cm. Top left: target conductivity (the colormap does not cover the whole dynamic range). Top right: reconstruction by the complete Algorithm~\ref{alg:1}. Bottom left: reconstruction when $\alpha = 0$ is fixed. Bottom right: reconstruction when $\alpha = 0$, $\theta = \bar{\theta}$ and $\phi = \bar{\phi}$ are fixed.}
  \label{fig:reco1_1}
\end{figure}

Figure~\ref{fig:reco1_1} compares the target conductivity with three reconstructions. The top right reconstruction
corresponds to the complete Algorithm~\ref{alg:1}; the Gauss--Newton iteration converged in six rounds and about twenty minutes. The bottom left image
resulted from not taking the mismodeling of the head shape into account, that is, the reconstruction was formed by setting $\alpha = 0$ in Algorithm~\ref{alg:1} and only estimating the other four unknowns. As the computation of the derivatives with respect to the shape parameters is the most time consuming part of Algorithm~\ref{alg:1}, forming this second reconstruction took only a few minutes, although the Gauss--Newton iteration lasted for one extra round. Finally, the bottom right reconstruction corresponds to altogether ignoring geometric inaccuracies and fixing $\alpha = 0$, $\theta = \bar{\theta}$ and $\phi = \bar{\phi}$ in Algorithm~\ref{alg:1}. Computing this third reconstruction in seven Gauss--Newton iterations was somewhat faster than the second one, although the difference was not significant.  

The reconstruction corresponding to the complete Algorithm~\ref{alg:1} in the top right image of Figure~\ref{fig:reco1_1} and the one ignoring (only) the mismodeled head shape in the bottom left image are comparable in quality. One can approximately locate the two inhomogeneities in each of these images, although the reconstructions are significantly blurred and unable to capture the full dynamic range of the target conductivity (due to the employed smoothness prior). The reconstruction corresponding to the complete algorithm carries, arguably, less significant artifacts close to the object boundaries, but it also overestimates the size of the imaged head. This is a recurring phenomenon: quite often the complete algorithm is unable to estimate the overall size of the head, even if it is otherwise able to reproduce the head shape. Our hypothesis is that this phenomenon can be explained by the scaling of the object and the average level of the contact conductances having a similar effect on the measurements. Although the reconstruction ignoring all uncertainties in the measurement setup in the bottom right image of Figure~\ref{fig:reco1_1} can be considered tolerable, it suffers from severe artifacts and is certainly inferior to the other two.

\begin{figure}[t]

  \center{
    {\includegraphics[width=6cm]{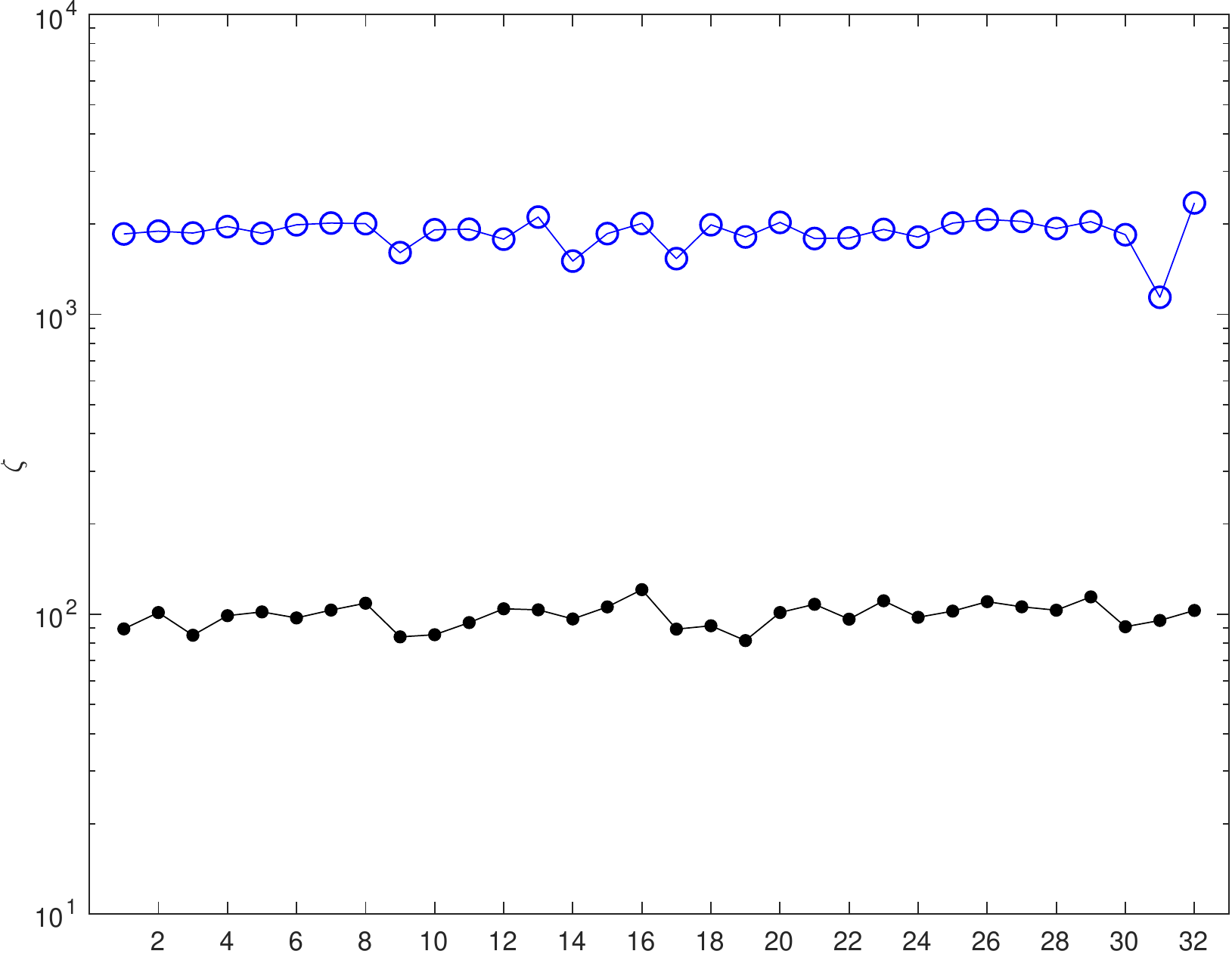}}
    \qquad
{\includegraphics[width=6cm]{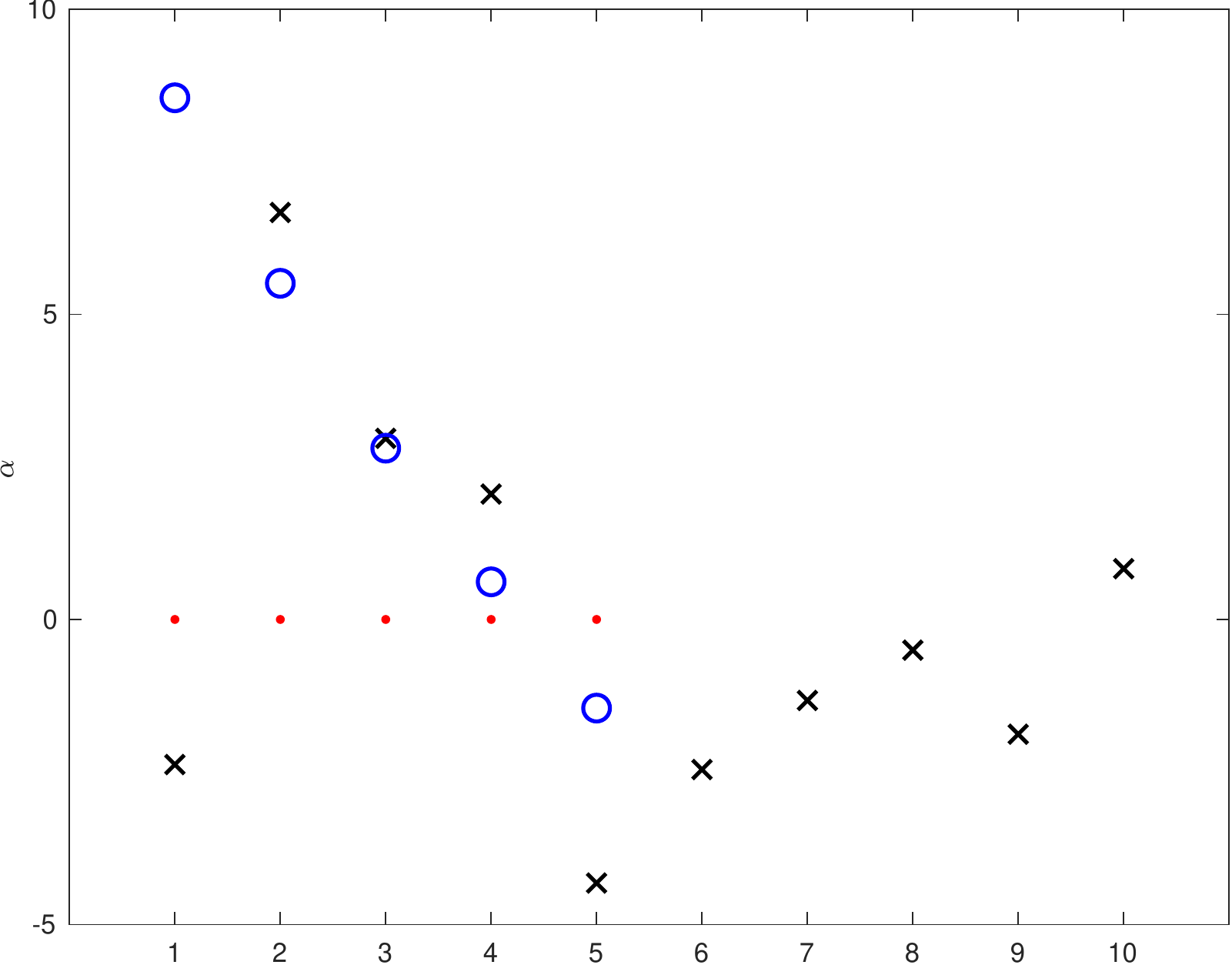}}
  }

  \bigskip
  \bigskip
  
  \center{
    {\includegraphics[width=6cm]{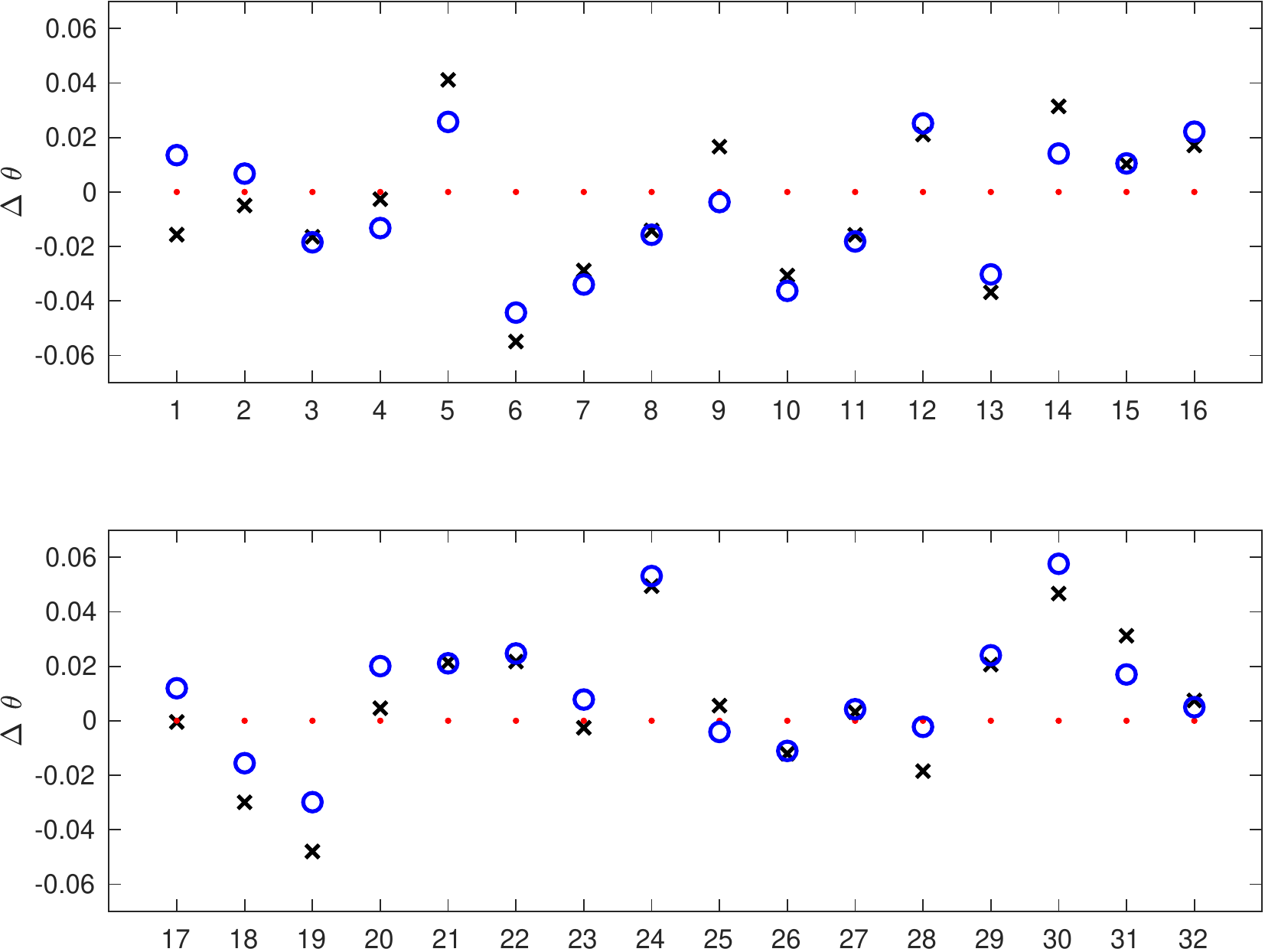}}
    \qquad
{\includegraphics[width=6cm]{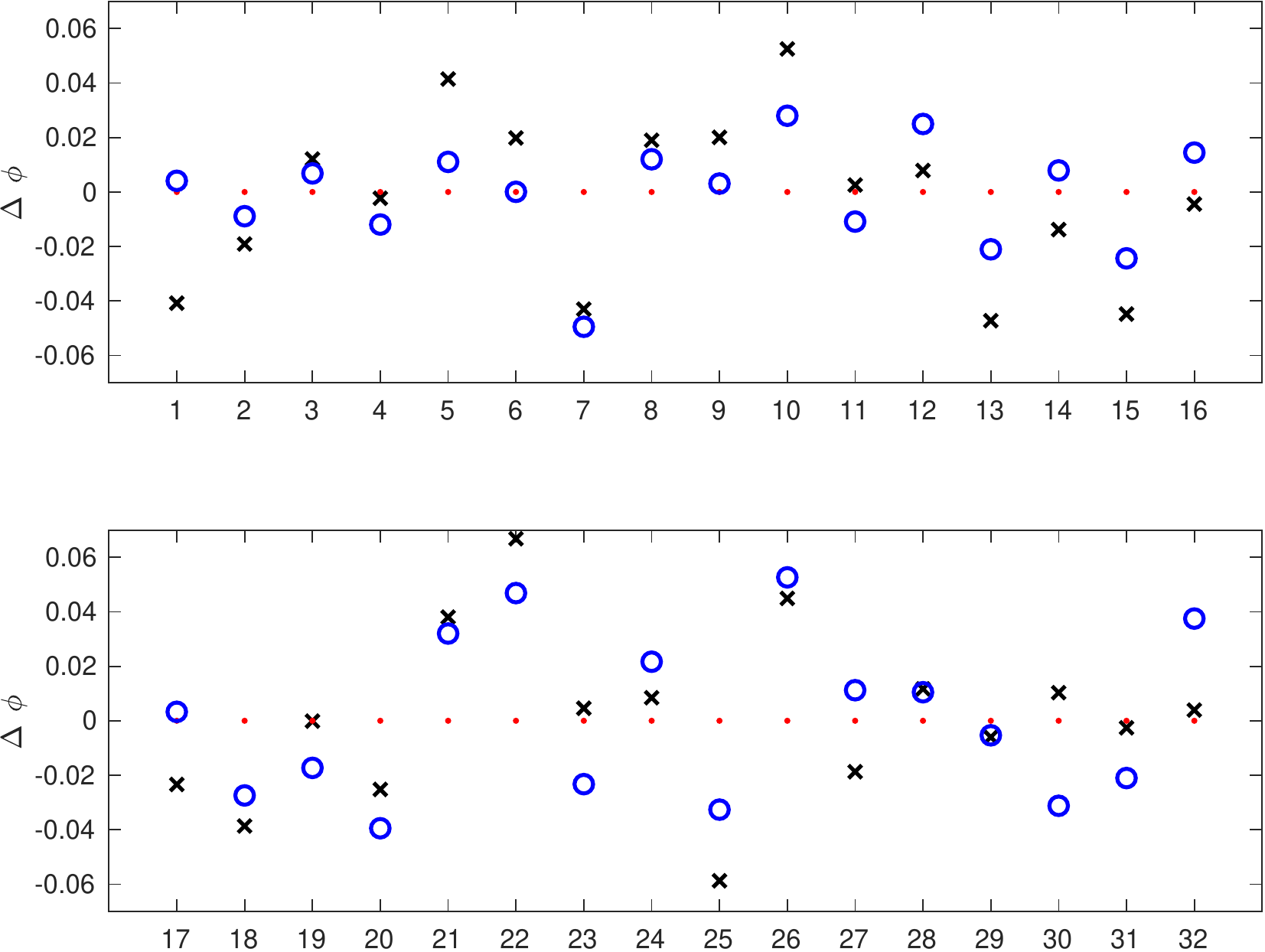}}
  }
  \caption{Target parameters and those reconstructed by the complete Algorithm~\ref{alg:1} for test~1. The electrodes are numbered starting from the frontal electrode on the bottom belt and circulating around the head in the positive direction when viewed from above. Top left: $\zeta_*$ (circles) and $\zeta_{\rm trgt}$ (dots) on a logarithmic scale. Top right: $\alpha_*\in \R^5$ (circles), $\alpha_{\rm trgt}\in \R^{10}$ (crosses) and the expected values/initial guesses (dots).  Bottom left: $\theta_*$ (circles)  and $\theta_{\rm trgt}$ (crosses). Bottom right: $\phi_*$ (circles) and $\phi_{\rm trgt}$ (crosses). The electrode angles are depicted relative to their intended positions.}
  \label{fig:reco1_2}
\end{figure}

Figure~\ref{fig:reco1_2} compares the target parameter values $\alpha_{\rm trgt}$, $\zeta_{\rm trgt}$, $\theta_{\rm trgt}$ and $\phi_{\rm trgt}$ with those reconstructed by the complete Algorithm~\ref{alg:1}. Recall that the algorithm only reconstructs five shape parameters,~i.e.~$\alpha_* \in \R^5$, while the target head shape is described by $\alpha_{\rm trgt} \in \R^{10}$. Moreover, the comparison between the reconstructed $\zeta_*$ and the corresponding target $\zeta_{\rm trgt}$ is not straightforward as they correspond to different contact shapes in \eqref{eq:zeta_param}; see~\cite{Hyvonen17b} for further information on this matter. Apart from $\alpha_1$ that mainly controls the overall size of the object (cf.~Figure~\ref{fig:pricomp}), the other four estimated shape parameters are reasonable, yet by no means completely accurate. Moreover, the algorithm seems to be able to detect general patterns in the polar and azimuthal angles of the electrodes relative to each other, although many of the electrode positions themselves are not reconstructed accurately. The interplay between the effects the electrode positions, the object shape and the contact conductances have on the measurements is altogether nontrivial, which is definitely reflected in the estimated parameter values illustrated in Figure~\ref{fig:reco1_2}. Be that as it may, the main objective of Algorithm~\ref{alg:1} is achieved because the conductivity reconstruction in the top right image of Figure~\ref{fig:reco1_1} can be considered informative.

\subsubsection{Test 2: one inclusion with more uncertain electrode positions}
The second test considers the same intended electrode configuration as the first one, but this time around we increased the uncertainty about the electrode positions and the contact conductances by choosing $\varsigma_\theta = \varsigma_\phi = 0.04$, $\tilde{\tau}_\zeta=100\, {\rm S}/{\rm m}^2$ and $\tilde{\varsigma}_\zeta = 20\, {\rm S}/{\rm m}^2$ when drawing the target values $\theta_{\rm trgt}$, $\phi_{\rm trgt}$ and $\zeta_{\rm trgt}$. The target shape parameters $\alpha_{\rm trgt}$ were redrawn to obtain a new head shape. The to-be-reconstructed conductivity still has a constant background described by $\sigma_0 = 0.2\, {\rm S}/{\rm m}^2$ with an embedded, insulating cylindrical inclusion with constant conductivity level $\sigma = 0.02 \, {\rm S}/{\rm m}$, height $H = 7$\,cm, and radius $r = 1.5$\,cm centered at $(1,3,3)$\,cm. The central axis of the cylinder is oriented orthogonal to the coronal plane. The prior for the conductivity and the noise model are the same as in the first test.

\begin{figure}[t]
  \center{
  {\includegraphics[width=6cm]{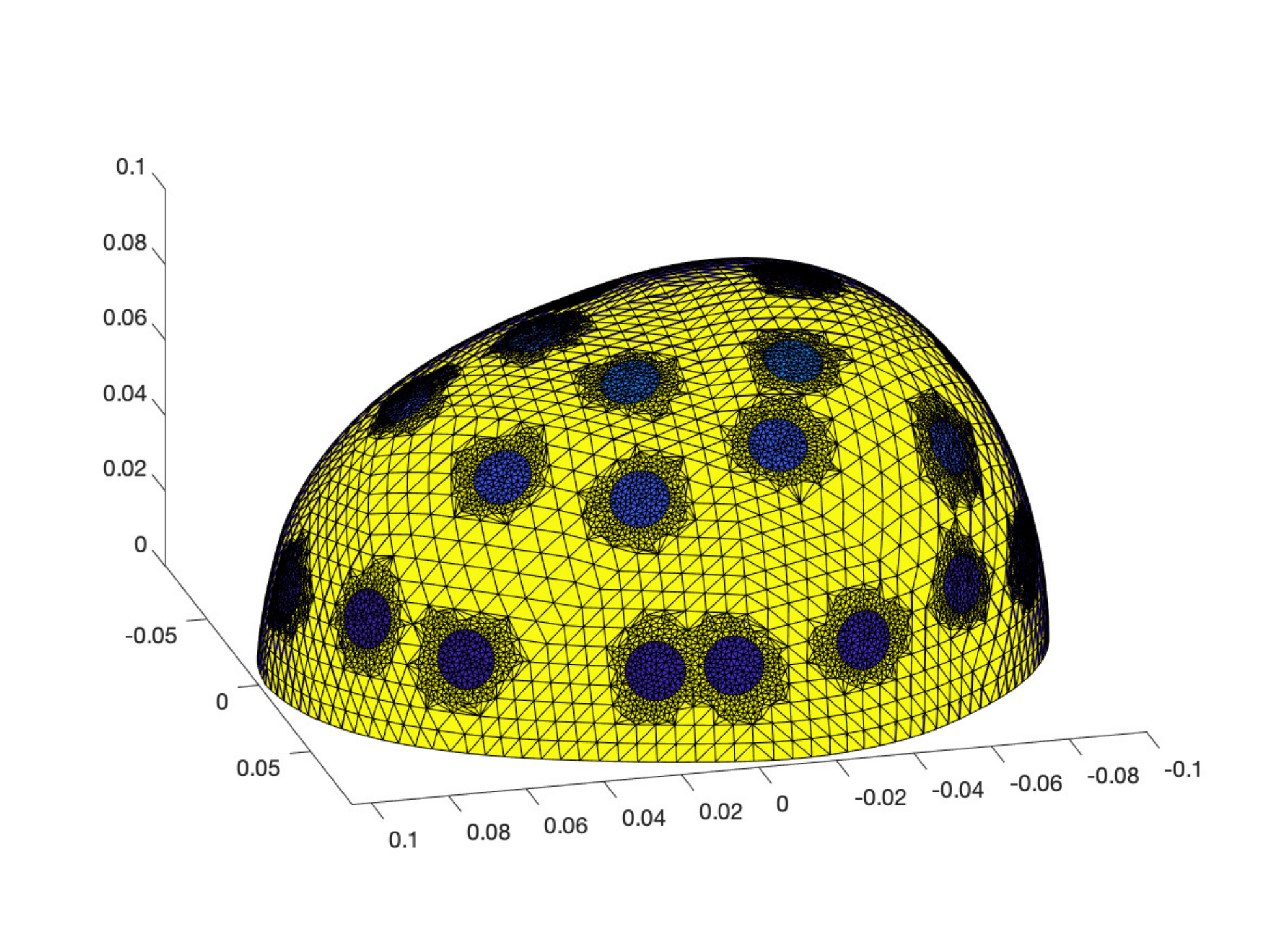}}
{\includegraphics[width=6cm]{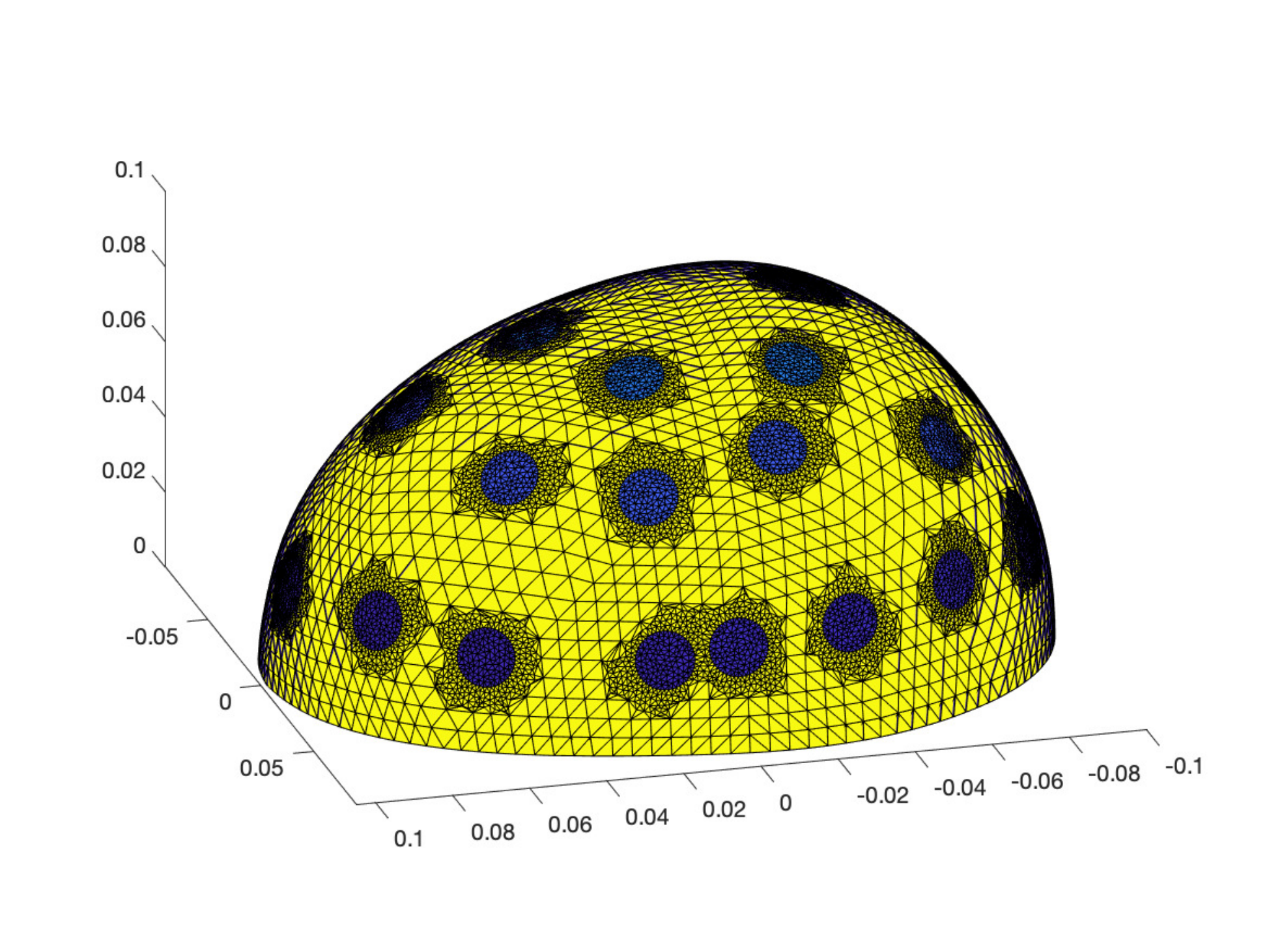}}
}
  \caption{Left: the second target head with the misplaced electrodes. Right: the reconstructed geometry by Algorithm~\ref{alg:1}. The unit of length is meter.}
  \label{fig:reco2_1}
\end{figure}

\begin{figure}[t]

  \center{
    {\includegraphics[width=6cm]{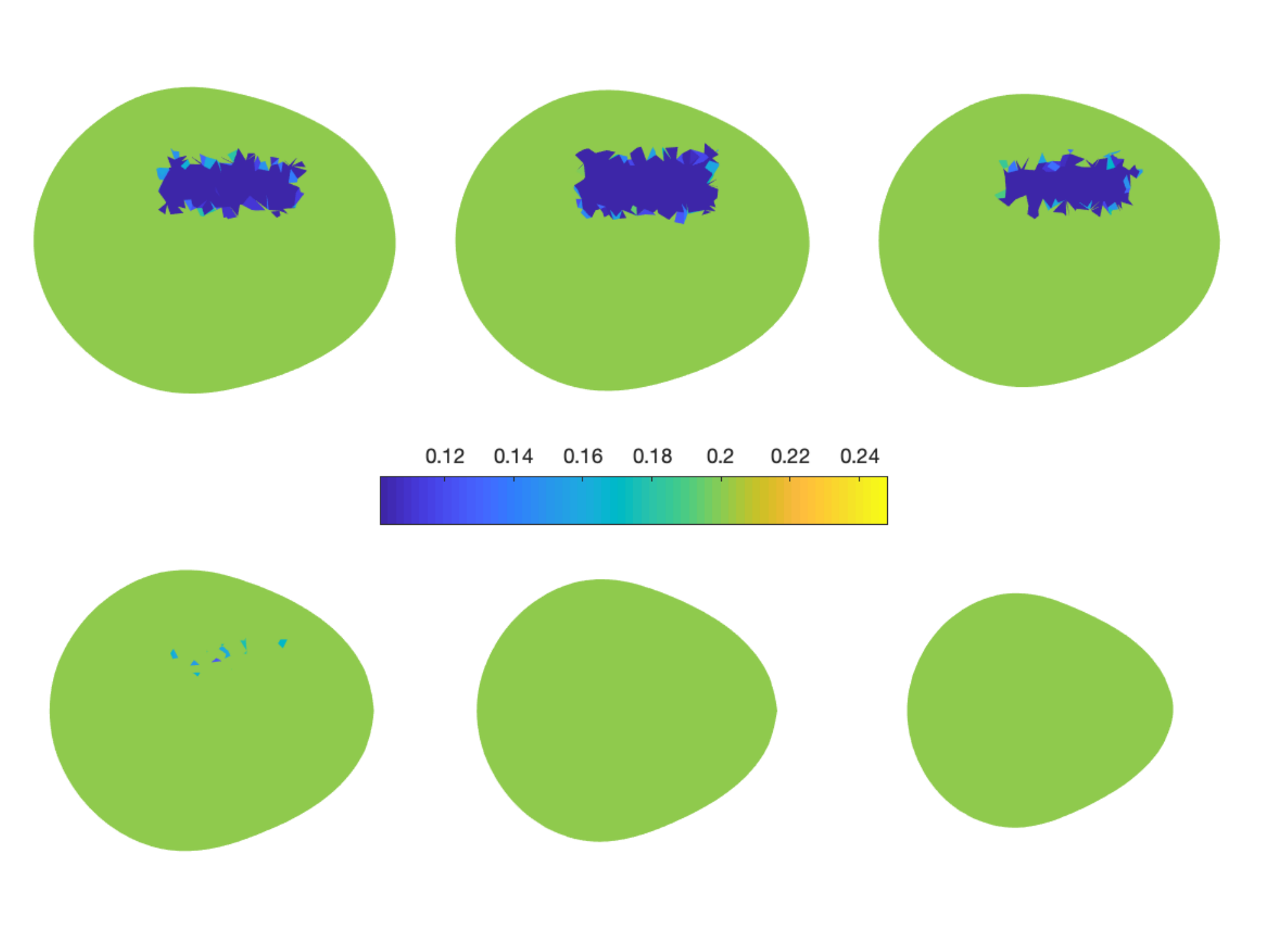}}
    \quad
{\includegraphics[width=6cm]{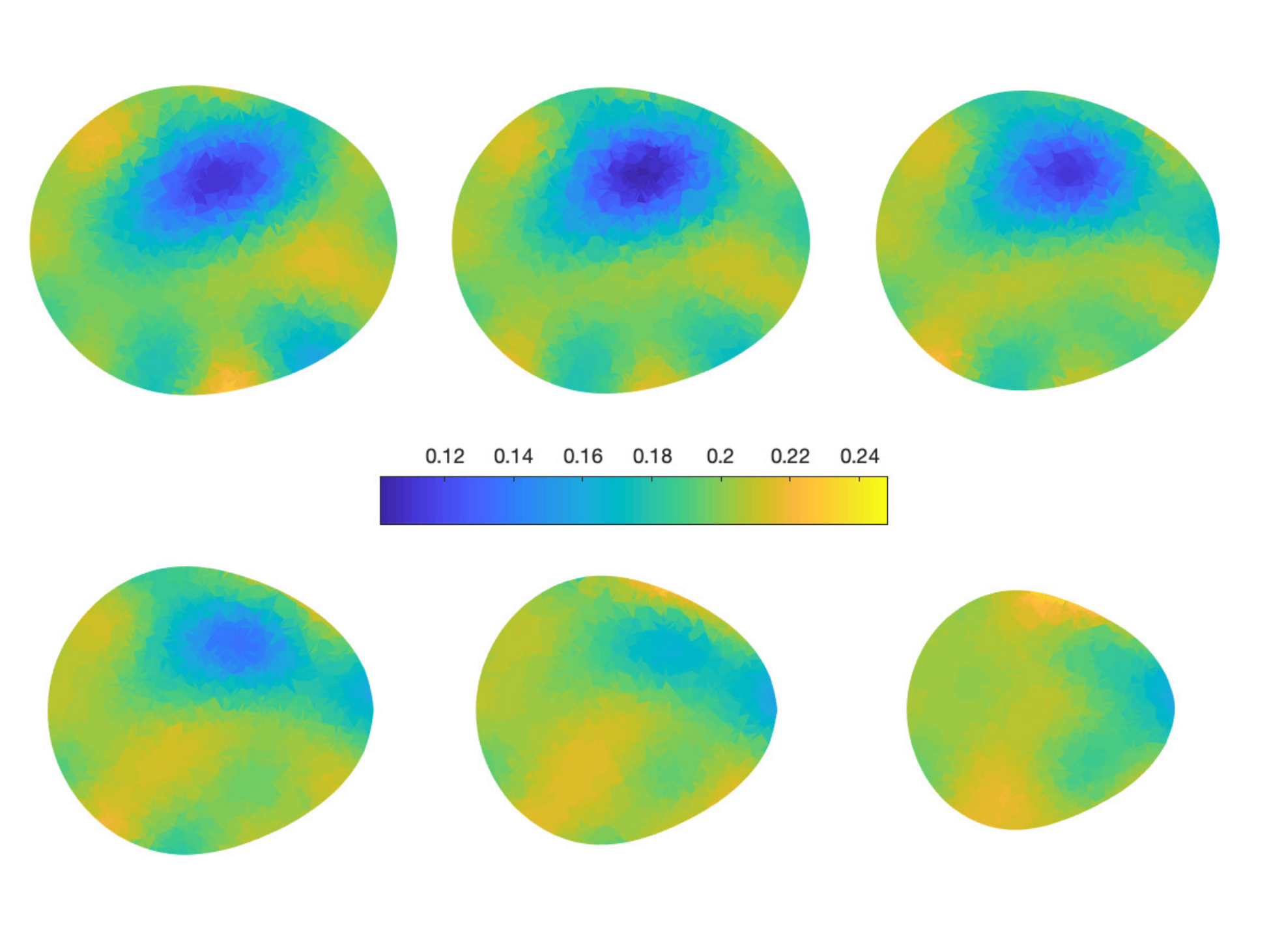}}
  }

  \smallskip
  
  \center{
    {\includegraphics[width=6cm]{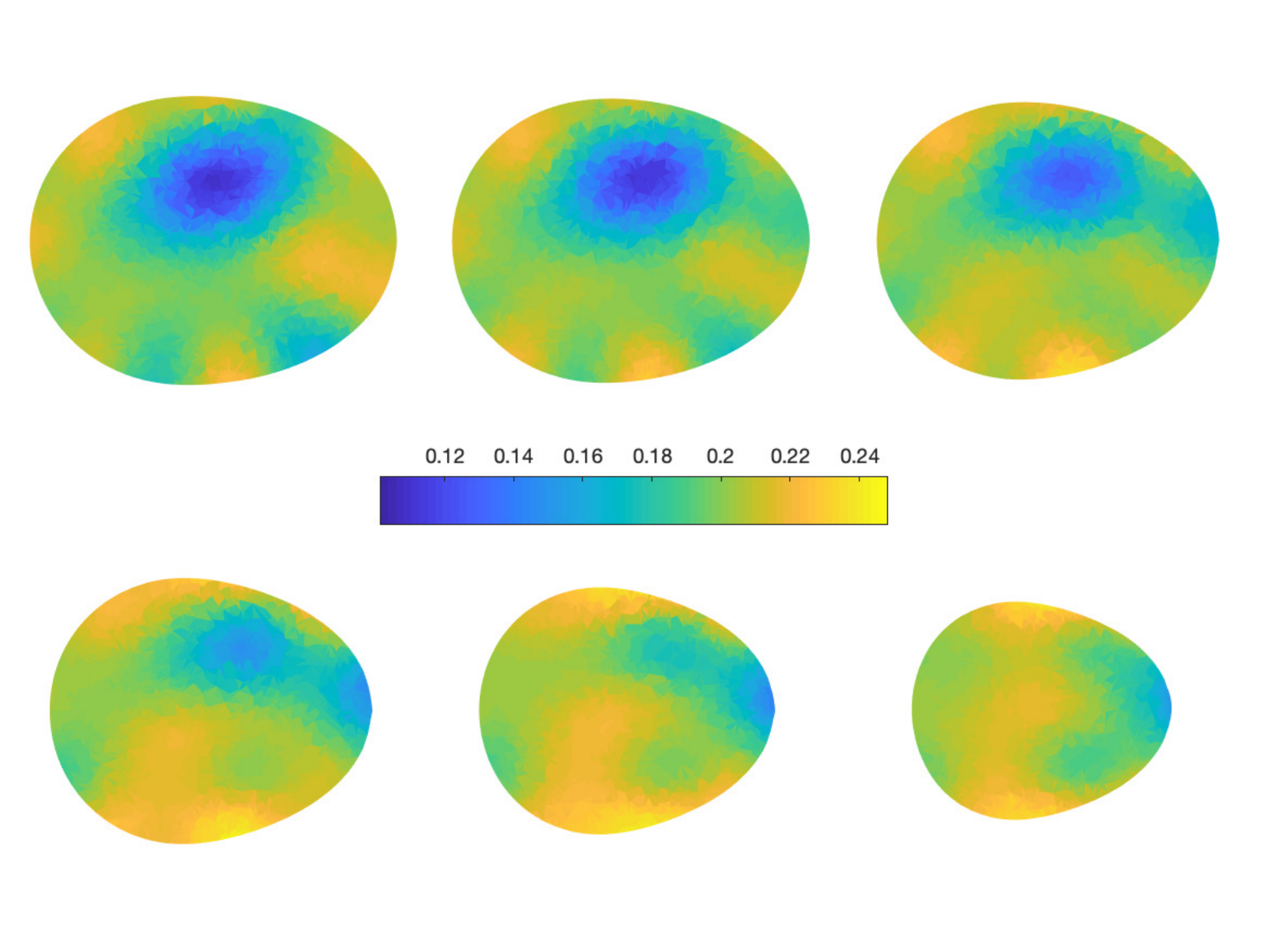}}
    \quad
{\includegraphics[width=6cm]{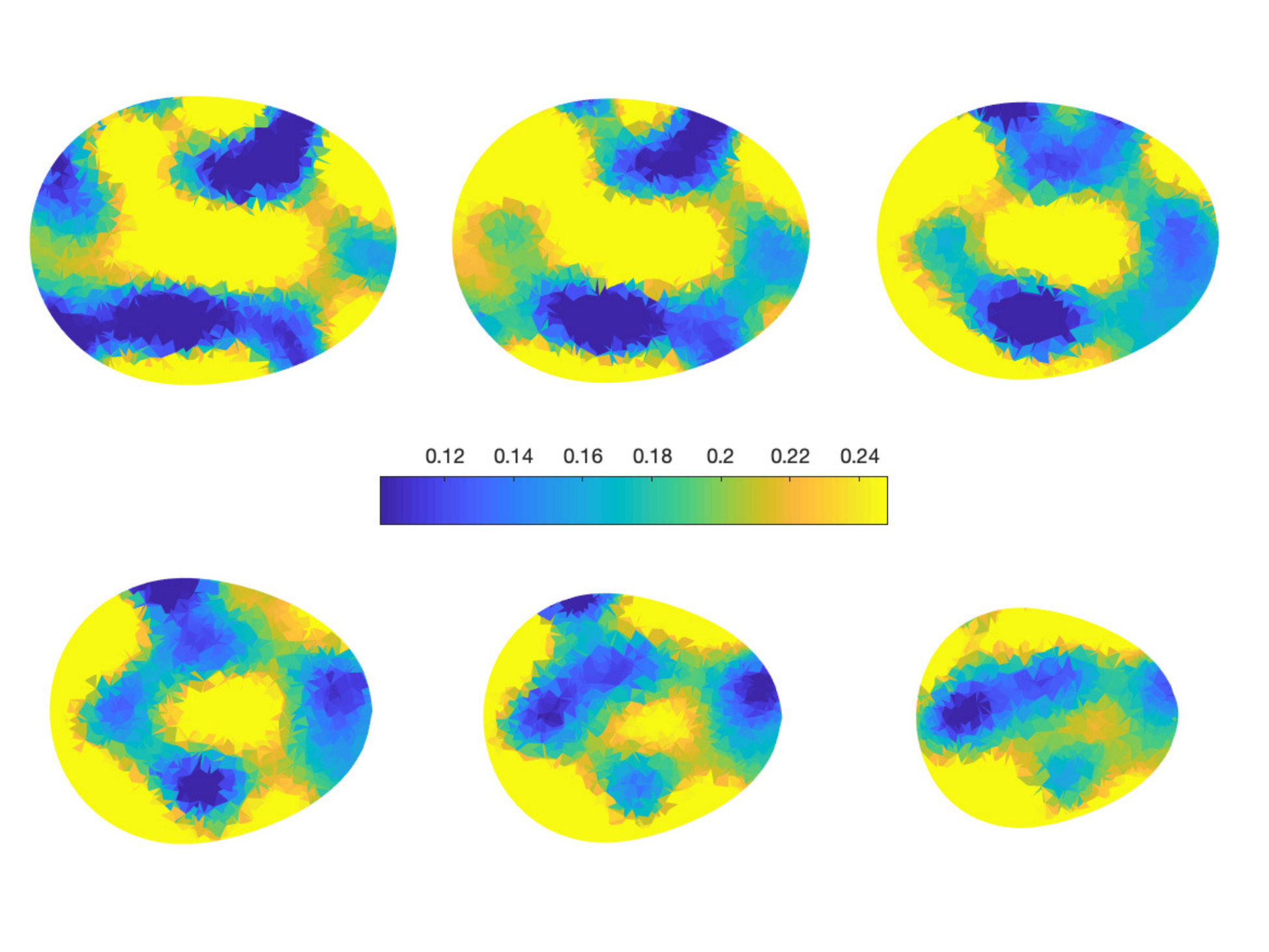}}
  }
\caption{Six horizontal cross-sections of the target conductivity and of three reconstructions for test~2; the slices are at levels $2$, $3$, $4$, $5$, $6$ and $7$\,cm. Top left: target conductivity (the colormap does not cover the whole dynamic range). Top right: reconstruction by the complete Algorithm~\ref{alg:1}. Bottom left: reconstruction when $\alpha = 0$ is fixed. Bottom right: reconstruction when $\alpha = 0$, $\theta = \bar{\theta}$ and $\phi = \bar{\phi}$ are fixed.}
  \label{fig:reco2_2}
\end{figure}

The target geometry is compared with the one reconstructed by the complete Algorithm~\ref{alg:1} in Figure~\ref{fig:reco2_1}; this time the Gauss--Newton iteration converged in eight rounds and about 25 minutes.  The shape parameter vector converged to $\alpha_* = [2.47, -5.72, 1.82, 1.76, 0.32]^{\rm T}$, whereas the first five components of $\alpha_{\rm trgt}$ are $1.62$, $-5.50$, $0.24$, $2.57$ and $-3.17$. Even though the estimates for the shape parameters are not very accurate, the reconstructed head shape anyway resembles the target head to a certain degree. Even more obviously, the electrode pattern has converged toward the target setup; see the right-hand image of Figure \ref{fig:target_head1} for the geometric initial guess.

Figure~\ref{fig:reco2_2} presents the target conductivity together with three reconstructions. As in Figure~\ref{fig:reco1_1} of the first numerical test, the reconstructions correspond to the complete Algorithm~\ref{alg:1}, to ignoring the mismodeled head shape, and to paying no attention to the inaccuracies in either the electrode positions or the head model. This time around, the reconstruction ignoring all inaccuracies in the measurement setting did not converge in twenty Gauss--Newton iterations, leading to the intolerably bad reconstruction of the conductivity in the bottom right image of Figure~\ref{fig:reco2_2}. The other two reconstructions in Figure~\ref{fig:reco2_2} contain useful information about the conductivity inside the target head, although their quality is worse than in test 1, most probably due to the initial higher degree of uncertainty about the electrode positions and the contact conductances. The reconstruction produced by the complete algorithm in the top right image of Figure~\ref{fig:reco2_2} is arguably the best one, separating itself from the one ignoring the inaccuracy in the head model in the bottom left image by exhibiting a more constant background conductivity level. It is also worth noticing the shapes of the cross-sections in the reconstruction corresponding to the complete algorithm clearly mimic those of the target head.

\subsubsection{Test 3: two conductive inclusion and larger electrodes}

\begin{figure}[t]
   \center{
     {\includegraphics[width=6cm]{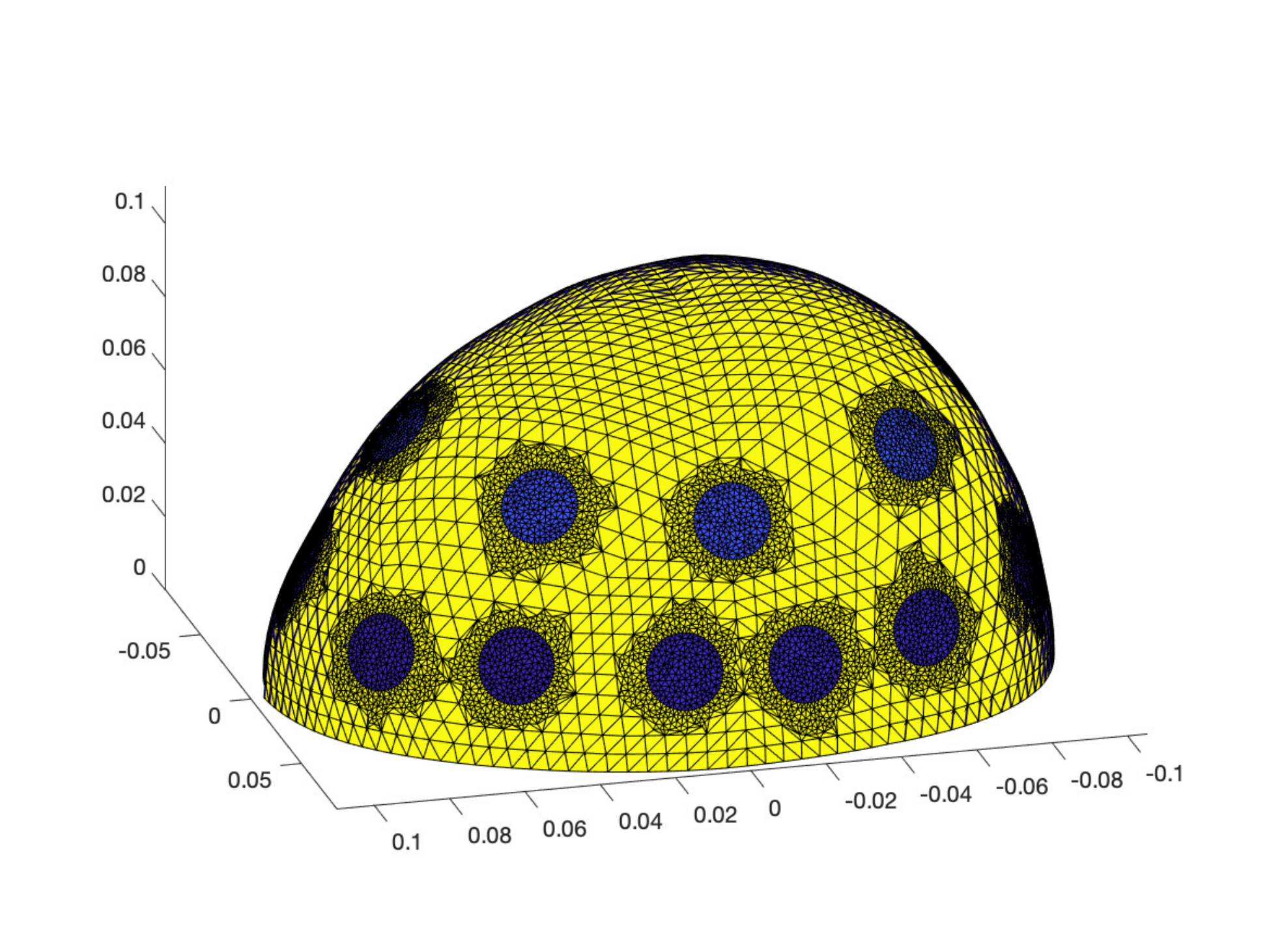}}
         {\includegraphics[width=6cm]{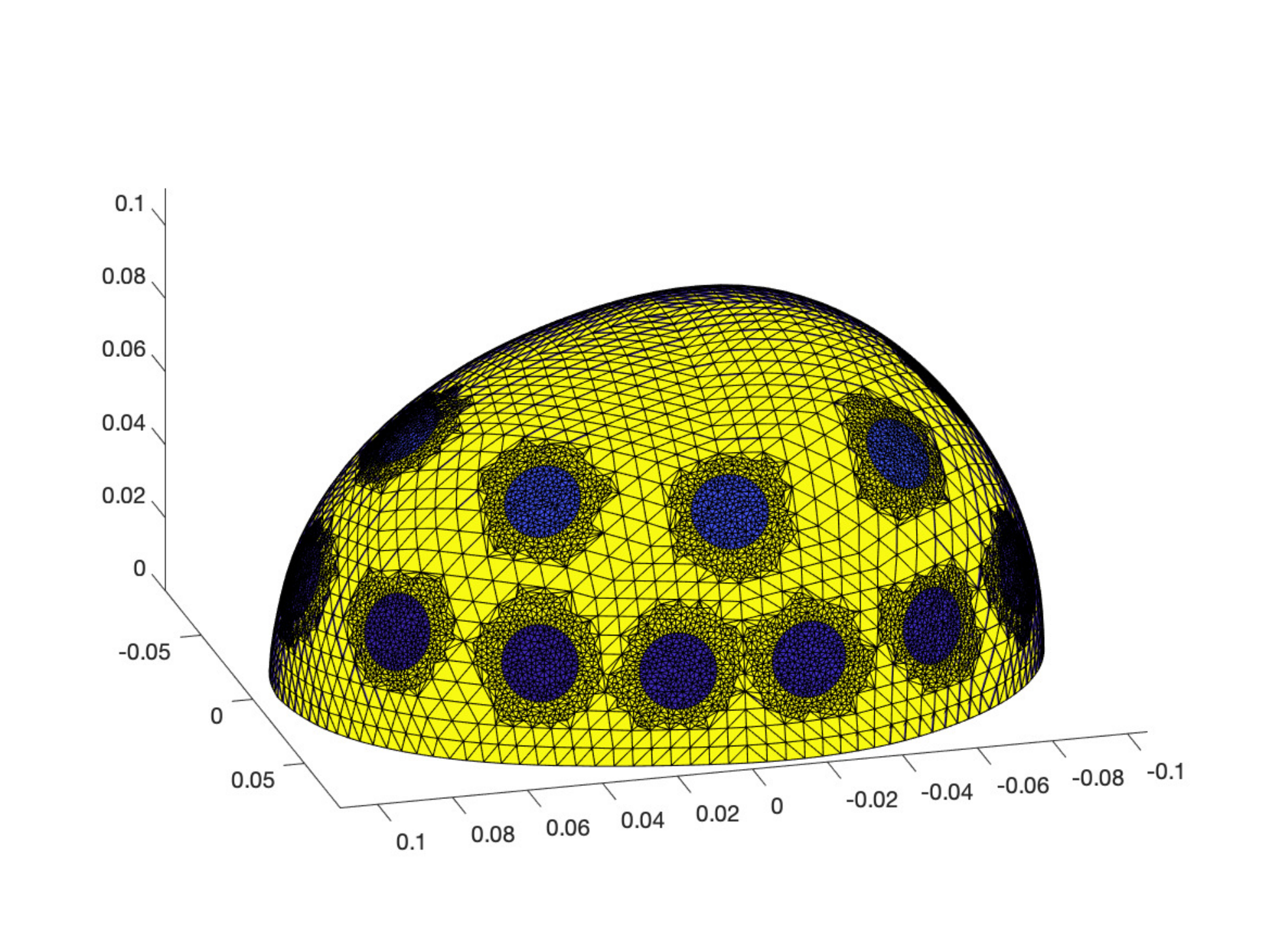}}
}
  \caption{Left: the third target head with the misplaced electrodes. Right: the corresponding initial guess for Algorithm~1 with electrodes at the intended positions. The unit of length is meter.}
  \label{fig:target_head3}
\end{figure}

In our final experiment, we consider an alternative measurement configuration with just two electrode belts and  a lower number of $M =  14 + 8 = 22$ electrodes. The electrode radius is also larger than before, namely $R = 1\,{\rm cm}$. The probability distributions from which the parameters $\zeta_{\rm trgt}$, $\theta_{\rm trgt}$ and $\phi_{\rm trgt}$ were drawn are the same as in the first experiment, apart from the means $\bar{\theta}$ and $\bar{\phi}$ that are determined by the intended electrode positions and a slightly lower level of uncertainty about the electrode position, i.e.~$\varsigma_\theta = \varsigma_\phi = 0.02$. As another novelty, the covariance for the head shapes is multiplied by two: $\alpha_{\rm trgt} \in \R^{10}$ was drawn from $\mathcal{N}(0, 2 \Gamma_\alpha)$ where $\Gamma_\alpha \in \R^{10 \times 10}$ is defined by \eqref{eq:Gamma_alpha}. This corresponds to allowing larger geometric variations in the head shape compared to what the library of \cite{Lee16} predicts. We assume to be aware of this fact and also replace $\Gamma_\alpha \in \R^{5 \times 5}$ by $2 \Gamma_\alpha$ in the Tikhonov functional \eqref{eq:tikhfun} and Algorithm~\ref{alg:1}. The noise model and the prior for the conductivity remain the same as in the preceding two tests. The target conductivity is characterized by two balls with a common constant conductivity level $\sigma_1 = \sigma_2 = 2\,\,{\rm S}/{\rm m}$, a common radius $r_1 = r_2 = 1.5$\,cm and center points $y_1=(-5,0,5)$\,cm and $y_2 = (4,-2, 3)$\,cm, respectively. The background conductivity level is still $\sigma_0 \equiv 0.2 \,{\rm S}/{\rm m}$.

The target head with the misplaced electrodes and the initial geometric guess for Algorithm~\ref{alg:1} are visualized in Figure~\ref{fig:target_head3}. Figure~\ref{fig:reco3_1}, which is organized in the same way as Figures~\ref{fig:reco1_1} and \ref{fig:reco2_2} above, visualizes the target conductivity and three reconstructions. The complete algorithm converged after eight rounds of Gauss--Newton iteration in about 25 minutes; the other two reconstructions were once again obtained considerably faster. The main conclusions are the same as in the previous two tests: The complete Algorithm~\ref{alg:1} and the one ignoring (only) the uncertainties in the head shape produce comparable reconstructions, with the former being somewhat more accurate and also being able to mimic the main characteristics of the shape of the target head. On the other hand, the reconstruction that also ignores the inaccuracies in the electrode positions is contaminated by more severe artifacts. 

The geometric target parameters $\alpha_{\rm trgt}$, $\theta_{\rm trgt}$ and $\phi_{\rm trgt}$ are compared with the ones reconstructed by the complete Algorithm~\ref{alg:1} in Figure~\ref{fig:reco3_2}. The reconstructed parameters once again exhibit similar patterns as the target values, but the correspondence is not perfect. It is worth mentioning that we also ran the complete algorithm with $\tilde{n} = 10$, i.e., with the number of the reconstructed and target shape parameters being the same, but this had essentially no effect on the reconstructions presented in Figures~\ref{fig:reco3_1} and \ref{fig:reco3_2}.

\begin{figure}[t]
    \center{
    {\includegraphics[width=6.3cm]{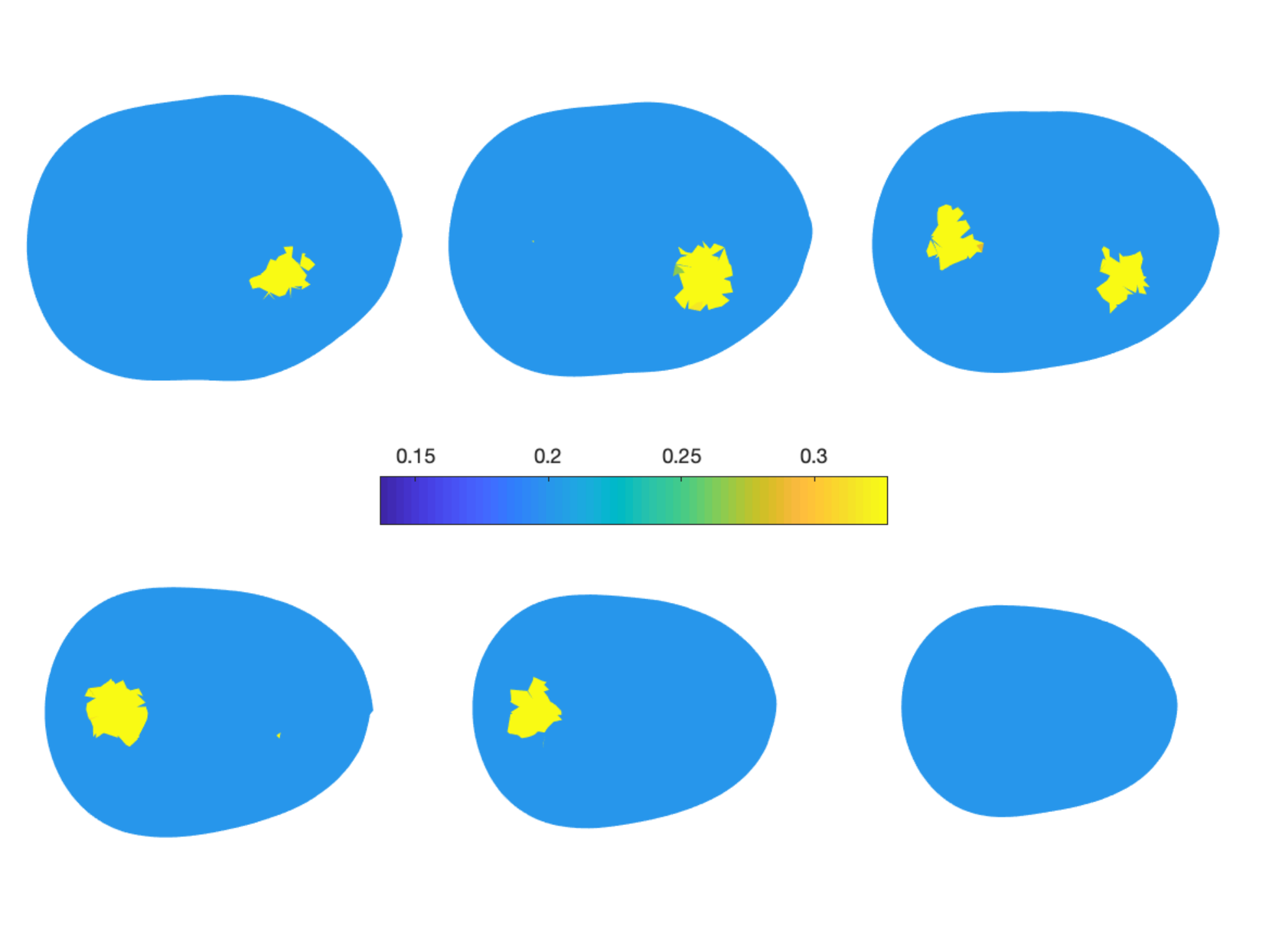}}
    \ \
{\includegraphics[width=6.3cm]{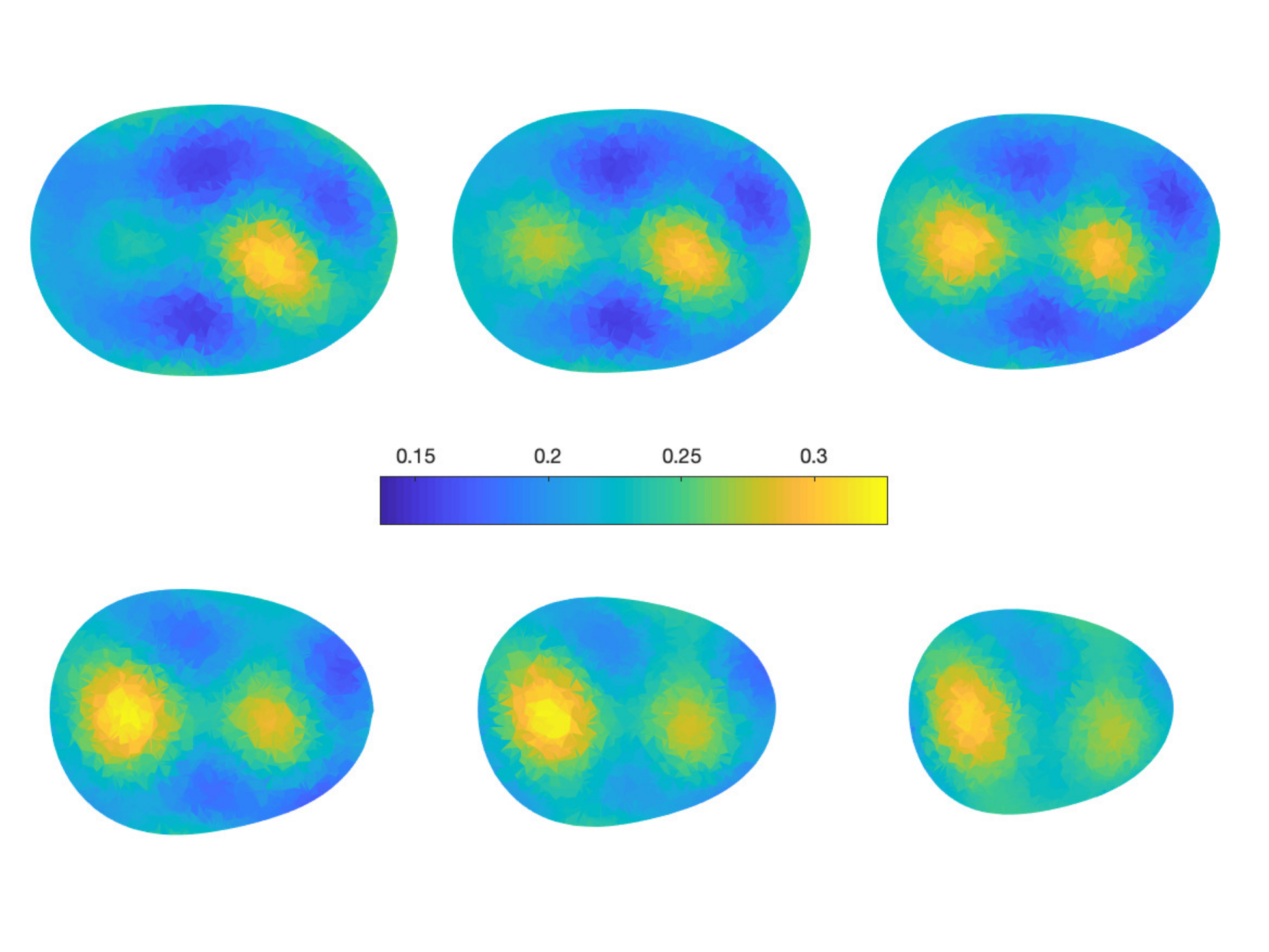}}
  }

  \smallskip
  
  \center{
    {\includegraphics[width=6.3cm]{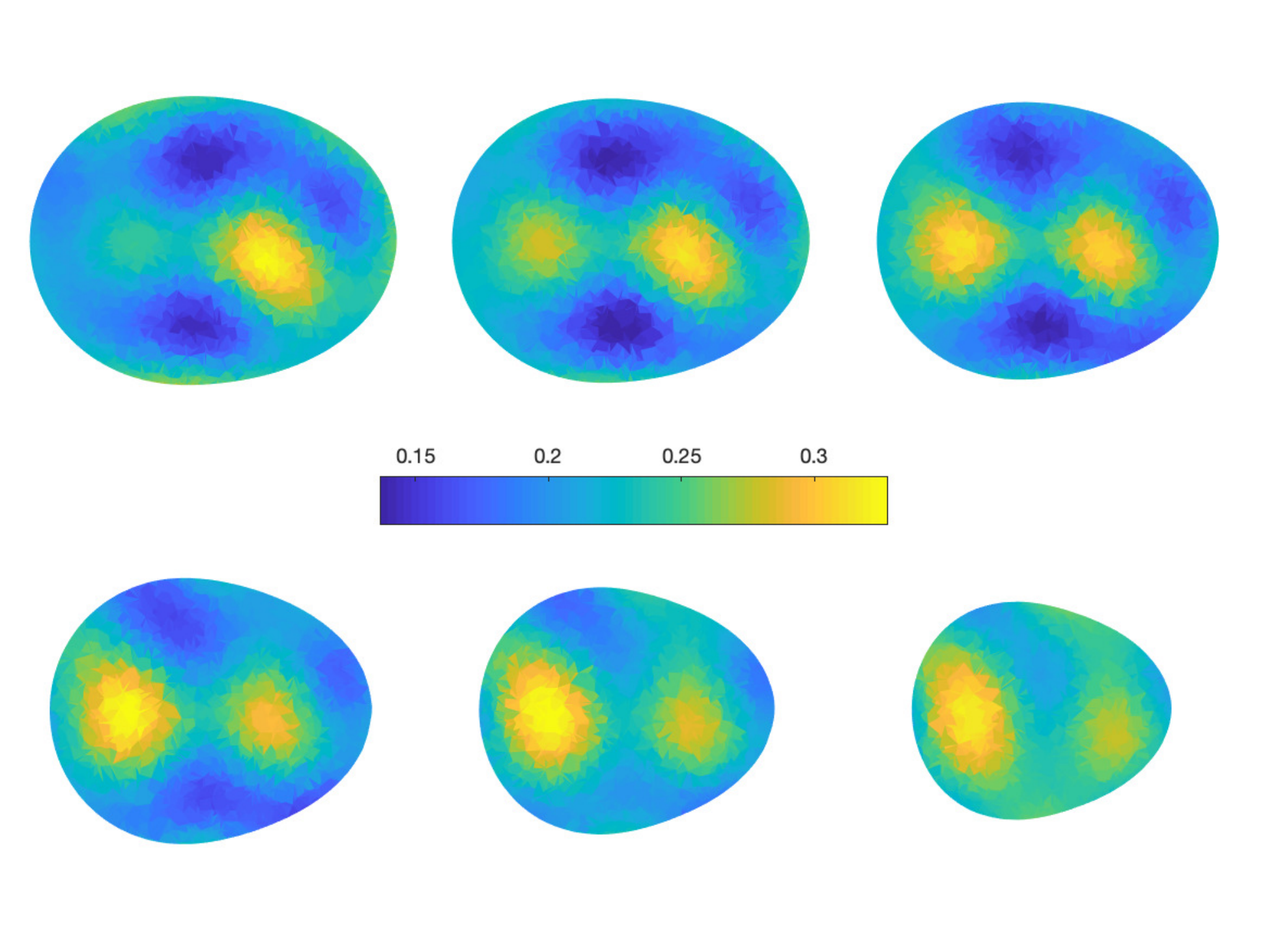}}
    \ \
{\includegraphics[width=6.3cm]{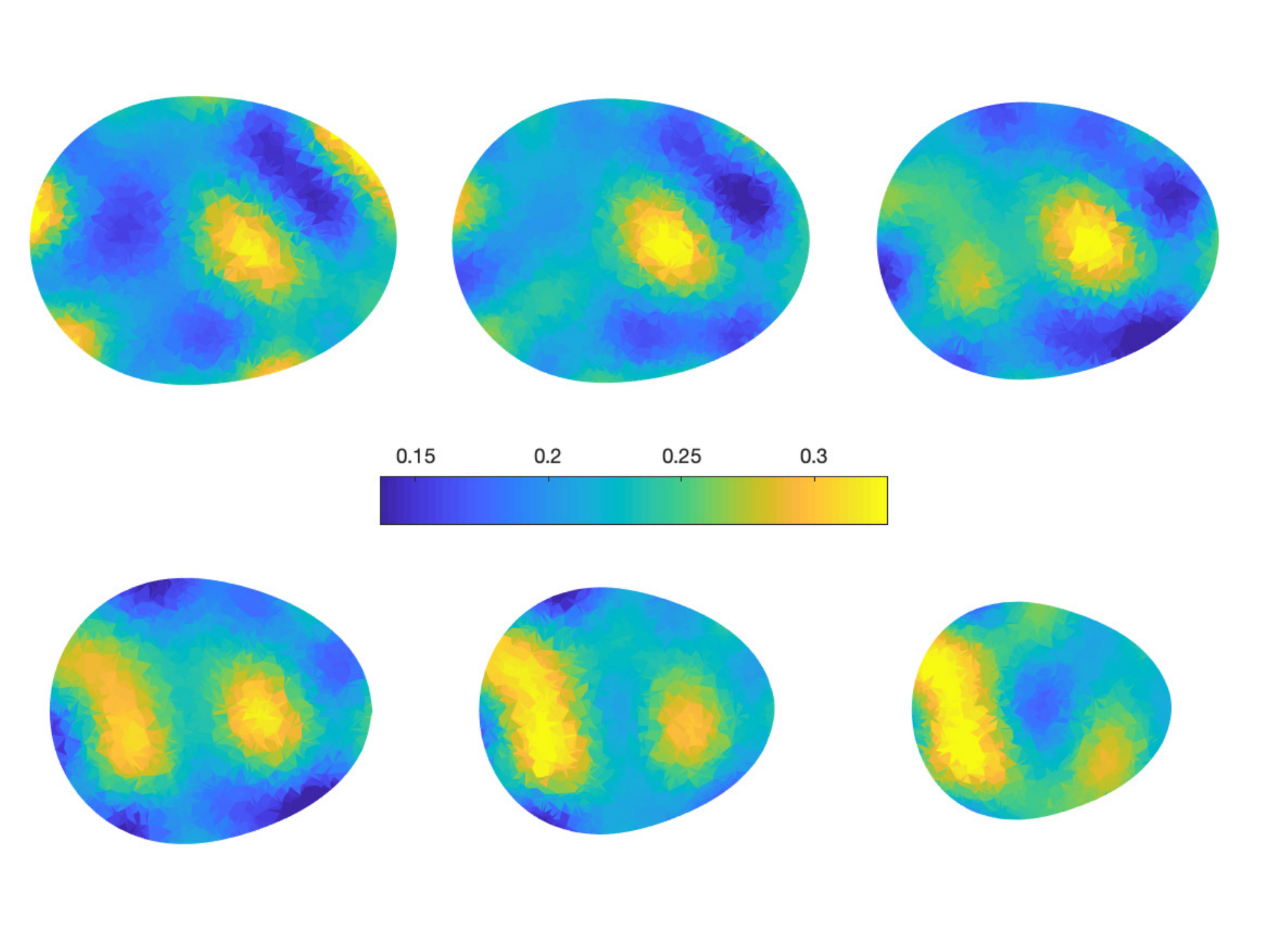}}
  }
\caption{Six horizontal cross-sections of the target conductivity and of three reconstructions for test~2; the slices are at levels $2$, $3$, $4$, $5$, $6$ and $7$\,cm. Top left: target conductivity (the colormap does not cover the whole dynamic range). Top right: reconstruction by the complete Algorithm~\ref{alg:1}. Bottom left: reconstruction when $\alpha = 0$ is fixed. Bottom right: reconstruction when $\alpha = 0$, $\theta = \bar{\theta}$ and $\phi = \bar{\phi}$ are fixed.}
  \label{fig:reco3_1}
\end{figure}

\begin{figure}[t]
  \center{
     {\includegraphics[width=3.5cm]{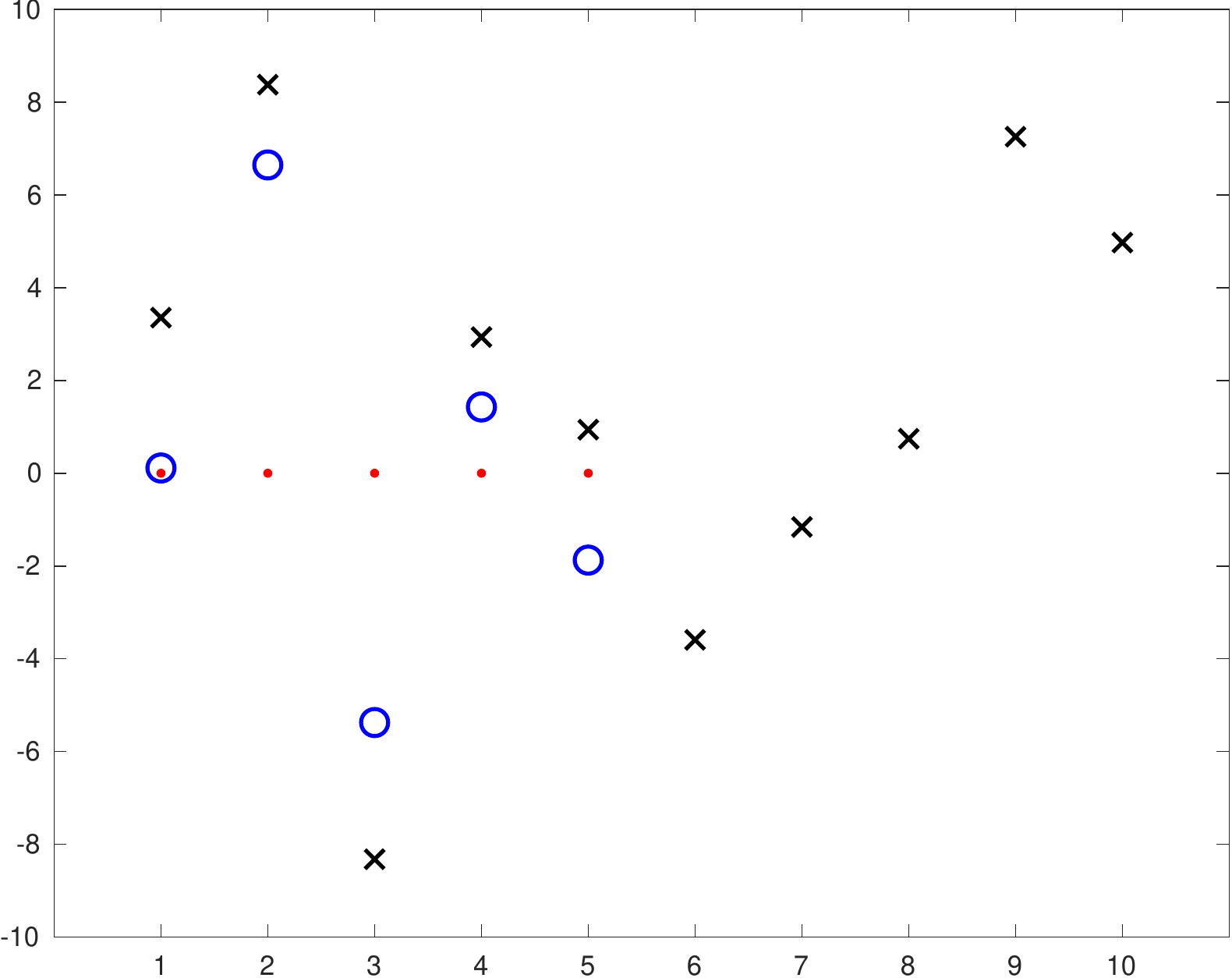}}
     \qquad
         {\includegraphics[width=3.5cm]{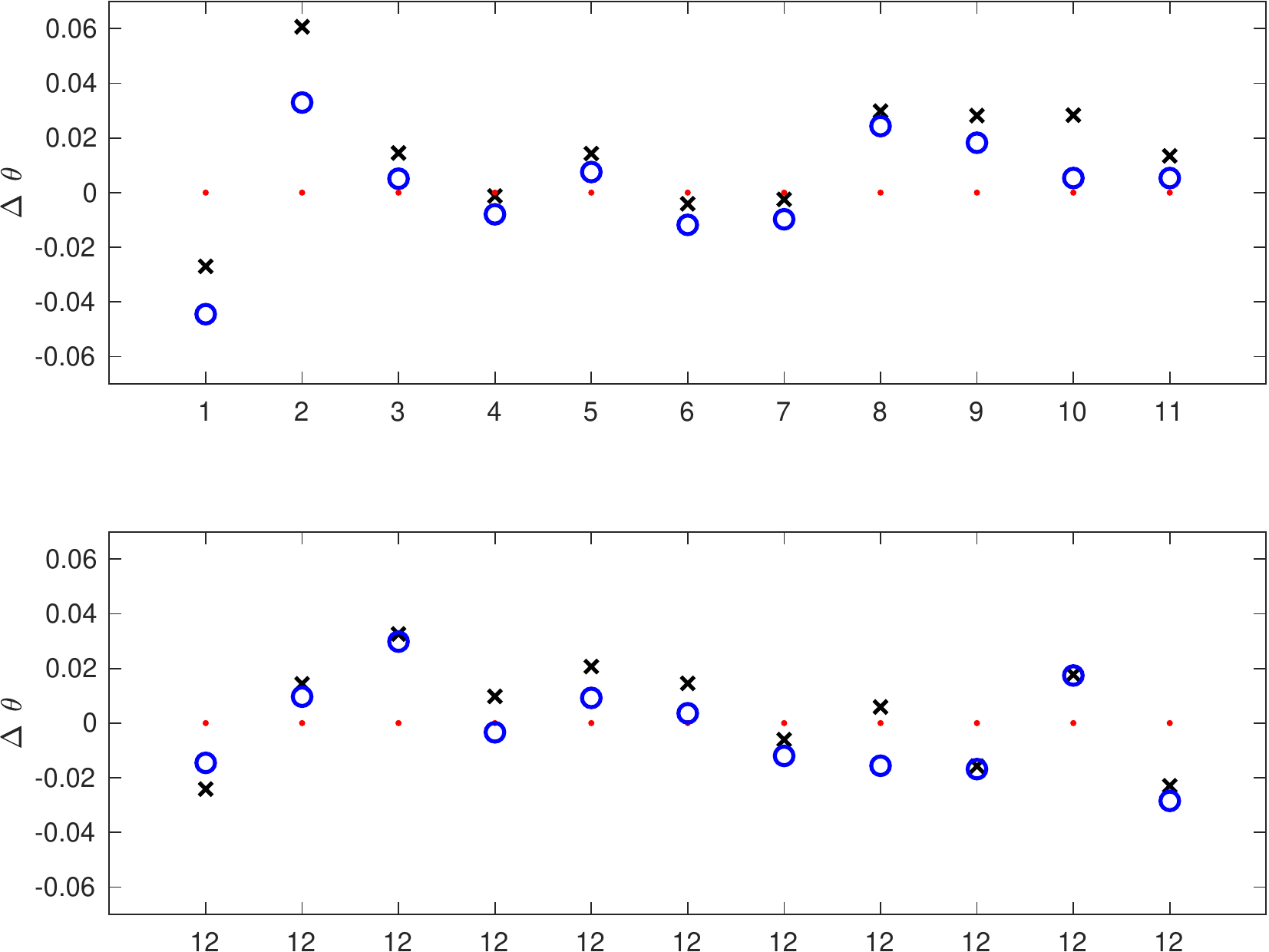}}
         \qquad
         {\includegraphics[width=3.5cm]{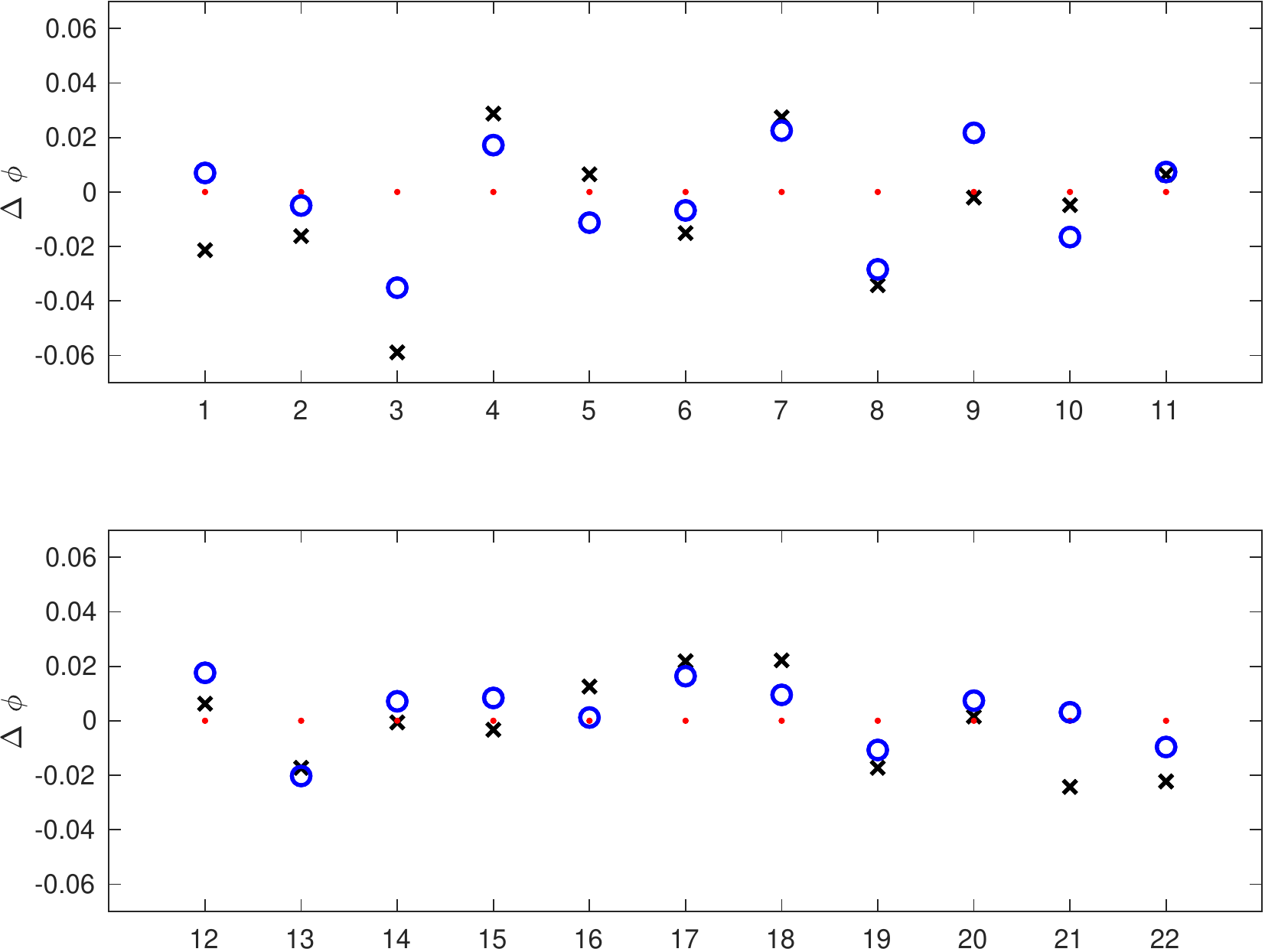}}
}
  \caption{Target parameters and those reconstructed by Algorithm~\ref{alg:1} for test~3. The electrodes are numbered starting from the frontal electrode on the bottom belt and circulating around the head in the positive direction when viewed from above. Left: $\alpha_*\in \R^5$ (circles), $\alpha_{\rm trgt}\in \R^{10}$ (crosses) and the expected values/initial guesses (dots). Center: $\theta_*$ (circles) and $\theta_{\rm trgt}$ (crosses). Right: $\phi_*$ (circles) and $\phi_{\rm trgt}$ (crosses). The electrode angles are depicted relative to their intended positions.}
  \label{fig:reco3_2}
\end{figure}

\section{Conclusion}
\label{sec:conclusion}

We have introduced a computational framework for applying absolute EIT to practical head imaging, with the leading idea being to reconstruct the head shape and the electrode positions as a part of a Newton-type output least squares algorithm. Our numerical experiments demonstrated that significant enough variations in the conductivity of an imaged head can be detected by EIT even if the uncertainties about the head shape and electrode positions are on the level that is to be expected in practice. Moreover, it seems that the natural variations in the head shape can be handled to a certain extent by (only) including the estimation of the electrode positions in a reconstruction algorithm implemented in a fixed average head geometry. That is, mismodeling the head shape can be partially compensated by letting the electrodes move freely during the reconstruction process (cf.~\cite{Hyvonen17c}).

Our investigations did not touch the most severe hindrance for introducing EIT as a head imaging modality, namely the highly resistive skull that makes it difficult to get reliable information about the conductivity in the brain. The most natural way to include the effect of skull in our framework is to employ a head library to form a joint principal component model for the head and skull shapes that are presumably strongly correlated. The joint head and skull geometry could then potentially be parametrized by a low number of shape parameters whose estimation could be included in a Newton-type reconstruction algorithm. Such considerations and testing our algorithm with real-world data are left for future studies.

\bibliographystyle{acm}
\bibliography{comphead-refs}
\end{document}